\documentclass{amsart}
\usepackage{amssymb, latexsym}
\usepackage{color}

\author{Tristan L\'{e}ger}
\address{Courant Institute of Mathematical Sciences, 251 Mercer Street, New York, NY 10012, USA }
\email{tleger@cims.nyu.edu}

\theoremstyle{plain}
\newtheorem{theorem}{Theorem}
\newtheorem{definition}[theorem]{Definition}
\newtheorem{remark}[theorem]{Remark}
\newtheorem{proposition}[theorem]{Proposition}
\newtheorem{lemma}[theorem]{Lemma}
\newtheorem{corollary}[theorem]{Corollary}

\numberwithin{equation}{section} \numberwithin{theorem}{section}

\begin{document}

\title[3D quadratic NLS equation with electromagnetic perturbations]
{3D quadratic NLS equation with electromagnetic perturbations}

\vspace{-0.3in}
\begin{abstract}
In this paper we study the asymptotic behavior of a quadratic Schr\"{o}dinger equation with electromagnetic potentials. We prove that small solutions scatter. The proof builds on earlier work of the author for quadratic NLS with a non magnetic potential. The main novelty is the use of various smoothing estimates for the linear Schr\"{o}dinger flow in place of boundedness of wave operators to deal with the loss of derivative.\\
As a byproduct of the proof we obtain boundedness of the wave operator of the linear electromagnetic Schr\"{o}dinger equation on an $L^2$ weighted space for small potentials, as well as a dispersive estimate for the corresponding flow.
\end{abstract}

\maketitle 

\tableofcontents

\section{Introduction} \label{Introduction}

\subsection{Background} 
We consider a quadratic NLS equation with electromagnetic potential set on $\mathbb{R}^3$:
\begin{equation} \label{NLSmagn}
\begin{cases}
i \partial_t u + \Delta u &= \sum_{i=1}^3 a_i (x) \partial_i u + V(x) u + u^2 \\
u(t=1) &= u_1.
\end{cases}
\end{equation} 
This general form of potential includes the classical Hamiltonian Schr\"{o}dinger equation with electromagnetic potentials:
\begin{equation}\label{NLSelectro}
\begin{cases}
i \partial_t u = H_A u  \\ 
H_A = -(\nabla - i \overrightarrow{A}(x))^2 + V \\
\overrightarrow{A}(x) = (a_1(x), a_2(x), a_3(x))
\end{cases},
\end{equation}
which corresponds to the usual Schr\"{o}dinger equation with an external magnetic field $\overrightarrow{B} = curl(\overrightarrow{A})$ as well as an external electric field $\overrightarrow{E} = - \overrightarrow{\nabla} V.$ Its Hamiltonian is
\begin{align*}
\mathcal{H}(u) = \frac{1}{2} \int_{\mathbb{R}^3} \big \vert ( \nabla - i \overrightarrow{A}(x) ) u \big \vert^2 + V(x) \vert u \vert^2 ~dx .
\end{align*}
In our setting the potentials will be assumed to be small, and our goal is to study the asymptotic behavior of solutions to the equation \eqref{NLSmagn}. \\
\\
Besides the physical interest of the problem, we are motivated by the fact that \eqref{NLSmagn} is a toy model for the study of linearizations of dispersive equations around non-zero remarkable solutions (traveling waves or solitons for example). This is why we elected to work with the equation \eqref{NLSmagn} and not a nonlinear version of \eqref{NLSelectro}. \\
The present article is a continuation of the author's earlier work \cite{L}, where the above equation was considered for $a_i =0.$ The main reason for adding the derivative term is that most models of physical relevance are quasilinear, and their linearizations will generically contain such derivative terms.
Note that a more complete model would consist in treating a general quadratic nonlinearity $Q(u,\overline{u}):$ the present article leaves out the cases of nonlinearities $\overline{u}^2$ and $\vert u \vert^2$. Our method would apply for the nonlinearity $\overline{u}^2$ (in fact this is a strictly easier problem). However the nonlinearity $\vert u \vert^2$ is currently out of reach. In fact even in the flat case (no potentials are present), the problem has not been completely answered. We refer to the article of X. Wang \cite{Wang} for more on this subject.
\\
\\
Regarding existing results on the behavior as $t\to \infty$ of solutions to equations of this type (nonlinearity with potential) we mention some works on the Strauss conjecture on non flat backgrounds: the equations considered have potential parts that lose derivatives, but the nonlinearity (of power type) typically has a larger exponent than what we consider in the present work. We can for example cite the work of K. Hidano, J. Metcalfe, H. Smith, C. Sogge, Y. Zhou (\cite{HMSSZ}) where the conjecture is proved outside a nontrapping obstacle. H. Lindblad, J. Metcalfe, C. Sogge, M. Tohaneanu and C. Wang (\cite{LMSTW}) proved the conjecture for Schwarzschild and Kerr backgrounds. The proofs of these results typically rely on weighted Strichartz estimates to establish global existence of small solutions.
\\
\\
In this paper we prove that small, spatially localized solutions to \eqref{NLSmagn} exist globally and scatter. We take the opposite approach to the works cited above. Indeed we deal with a stronger nonlinearity which forces us to take into account its precise structure. We rely on the space-time resonance theory of P. Germain, N. Masmoudi and J. Shatah (see for example \cite{GMS}). It was developed to study the asymptotic behavior of very general nonlinear dispersive equations for power-type nonlinearities with small exponent (below the so-called Strauss exponent). In the present study the nonlinearity is quadratic (which is exactly the Strauss exponent in three dimension) hence the need to resort to this method. It has been applied to many models by these three authors and others. Without trying to be exhaustive, we can mention for example the water waves problem treated in various settings (\cite{GMSww}, \cite{GMSww2}, \cite{IPu1}), the Euler-Maxwell equation in \cite{GIP}, or the Euler-Poisson equation in \cite{IP}. Similar techniques have also been developed by other authors: S. Gustafson, K. Nakanishi and T.-P. Tsai studied the Gross-Pitaevskii equation in \cite{GNT} using a method more closely related to J. Shatah's original method of normal forms \cite{Sh}. M. Ifrim and D. Tataru used the method of testing against wave packets in \cite{IT1} and \cite{IT2} to study similar models, namely NLS and water waves in various settings. 
\\
\\ 
The difficulty related to the strong nonlinearity was already present in \cite{L}. However, in the present context, the derivative forces us to modify the approach developed in that paper since we must incorporate smoothing estimates into the argument. The inequalities we use in the present work were first introduced by C. Kenig, G. Ponce and L. Vega in \cite{KPV} to prove local well-posedness of a large class of nonlinear Schr\"{o}dinger equations with derivative nonlinearities. They allow to recover one derivative, which will be enough to deal with \eqref{NLSmagn}. Another estimate of smoothing-Strichartz type proved by A. Ionescu and C. Kenig in \cite{IK} will play an important role in the paper. Let us also mention that, in the case of magnetic Schr\"{o}dinger equations with small potentials, a large class of related Strichartz and smoothing estimates have been proved by V. Georgiev, A. Stefanov and M. Tarulli in \cite{GST}. For large potentials of almost critical decay, Strichartz and smoothing estimates have been obtained by B. Erdo\u gan, M. Goldberg and W. Schlag in \cite{EGS}, \cite{EGS2}, and a similar result was obtained recently by d'Ancona in \cite{DA}. A decay estimate for that same linear equation was proved by P. d'Ancona and L. Fanelli in \cite{DF} for small but rough potentials. A corollary of the main result of the present paper is a similar decay estimate under stronger assumptions on the potentials, but for a more general equation. Regarding dispersive estimates for the linear flow, we also mention the work of L. Fanelli, V. Felli, M. Fontelos and A. Primo in \cite{FFFP} where decay estimates for the eletromagnetic linear Schr\"{o}dinger equation are obtained in the case of particular potentials of critical decay. The proof is based on a representation formula for the solution of the equation. Finally, for the linear electromagnetic Schr\"{o}dinger equation, a corollary of our main theorem is that its wave operator is bounded on a space that can heuristically be thought of as $\langle x \rangle L^2 \cap H^{10}, $ see the precise statement in Corollary \ref{bddwaveop} below.     
\\
\\
To deal with the full equation, we must therefore use both smoothing and space-time resonance arguments simultaneously. The general idea is to expand the solution as a power series using Duhamel's formula repeatedly. This type of method is routinely used in the study of linear Schr\"{o}dinger equations through Born series, for example. As we mentioned, this general plan was already implemented by the author for a less general equation in \cite{L}, where there is no loss of derivatives. The additional difficulty coming from high frequencies forces us to modify the approach, and follow a different strategy in the multilinear part of the proof (that is in the estimates on the iterated potential terms). Note that along the way we obtain a different proof of the result in \cite{L}. In particular we essentially rely on Strichartz and smoothing estimates for the free linear Schr\"{o}dinger equation, instead of the more stringent boundedness of wave operators. This would allow us to relax the assumptions on the potential in \cite{L}.

\subsection{Main result}
\subsubsection{Notations} \label{prelim}
We start this section with some notations that will be used in the paper. \\
First we recall the formula for the Fourier transform:
\begin{align*}
\widehat{f}(\xi) = \mathcal{F}f (\xi)= \int_{\mathbb{R}^3} e^{-i x \cdot \xi} f(x) dx
\end{align*}
hence the following definition for the inverse Fourier transform:
\begin{align*}
\check{f}(x) = [\mathcal{F}^{-1} f ](x) = \frac{1}{(2 \pi)^3} \int_{\mathbb{R}^3} e^{i x \cdot \xi} f(\xi) d\xi.
\end{align*}
Now we define Littlewood-Paley projections. Let $\phi$ be a smooth radial function supported on the annulus $\mathcal{C} = \lbrace \xi \in \mathbb{R}^3 ; \frac{1}{1.04} \leqslant \vert \xi \vert \leqslant 1.04 \times 1.1 \rbrace $ such that
\begin{align*}
\forall \xi \in \mathbb{R}^3 \setminus \lbrace 0 \rbrace, ~~~ \sum_{j \in \mathbb{Z}} \phi\big( 1.1^{-j} \xi \big) =1.
\end{align*}
Notice that if $j-j'>1$ then $1.1^j \mathcal{C} \cap 1.1^{j'} \mathcal{C} = \emptyset.$ \\
We will denote $P_k (\xi) := \phi(1.1^{-k}\xi)$ the Littlewood-Paley projection at frequency $1.1^k.$ \\
Similarly, $P_{\leqslant k} (\xi)$ will denote the Littlewood-Paley projection at frequencies less than $1.1^k.$ \\
It was explained in \cite{L} why we localize at frequency $1.1^k$ and not $2^k.$ (See Lemma 7.8 in that paper).\\
We will also sometimes use the notation $\widehat{f_k}(\xi) =P_k (\xi) \widehat{f}(\xi).$ 
\\
\\
Now we come to the main norms used in the paper: \\
We introduce the following notation for mixed norms of Lebesgue type:
\begin{align*}
\Vert f \Vert_{L^p _{x_j} L^q_{\widetilde{x_j}}} = \bigg \Vert \Vert f(\cdot,...,\cdot,x_j, \cdot, ..) \Vert_{L^q_{x_1,...,x_{j-1},x_{j+1},...}}  \bigg \Vert_{L^p_{x_j}}.
\end{align*}
To control the profile of the solution we will use the following norm: 
\begin{align*}
\Vert f \Vert_X = \sup_{k \in \mathbb{Z}} \Vert \nabla_{\xi} \widehat{f_k} \Vert_{L^2_x}
\end{align*}
Roughly speaking, it captures the fact that the solution has to be spatially localized around the origin.
\\
For the potentials, we introduce the following controlling norm:
\begin{align*}
\Vert V \Vert_{Y} &= \Vert V \Vert_{L^1 _x} + \Vert V \Vert_{L^{\infty}_x} + \sum_{j=1}^3 \bigg \Vert \Vert \vert V \vert^{1/2} \Vert_{L^{\infty}_{\widetilde{x_j}}} \bigg \Vert_{L^2 _{x_j}}.
\end{align*}

\subsubsection{Main Theorem} \label{results}
With these notations, we are ready to state our main theorem. \\
We prove that small solutions to \eqref{NLSmagn} with small potentials exist globally and that they scatter. More precisely the main result of the paper is
\begin{theorem} \label{maintheorem}
There exists $\varepsilon>0$ such that if $\varepsilon_0, \delta < \varepsilon$ and if $u_{1},$ $V$ and $a_i$ satisfy
\begin{align*}
\Vert V \Vert_{Y} + \Vert \langle x \rangle V \Vert_{Y} + \Vert (1-\Delta)^{5} V \Vert_{Y} & \leqslant \delta, \\
\Vert a_i \Vert_{Y} + \Vert \langle x \rangle a_i \Vert_{Y} + \Vert (1-\Delta)^{5} a_i \Vert_{Y} & \leqslant  \delta, \\
\Vert e^{-i \Delta} u_{1} \Vert_{H^{10}_x} + \Vert e^{-i \Delta} u_{1} \Vert_{X} & \leqslant  \varepsilon_0,
\end{align*}
then \eqref{NLSmagn} has a unique global solution. Moreover it satisfies the estimate
\begin{align} \label{mainestimate}
\sup_{t \in [1;\infty)} \Vert u(t) \Vert_{H^{10}_x} + \Vert e^{-it \Delta} u(t) \Vert_{X} + \sup_{k \in \mathbb{Z}} t \Vert u_k (t) \Vert_{L^6_x} \lesssim \varepsilon_0.
\end{align}
Moreover it scatters in $H^{10}_x$: there exists $u_{\infty} \in H^{10}_x$ and a bounded operator $W:H^{10}_x \rightarrow H^{10}_x$ such that
\begin{align*}
\Vert e^{-it \Delta} W u(t) - u_{\infty} \Vert_{H^{10}_x} \rightarrow 0
\end{align*}
as $t \to \infty.$ 
\end{theorem}
\begin{remark}
We have not strived for the optimal assumptions on the potentials or the initial data. It is likely that the same method of proof, at least in the case where $a_i=0$, allows for potentials with almost critical decay (that is $V \in L^{3/2-}_x$ and $x V \in L^{3-}_x $ and similar assumptions on its derivative). Similarly the $H^{10}_x$ regularity can most likely be relaxed. 
\end{remark}
\begin{remark}
Unlike in the earlier work \cite{L}, we cannot treat the case of time dependent potentials. This is mainly due to the identity \eqref{identite} and its use in the subsequent proofs of our multilinear lemmas.
\end{remark}
\begin{remark}
A similar scattering statement could be formulated in the space $X$ although it is more technical. For this reason we have elected to work in $H^{10}_x$ (see the proof in the appendix, Section \ref{scattering}).
\end{remark}
As we mentioned above, the result proved in Theorem \ref{maintheorem} has a direct consequence for the linear flow of the electromagnetic equation. We have the following corollary, which provides a decay estimate for the flow as well as a uniform in time boundedness of the profile of the solution on the space $X$.
\begin{corollary} \label{corlin}
There exists $\varepsilon>0$ such that if $ \delta < \varepsilon$ and if $u_{1},$ $V$ and $a_i$ satisfy
\begin{align*}
\Vert V \Vert_{Y} + \Vert \langle x \rangle V \Vert_{Y} + \Vert (1-\Delta)^{5} V \Vert_{Y} & \leqslant \delta, \\
\Vert a_i \Vert_{Y} + \Vert \langle x \rangle a_i \Vert_{Y} + \Vert (1-\Delta)^{5} a_i \Vert_{Y} & \leqslant  \delta, \\
\Vert e^{-i \Delta} u_{1} \Vert_{H^{10}_x} + \Vert e^{-i \Delta} u_{1} \Vert_{X} & < + \infty,
\end{align*}
then the Cauchy problem 
\begin{equation} \label{Lmagn}
\begin{cases}
i \partial_t u + \Delta u &= \sum_{i=1}^3 a_i (x) \partial_i u + V(x) u  \\
u(t=1) &= u_1
\end{cases}
\end{equation} 
has a unique global solution $u(t)$ that obeys the estimate 
\begin{align} \label{mainestimatelin}
\sup_{t \in [1;\infty)} \Vert u(t) \Vert_{H^{10}_x} + \Vert e^{-it \Delta} u(t) \Vert_{X} + \sup_{k \in \mathbb{Z}} t \Vert u_k (t) \Vert_{L^6_x} \lesssim \Vert e^{-i \Delta} u_{1} \Vert_{H^{10}_x} + \Vert e^{-i \Delta} u_{1} \Vert_{X}.
\end{align}
\end{corollary}
\begin{remark}
As we noted above for the nonlinear equation, the assumptions made on the regularity and decay of the potentials are far from optimal. We believe that minor changes in the proof would lead to much weaker conditions in the statement above. 
\end{remark}
Now we write a corollary of the above estimate \eqref{mainestimatelin} in terms of the wave operator for the linear electromagnetic Schr\"{o}dinger operator corresponding to \eqref{NLSelectro}. \\
Indeed the $X$ part of this estimate can be written in that setting
\begin{align*}
\sup_{t \in [1; \infty)} \Vert e^{-it \Delta} u(t) \Vert_{X} = \sup_{t \in [1; \infty)} \Vert e^{-it \Delta} e^{i(t-1) H_A} u_1 \Vert_{X}
\end{align*}
where we recall that $H_A = -(\nabla - i \overrightarrow{A}(x))^2 + V$. We directly deduce the following 
\begin{corollary} \label{bddwaveop}
Let $W$ denote the wave operator of $H_A.$ There exists $\varepsilon>0$ such that for every $\delta<\varepsilon,$ if the potentials $A,$ $V$ and the initial data $u_1$ satisfy
\begin{align*}
\Vert V \Vert_{Y} + \Vert \langle x \rangle V \Vert_{Y} + \Vert (1-\Delta)^{5} V \Vert_{Y} & \leqslant \delta, \\
\Vert A_i \Vert_{Y} + \Vert \langle x \rangle A_i \Vert_{Y} + \Vert (1-\Delta)^{5} A_i \Vert_{Y} & \leqslant  \delta, \\
\Vert (A_i)^2 \Vert_{Y} + \Vert \langle x \rangle (A_i)^2 \Vert_{Y} + \Vert (1-\Delta)^{5} (A_i)^2 \Vert_{Y} & \leqslant  \delta, \\
\Vert e^{-i \Delta} u_{1} \Vert_{H^{10}_x} + \Vert e^{-i \Delta} u_{1} \Vert_{X} & < + \infty,
\end{align*}
then we have 
\begin{align*}
\Vert W u_1 \Vert_{H^{10}_x} + \Vert W u_1 \Vert_{X} \lesssim \Vert e^{-i \Delta} u_{1} \Vert_{H^{10}_x} + \Vert e^{-i \Delta} u_{1} \Vert_{X}.
\end{align*}
\end{corollary}

\subsection{Set-up and general idea of the proof} \label{strategy}
We work with the profile of the solution $f(t) = e^{-it \Delta} u(t).$ \\
\subsubsection{Local well-posedness} The local wellposedness of \eqref{NLSmagn} in $H^{10}_x$ follows from the estimate proved (in a much more general setting) by S. Doi in \cite{D}:
\begin{align*}
\Vert f \Vert_{L^{\infty}_t([1;T])H^{10}_x} \lesssim \Vert f_1 \Vert_{H^{10}_x} + (T-1) \Vert f \Vert_{L^{\infty}_t ([1;T])H^{10}_x}^2
\end{align*}
\subsubsection{The bootstrap argument}
The proof of global existence and decay relies on a bootstrap argument: we assume that  for some $T>1$ and for $\varepsilon_1 = A \varepsilon_0$ ($A$ denotes some large number) the following bounds hold
\begin{align*}
\sup_{t \in [1;T]} \Vert f(t) \Vert_X & \leqslant \varepsilon_1, \\
\sup_{t \in [1;T]} \Vert f(t) \Vert_{H^{10}_x} & \leqslant \varepsilon_1,
\end{align*}
and then we prove that these assumptions actually imply the stronger conclusions
\begin{align}
\label{goal1} \sup_{t \in [1;T]} \Vert f(t) \Vert_X & \leqslant \frac{ \varepsilon_1}{2}, \\
\label{goal2} \sup_{t \in [1;T]} \Vert f(t) \Vert_{H^{10}_x} & \leqslant \frac{\varepsilon_1}{2}.
\end{align}
The main difficulty is to estimate the $X$ norm. To do so we expand $\partial_{\xi} \widehat{f}$ as a series by essentially applying the Duhamel formula recursively. The difference with \cite{L} is that, for high output frequencies, iterating the derivative part will prevent the series from converging if we use the same estimates as in that paper. To recover the derivative loss we use smoothing estimates which allow us to gain one derivative back at each step of the iteration. It is at this stage (the multilinear analysis) that the argument from \cite{L} must be modified. Instead of relying on the method of M. Becenanu and W. Schlag (\cite{BSmain}) the estimations are done more in the spirit of K. Yajima's paper \cite{Y}. \\
\subsubsection{The series expansion}
Our discussion here and in the next subsection will be carried out for a simpler question than that tackled in this paper. However it retains the main novel difficulty compared to the earlier paper \cite{L}, namely the loss of derivative in the potential part. More precisely, we see how to estimate the $L^2_x$ norm of $\widehat{f}(t,\xi).$ First, we explain the way we generate the series representation of $\widehat{f}(t,\xi)$. We consider the Duhamel formula for $\widehat{f}:$ The potential part has the form
\begin{align*}
\int_0 ^t \int_{\mathbb{R}^3}  e^{is (\vert \xi \vert^2 - \vert \eta_1 \vert^2)} \widehat{W_1}(\xi-\eta_1) \alpha_1(\eta_1) \widehat{f}(s,\eta_1) d\eta_1 ds,
\end{align*}
where $W_1$ denotes either $a_i$ or $V$ and $\alpha_1(\eta_1) = 1$ if $W_1 = V$ and $\eta_{1,i}$ if $W_1 = a_i.$ \\
The general idea is to integrate by parts in time in that expression iteratively to write $\widehat{f}$ as a series made up of the boundary terms remaining at each step. Roughly speaking we will obtain two types of terms, corresponding either to the potential part or the bilinear part of the nonlinearity: 
\begin{align} \label{potentialmod}
 \int \prod_{\gamma=1}^{n-1} \frac{\widehat{W_{\gamma}}(\eta_{\gamma}) \alpha_{\gamma}(\eta_{\gamma})}{\vert \xi \vert^2 - \vert \eta_{\gamma} \vert^2} \int_{\eta_{n}} \widehat{W_n}(\eta_{n-1}-\eta_n) \alpha_{n}(\eta_n) e^{it(\vert \xi \vert^2-\vert \eta_n \vert^2)} \widehat{f}(t,\eta_n) d\eta_n d\eta
\end{align}
and 
\begin{align} \label{bilinmod}
\int \prod_{\gamma=1}^{n-1} \frac{\widehat{W_{\gamma}}(\eta_{\gamma}) \alpha_{\gamma}(\eta_{\gamma})}{\vert \xi \vert^2 - \vert \eta_{\gamma} \vert^2} \int_{\eta_{n}} e^{it(\vert \xi \vert^2-\vert \eta_{n-1} - \eta_n \vert^2-\vert \eta_n \vert^2)} \widehat{f}(t,\eta_{n-1}-\eta_n) \widehat{f}(t,\eta_n) d\eta_n d\eta.
\end{align}
\subsubsection{Convergence of the series in $L^2$} 
Now we prove that the series obtained in the previous section converges in $L^{\infty}_t L^{2}_x.$ We prove estimates like 
\begin{align*}
\Vert \eqref{potentialmod}, \eqref{bilinmod} \Vert_{L^{\infty}_t L^{2}_x} \lesssim C^n \delta^n \varepsilon_1
\end{align*}
for some universal constant $C$ and where the implicit constant does not depend on $n.$ Heuristically, each $V$ factor contributes a $\delta$ in the estimate. \\
First, we write that in physical space we have, roughly speaking:
\begin{align*}
\eqref{potentialmod} &=\int_{0 \leqslant \tau_1} e^{i \tau_1 \Delta} W_1 \widetilde{D_1} \int_{0 \leqslant \tau_2 \leqslant \tau_1} e^{i(\tau_2-\tau_1) \Delta} W_2 \widetilde{D_2} ... \\
& \times \int_{0 \leqslant \tau_{n} \leqslant \tau_{n-1}} e^{i(\tau_n - \tau_{n-1})\Delta} W_{n} \widetilde{D_n} e^{-i \tau_n \Delta} f d\tau_1 ... d\tau_n ,
\end{align*}
where we denoted $\widetilde{D_i}$ the operator equal to $1$ if $W_i=V$ and $\partial_{x_j}$ if $W_i =a_j.$ \\
Then, using Strichartz estimates, we can write
\begin{align*}
\Vert \eqref{potentialmod} \Vert_{L^2_x} &\lesssim \Bigg \Vert W_1 \widetilde{D_1} \int_{0 \leqslant \tau_2 \leqslant \tau_1} e^{i(\tau_2-\tau_1) \Delta} W_2 \widetilde{D_2} ... \\
& \times \int_{0 \leqslant \tau_{n} \leqslant \tau_{n-1}} e^{i(\tau_n - \tau_{n-1})\Delta} W_{n} \widetilde{D_n} e^{-i \tau_n \Delta} f d\tau_2 ... d\tau_n  \Bigg \Vert_{L^2_{\tau_1} L^{6/5}_x}.
\end{align*}
Next if $W_1 = V$, then we can use Strichartz estimates again and write
\begin{align*}
\Vert \eqref{potentialmod} \Vert_{L^2_x} &\lesssim C \delta \Bigg \Vert \int_{0 \leqslant \tau_2 \leqslant \tau_1} e^{i(\tau_2-\tau_1) \Delta} W_2 \widetilde{D_2} ... \\
& \times \int_{0 \leqslant \tau_{n} \leqslant \tau_{n-1}} e^{i(\tau_n - \tau_{n-1})\Delta} W_{n} \widetilde{D_n} e^{-i \tau_n \Delta} f d\tau_2 ... d\tau_n  \Bigg \Vert_{L^2_{\tau_1} L^{6}_x} \\
&\lesssim C \delta \Bigg \Vert W_2 \widetilde{D_2}... \int_{0 \leqslant \tau_{n} \leqslant \tau_{n-1}} e^{i(\tau_n - \tau_{n-1})\Delta} W_{n} \widetilde{D_n} e^{-i \tau_n \Delta} f d\tau_3 ... d\tau_n  \Bigg \Vert_{L^2_{\tau_2} L^{6/5}_x}.
\end{align*}
If $W_1=a_i$ then in this case we use smoothing estimates to write that
\begin{align*}
\Vert \eqref{potentialmod} \Vert_{L^2_x} &\lesssim C \delta \Bigg \Vert \partial_{x_i} \int_{0 \leqslant \tau_2 \leqslant \tau_1} e^{i(\tau_2-\tau_1) \Delta} W_2 \widetilde{D_2} ... \\
& \times \int_{0 \leqslant \tau_{n} \leqslant \tau_{n-1}} e^{i(\tau_n - \tau_{n-1})\Delta} W_{n} \widetilde{D_n} e^{-i \tau_n \Delta} f d\tau_2 ... d\tau_n  \Bigg \Vert_{L^{\infty}_{x_i} L^2_{\tau_1,\widetilde{x_i}}} \\
&\lesssim C \delta \Bigg \Vert W_2 \widetilde{D_2}... \int_{0 \leqslant \tau_{n} \leqslant \tau_{n-1}} e^{i(\tau_n - \tau_{n-1})\Delta} W_{n} \widetilde{D_n} e^{-i \tau_n \Delta} f d\tau_3 ... d\tau_n  \Bigg \Vert_{L^{1}_{x_i} L^2_{\tau_2,\widetilde{x_i}}}.
\end{align*}
Then we continue this process recursively to obtain the desired bound: if we encounter a potential without derivative, we use Strichartz estimates, and if the potential carries a derivative, then we use smoothing estimates.\\
Say for example that in the expression above, $\widetilde{D_2}=1.$ Then we write that
\begin{align*}
\Vert \eqref{potentialmod} \Vert_{L^2_x} & \lesssim C^2 \delta^2 \Bigg \Vert \int_{0 \leqslant \tau_{n} \leqslant \tau_{n-1}} e^{i(\tau_n - \tau_{n-1})\Delta} W_{n} \widetilde{D_n} e^{-i \tau_n \Delta} f d\tau_3 ... d\tau_n  \Bigg \Vert_{L^{2}_{\tau_2} L^6_x}.
\end{align*}
Otherwise, if $\widetilde{D_2}=\partial_{x_k}$ then we obtain
\begin{align*}
\Vert \eqref{potentialmod} \Vert_{L^2_x} & \lesssim C^2 \delta^2 \Bigg \Vert \int_{0 \leqslant \tau_{n} \leqslant \tau_{n-1}} e^{i(\tau_n - \tau_{n-1})\Delta} W_{n} \widetilde{D_n} e^{-i \tau_n \Delta} f d\tau_3 ... d\tau_n  \Bigg \Vert_{L^{\infty}_{x_k} L^2_{\tau_2,\widetilde{x_k}}}.
\end{align*}
To close the estimates when $W_n = V,$ we write that, using Strichartz inequalities, we have: 
\begin{align*}
\Vert \eqref{potentialmod} \Vert_{L^2_x} & \lesssim C^{n-1} \delta^{n-1} \Vert V e^{-i\tau_n \Delta} f \Vert_{Z} \\
                                         & \lesssim C^n \delta^n \Vert e^{-i\tau_n \Delta} f \Vert_{L^{2}_{\tau_n} L^{6}_{x}}  \\
                                         & \lesssim C^n \delta^n \Vert f \Vert_{L^2_x},
\end{align*}
where $Z$ denotes either $L^{2}_{\tau_n} L^{6/5}_{x}$ or $L^1_{x_i} L^2_{\tau_n,\widetilde{x_i}} $ depending on whether $W_{n-1}=V$ or $a_i.$ The case where $W_n = a_k$ is treated similarly, except that we use smoothing instead of Strichartz estimates.\\
\\
In the case of the nonlinear term in $f$ \eqref{bilinmod}, the $f$ that was present in \eqref{potentialmod} is replaced by the quadratic nonlinearity. As a result, the same strategy essentially reduces to estimating that quadratic term in $L^2$. This takes us back to a situation that is handled by the classical theory of space-time resonances: such a term was already present in the work of P. Germain, N. Masmoudi and J. Shatah (\cite{GMS}).
\\
Of course, in reality, the situation is more complicated: here we were imprecise as to which smoothing effects we were using. Moreover we have mostly ignored the difficulty to combine the above smoothing arguments with the classical space-time resonance theory. In the actual proof we must resort to several smoothing estimates, see Section \ref{recallsmoothing} for the complete list.

\subsubsection{Bounding the $X-$ norm of the profile}
In the actual proof we must keep the $X-$norm of the profile under control. The situation is more delicate than for the $L^2$ norm, but the general idea is similar and was implemented in our previous paper \cite{L}. We recall it in this section for the convenience of the reader. \\
To generate the series representation we cannot merely integrate by parts in time since when the $\xi-$derivative hits the phase, an extra $t$ factor appears. Roughly speaking we are dealing with terms like 
\begin{align}\label{termpotheur}
\int_0 ^t \int_{\mathbb{R}^3} s  e^{is (\vert \xi \vert^2 - \vert \eta_1 \vert^2)} \widehat{W_1}(\xi-\eta_1) \alpha_1(\eta_1) \widehat{f}(s,\eta_1) d\eta_1 ds
\end{align}
for the potential part. \\
The idea is to integrate by parts in frequency to gain additional decay, and then perform the integration by parts in time. For the term above this yields an expression like
\begin{align*}
\int_0 ^t \int_{\mathbb{R}^3} \frac{\eta_{1}}{\vert \eta_1 \vert^2} e^{is (\vert \xi \vert^2 - \vert \eta_1 \vert^2)} \widehat{W_1}(\xi-\eta_1) \alpha_1(\eta_1) \partial_{\eta_1} \widehat{f}(s,\eta_1) d\eta_1 ds + \lbrace \textrm{easier terms} \rbrace.
\end{align*}
Then we can integrate by parts in time to obtain terms like:
\begin{align*}
&\int_{\mathbb{R}^3} \frac{\eta_{1}}{\vert \eta_1 \vert^2 (\vert \xi \vert^2-\vert \eta_1 \vert^2)} e^{it (\vert \xi \vert^2 - \vert \eta_1 \vert^2)} \widehat{W_1}(\xi-\eta_1) \alpha_1(\eta_1) \partial_{\eta_1} \widehat{f}(t,\eta_1) d\eta_1 \\
&+ \int_0 ^t \int_{\mathbb{R}^3} \frac{\eta_{1}}{\vert \eta_1 \vert^2 (\vert \xi \vert^2-\vert\eta_1 \vert^2)}  e^{is (\vert \xi \vert^2 - \vert \eta_1 \vert^2)} \widehat{W_1}(\xi-\eta_1) \alpha_1(\eta_1) \partial_s \partial_{\eta_1} \widehat{f}(s,\eta_1) d\eta_1 ds \\
&+ \lbrace \textrm{easier terms} \rbrace.
\end{align*}
The boundary term will be the first term of the series representation. Then we iterate this process on the integral part. This is indeed possible since $\partial_{\eta_1} \widehat{f}$ and $\widehat{f}$ satisfy the same type of equation $\partial_t \widehat{f} \sim e^{-it\Delta} \big( Vu + u^2 \big)$ up to lower order terms and with different potentials (essentially $V$ and $xV$ respectively). 
\\
After generating the series, the next step is to prove a geometric sequence type bound for the $L^2$ norm of its terms. If we are away from space resonances, namely if the multiplier $\eta_1/\vert \eta_1 \vert^2$ is not singular, then we are essentially in the same case as in the previous section on the $L^2$ estimate of the solution. However if the added multiplier is singular, then we cannot conclude as above. The scheme we have described here is only useful away from space resonances. \\
We can modify this approach and choose to integrate by parts in time first. We obtain boundary terms with an additional $t$ factor compared to the previous section:
\begin{align*}
&\int_{\mathbb{R}^3} \frac{1}{\vert \xi \vert^2-\vert \eta_1 \vert^2} e^{it (\vert \xi \vert^2 - \vert \eta_1 \vert^2)} \widehat{W_1}(\xi-\eta_1) \alpha_1(\eta_1) t \widehat{f}(t,\eta_1) d\eta_1 .
\end{align*}
The key observation here is that if we are away from time resonances, that is if the multiplier $(\vert \xi \vert^2 - \vert \eta_1 \vert^2)^{-1}$ can be seen as a standard Fourier multiplier, then we can use the decay of $e^{it \Delta} f$ in $L^6$ to balance the $t$ factor. \\
Overall we have two strategies: one that works well away from space resonances, and the other away from time resonances. Since the space-time resonance set is reduced to the origin (that is the multipliers $(\vert \xi \vert^2 - \vert \eta_1 \vert^2)^{-1}$ and $\eta_1/\vert \eta_1 \vert^2$ are both singular simultaneously only at the origin) we use the appropriate one depending on which region of the frequency space we are located in. This general scheme was developed by the author in \cite{L}. 

\subsection{Organization of the paper}
We start by recalling some known smoothing and Strichartz estimates in Section \ref{recallsmoothing}. We then prove easy corollaries of these that are tailored to our setting. The next three sections are dedicated to the main estimate \eqref{goal1} on the $X$ norm of the solution: Section \ref{expansion} is devoted to expanding the derivative of the Fourier transform of the profile as a series. In Section \ref{firsterm} we estimate the first iterates. As we pointed out above, this is a key step since our multilinear approach essentially reduces the estimation of the $n-$th iterates to that of the first iterates. Finally we prove in Section \ref{boundmultilinear} that the $L^2 _x$ norm of the general term of the series representation of $\partial_{\xi} \widehat{f}$ decays fast (at least like $\delta^n$ for some $\delta<1$). This allows us to conclude that the series converges. We start in Section \ref{multilinkey} by developing our key multilinear lemmas that incorporate the smoothing effect of the linear Schr\"{o}dinger flow in the iteration. They are then applied to prove the desired bounds on the $n-$th iterates.
\\
We end the paper with the easier energy estimate \eqref{goal2} in Section \ref{energy}, which concludes the proof.
\\\\
\textbf{Acknowledgments:} The author is very thankful to his PhD advisor Prof. Pierre Germain for the many enlightening discussions that led to this work. He also wishes to thank Prof. Yu Deng for very interesting discussions on related models.

\section{Smoothing and Strichartz estimates} \label{recallsmoothing}

\subsection{Known results}
In this section we recall some smoothing and Strichartz estimates from the literature. In this paper we will use easy corollaries of these estimates (see next subsection) to prove key multilinear Lemmas in Section \ref{boundmultilinear}.
\\
\\
We start with the classical smoothing estimates of C. Kenig, G. Ponce and L. Vega (\cite{KPV}, theorem 2.1, corollary 2.2, theorem 2.3). Heuristically  the dispersive nature of the Schr\"{o}dinger equation allows, at the price of space localization, to gain one half of a derivative in the homogeneous case and one derivative in the inhomogeneous case.  
\begin{lemma}\label{smoothing}
We have for all $j \in \lbrace 1;2;3 \rbrace:$
\begin{align} \label{smo1}
\Vert \big[ D_{x_j}^{1/2} e^{it \Delta} f \big](x) \Vert_{L^{\infty}_{x_j} L^{2}_{\widetilde{x}_j,t}} \lesssim \Vert f \Vert_{L^2_x},
\end{align}
its dual
\begin{align} \label{smo2}
\Bigg \Vert D_{x_j}^{1/2} \int \big[ e^{it \Delta} f \big](\cdot,t) dt \Bigg \Vert_{L^2_x} \lesssim \Vert f \Vert_{L^1_{x_j} L^2_{t,\widetilde{x_j}}},
\end{align}
and the inhomogeneous version 
\begin{align} \label{smo3}
\Bigg \Vert D_{x_j} \int_{0 \leqslant s \leqslant t} \big[ e^{i(t-s) \Delta} f(\cdot,s) \big] ds \Bigg \Vert_{L^{\infty}_{x_j} L^2_{t,\widetilde{x_j}}} \lesssim \Vert f \Vert_{L^1_{x_j} L^2_{s,\widetilde{x_j}}}.
\end{align}
\end{lemma}
We will also need the following estimate proved by A. Ionescu and C. Kenig (\cite{IK}, Lemma 3).
\begin{lemma} \label{smostri}
We have for $j \in \lbrace 1;2;3 \rbrace:$
\begin{align*}
\Bigg \Vert D_{x_j}^{1/2} \int_{0 \leqslant s \leqslant t} e^{i(t-s) \Delta} F(s,\cdot) ds \Bigg \Vert_{L^{\infty}_{x_j} L^{2}_{t,\widetilde{x_j}}} & \lesssim \Vert F \Vert_X , 
\end{align*}
where
\begin{align*}
\Vert F \Vert_X = \inf_{F = F^{(1)} + F^{(2)}} \Vert F^{(1)} \Vert_{L^1_t L^2 _x} + \Vert F^{(2)} \Vert_{L^2 _t L^{6}_x}.
\end{align*}
\end{lemma}
This lemma has the following straightforward corollary
\begin{corollary} \label{smostriu}
We have for $j \in \lbrace 1,2,3 \rbrace$
\begin{align*}
\Bigg \Vert D_{x_j}^{1/2} \int_{0 \leqslant s \leqslant t} e^{i(t-s) \Delta} F(s,\cdot) ds \Bigg \Vert_{L^{\infty}_{x_1} L^{2}_{t,\widetilde{x_1}}} & \lesssim \Vert F \Vert_{L^{p'}_t L^{q'} _x}  
\end{align*}
for $(p,q)$ a Strichartz admissible pair, that is 
$2 \leqslant p,q \leqslant \infty $ and
\begin{align*}
\frac{2}{p} + \frac{3}{q} = \frac{3}{2}.
\end{align*}
\end{corollary}
Now recall the following, see for example \cite{KT}.
\begin{lemma} \label{Strichartz}
Let $(p,q), (r,l)$ be two admissible Strichartz pairs (see Corollary \ref{smostriu} for the definition) \\
Then we have the estimates
\begin{align} \label{Str1}
\Vert e^{it\Delta} f \Vert_{L^{p}_t L^q _x} \lesssim \Vert f \Vert_{L^2 _x},
\end{align}
and
\begin{align} \label{Str2}
\Bigg \Vert \int_{\mathbb{R}} e^{-is \Delta} F(s,\cdot) \ ds \Bigg \Vert_{L^2 _x} \lesssim \Vert F \Vert_{L^{p'}_t L^{q'}_x} ,
\end{align}
and finally the inhomogeneous version
\begin{align} \label{Str3}
\Bigg \Vert \int_{s \leqslant t} e^{i(t-s) \Delta} F(s,\cdot) ds \Bigg \Vert_{L^{p}_t L^{q}_x} \lesssim \Vert F \Vert_{L^{r'}_t L^{l'}_x}.
\end{align}
\end{lemma}

\subsection{Basic lemmas for the magnetic part}
In this subsection we prove easy corollaries of the estimates from the previous section. They will be useful in the multilinear analysis.
\begin{lemma} \label{mgnmgn}
Recall that the potentials satisfy 
\begin{align*}
\Vert a_i \Vert_{Y} + \Vert \langle x \rangle a_i \Vert_{Y} + \Vert (1-\Delta)^{5} a_i \Vert_{Y} & \leqslant  \delta.
\end{align*}
We have the following bounds for every $k,j \in \lbrace 1;2;3 \rbrace$:
\begin{align} \label{aim}
\bigg \Vert a_k (x) D_{x_k} \int_{\tau_3 \leqslant \tau_2} e^{i(\tau_2-\tau_3)\Delta} \mathcal{F}^{-1}_{\eta_2} F_2(\eta_2,\tau_3) d\tau_3 \bigg \Vert_{L^1_{x_j} L^2_{\tau_2,\widetilde{x_j}}} 
& \lesssim \delta \Vert \mathcal{F}^{-1}_{\eta_2} F_2(\eta_2) \Vert_{L^{1}_{x_k} L^2_{\tau_3,\widetilde{x_k}}} .
\end{align}
\end{lemma}
\begin{proof}
To simplify notations, we denote
\begin{align*}
\widetilde{F_2}(x) = \int_{\tau_3 \leqslant \tau_2} e^{-i\tau_3 \Delta} \mathcal{F}^{-1}_{\eta_2} F_2(\eta_2,\tau_3) d\tau_3.
\end{align*}
We proceed by duality. Let $h(x,\tau_2) \in L^{\infty}_{x_j} L^{2}_{\tau_2,\widetilde{x_j}}$. We test the expression above against that function and write using the Cauchy-Schwarz inequality
\begin{align*} 
& \Bigg \vert \int_{\mathbb{R}^4} a_k (x) D_{x_k} e^{i\tau_2\Delta} \big( \widetilde{F_2}\big) h(x,\tau_2) d\tau_2 dx \Bigg \vert \\
& \lesssim \Bigg \Vert \vert a_k \vert^{1/2} \big \vert D_{x_k} e^{i\tau_2 \Delta} \big(\widetilde{F_2}\big) \big \vert  \Bigg \Vert_{L^2_{x,\tau_2}} \big \Vert \vert  a_k \vert^{1/2} \vert h \vert   \big \Vert_{L^2_{x,\tau_2}} \\
& \lesssim \big \Vert \vert a_k \vert^{1/2} \big \Vert_{L^{2}_{x_j} L^{\infty}_{\widetilde{x_j}}} \big \Vert \vert a_k \vert^{1/2} \big \Vert_{L^{2}_{x_k} L^{\infty}_{\widetilde{x_k}}} \Vert h \Vert_{L^{\infty}_{x_j} L^{2}_{\tau_2,\widetilde{x_j}}} \big \Vert  D_{x_k} e^{i\tau_2 \Delta} \big(\widetilde{F_2}\big) \big \Vert_{L^{\infty}_{x_k} L^{2}_{\tau_2,\widetilde{x_k}}}.
\end{align*}
%Proof: we prove $\Vert \vert a \vert^{1/2} f(x,\tau) \Vert_{L^2_{\tau,x}} $ by duality: $g \in L^2_{\tau,x}$
%\begin{align*}
% &\int \vert a \vert^{1/2} F(x) g(x) d\widetilde{x_j} x_j \\
% & \leqslant \int_{x_j} \sup_{\widetilde{x_j}} \vert a \vert^{1/2} \Vert F \Vert_{L^2_{\tau_2,\widetilde{x_j}}} \Vert g Vert_{\tau_2,\widetilde{x_j}} dx_j \\
% &\leqslant \Vert \sup_{\widetilde{x_j}} \vert a \vert^{1/2} \Vert_{L^2_{x_j}}  \Vert F \Vert_{L^{\infty}_{x_j}L^2_{\tau,\widetilde{x_j}}} \Vert g \Vert_{L^2_{\tau,x}}
%\end{align*}
Therefore using Lemma \ref{smoothing}, \eqref{smo3}, we can conclude that
\begin{align*}
\textrm{L.H.S.}~~\eqref{aim} & \lesssim \delta \big \Vert D_{x_k} e^{i\tau_2 \Delta} \big(\widetilde{F_2}\big) \Vert_{L^{\infty}_{x_k} L^2 _ {\tau_2,\widetilde{x_k}}} \\
 & \lesssim \delta \Vert \mathcal{F}^{-1}_{\eta_2} F_2 \Vert_{L^{1}_{x_k}L^2 _ {\tau_3,\widetilde{x_k}}}.
\end{align*}
\end{proof}

We also have the following related lemma:
\begin{lemma} \label{mgnmgnfin}
We have the following bounds for every $k,j \in \lbrace 1;2;3 \rbrace$:
\begin{align*} 
\bigg \Vert a_k (x) D_{x_k}^{1/2} e^{i\tau_2\Delta} \big( \mathcal{F}^{-1}_{\eta_2} F_2(\eta_2) \big) \bigg \Vert_{L^1_{x_j} L^2_{\tau_2,\widetilde{x_j}}} 
& \lesssim \delta \Vert \mathcal{F}^{-1}_{\eta_2} F_2(\eta_2) \Vert_{L^2_{x}} 
\end{align*}
and
\begin{align*} 
\bigg \Vert a_k (x) D_{x_k}^{1/2} \int_{s \leqslant \tau_2} e^{i(\tau_2-s)\Delta} \big( \mathcal{F}^{-1}_{\eta_2} F_2(s, \eta_2) \big) ds \bigg \Vert_{L^1_{x_j} L^2_{\tau_2,\widetilde{x_j}}} 
& \lesssim \delta \Vert \mathcal{F}^{-1}_{\eta_2} F_2(s,\eta_2) \Vert_{L^{p'}_s L^{q'}_{x}} ,
\end{align*}
for every Strichartz-admissible pair $(p,q).$ 
\end{lemma}
\begin{proof}
The proof of this lemma is almost identical to that of the previous lemma. For the first inequality the inhomogeneous smoothing estimate is replaced by the homogeneous one. 
\\
For the second inequality we use Corollary \ref{smostriu}.
\end{proof}

We record another lemma of the same type:

\begin{lemma} \label{potmgn}
We have the bound
\begin{align*}
\bigg \Vert a_k (x) D_{x_k} \int_{\tau_3 \leqslant \tau_2} e^{i(\tau_2-\tau_3)\Delta} \mathcal{F}^{-1}_{\eta_2} F_2(\eta_2,\tau_3) d\tau_3 \bigg \Vert_{L^2_{\tau_2} L^{6/5}_x} \lesssim \delta \Vert \mathcal{F}^{-1}_{\eta_2} F_2(\eta_2) \Vert_{L^{1}_{x_k}L^2 _ {\tau_2,\widetilde{x_k}}}.
\end{align*}
\end{lemma}
\begin{proof}
In this case we must bound
\begin{align*}
\Bigg \Vert a_k (x) D_{x_k} \int_{\tau_3 \leqslant \tau_2} e^{i(\tau_2-\tau_3)\Delta} \mathcal{F}^{-1}_{\eta_2} F_2(\eta_2,\tau_3) d\tau_3 \Bigg \Vert_{L^2_{\tau_2} L^{6/5}_x} \\
:= \Vert a_k (x) D_{x_k} e^{i \tau_2 \Delta} \big(\widetilde{F_2} \big) \Vert_{L^2_{\tau_2} L^{6/5}_x}.
\end{align*}
The proof is similar to that of Lemma \ref{mgnmgn}. \\
We proceed by duality. Let $h(x,\tau_2) \in L^{2}_{\tau_2} L^6_{x}.$ We pair the expression with $h$ and use the Cauchy-Schwarz inequality in $\tau_2$ and H\"{o}lder's inequality in $x$. The pairing is bounded by
\begin{align*}
& \Bigg \vert \int_{x_k} \int_{\widetilde{x_k}} a_k(x) \Bigg( \int_{\tau_2} \bigg \vert D_{x_k} e^{i\tau_2 \Delta} \big(\widetilde{F_2} \big) \bigg \vert^2 d\tau_2 \Bigg)^{1/2} \bigg( \int_{\tau_2}  (h(x, \tau_2))^2 d\tau_2 \bigg)^{1/2} d\widetilde{x_k} dx_k \bigg \vert \\
& \lesssim \int_{x_k} \Vert a_k \Vert_{L^{3}_{\widetilde{x_k}}} \big \Vert D_{x_k} e^{i\tau_2 \Delta} \big(\widetilde{F_2} \big)  \big \Vert_{L^{2}_{\widetilde{x_k}} L^2_{\tau_2}} \Vert h \Vert_{L^{6}_{\widetilde{x_k}}L^{2}_{\tau_2}} dx_k \\
& \lesssim \bigg \Vert D_{x_k} e^{i\tau_2 \Delta} \big(\widetilde{F_2} \big) \bigg \Vert_{L^{\infty}_{x_k} L^{2}_{\tau_2,\widetilde{x_k}}} \Vert a_k \Vert_{L^{6/5}_{x_k} L^{3}_{\widetilde{x_k}}}  \Vert h  \Vert_{L^{6}_{x} L^2_{\tau_2}} \\
& \lesssim \bigg \Vert D_{x_k} e^{i\tau_2 \Delta} \big(\widetilde{F_2} \big) \bigg \Vert_{L^{\infty}_{x_k} L^{2}_{\tau_2,\widetilde{x_k}}} \Vert a_k \Vert_{L^{6/5}_{x_k} L^{3}_{\widetilde{x_k}}}  \Vert h  \Vert_{L^2_{\tau_2} L^6 _x},
\end{align*} 
where for the last line we used Minkowski's inequality. \\
\\
We can conclude that, using smoothing estimates from Lemma \ref{smoothing}, \eqref{smo3},
\begin{align*}
 \bigg \Vert D_{x_k} e^{i\tau_2 \Delta} \big(\widetilde{F_2} \big) \bigg \Vert_{L^{\infty}_{x_k} L^{2}_{\tau_2,\widetilde{x_k}}}   
& =  \bigg \Vert D_{x_k} \int_{\tau_3 \leqslant \tau_2} e^{i(\tau_2-\tau_3) \Delta} \mathcal{F}^{-1}_{\eta_2} \big( F_2 (\eta_2) \big) d\tau_3 \bigg \Vert_{L^{\infty}_{x_k} L^2_{\tau_2, \widetilde{x_k}}}  \\
& \lesssim \big \Vert \mathcal{F}^{-1}_{\eta_2} F_2 (\eta_2) \big \Vert_{L^1 _{x_k} L^{2}_{\tau_3,\widetilde{x_k}}}.
\end{align*}
\end{proof}

Now we write a similar lemma for a slightly different situation:

\begin{lemma} \label{potmgnfin}
We have the bounds
\begin{align*}
\Vert a_k (x) D_{x_k}^{1/2} e^{i \tau_2\Delta} \mathcal{F}^{-1}_{\eta_2} F_2(\eta_2) \big) \Vert_{L^2_{\tau_2} L^{6/5}_x}  \lesssim \delta \Vert \mathcal{F}^{-1}_{\eta_2} F_2(\eta_2) \Vert_{L^2 _ {x}},
\end{align*}
and 
\begin{align*}
\Bigg \Vert a_k (x) D_{x_k}^{1/2} \int_{s \leqslant \tau_2} e^{i (\tau_2-s) \Delta} \mathcal{F}^{-1}_{\eta_2} F_2(\eta_2) \big) ds \Bigg \Vert_{L^2_{\tau_2} L^{6/5}_x}  \lesssim \delta \Vert \mathcal{F}^{-1}_{\eta_2} F_2(\eta_2) \Vert_{L^{p'}_t L^{q'}_ {x}},
\end{align*}
for every Strichartz admissible pair $(p,q).$
\end{lemma}
\begin{proof}
The proof is the same as that of the previous lemma, with the inhomogenenous smoothing estimate replaced by its homogeneous version for the first inequality. \\
For the second one, we use Corollary \ref{smostriu} instead of Lemma \ref{smoothing}.
\end{proof}

\subsection{Basic lemmas for the electric part}
We record lemmas that will allow us to control electric terms.
\begin{lemma} \label{potpot}
Recall that the potential $V$ satisfies 
\begin{align*}
\Vert V \Vert_{Y} + \Vert \langle x \rangle V \Vert_{Y} + \Vert (1-\Delta)^{5} V \Vert_{Y} & \leqslant \delta. 
\end{align*}
We have the following bound:
\begin{align*}
\Bigg \Vert V(x) \int_{\tau_3 \leqslant \tau_2} e^{i(\tau_2-\tau_3)\Delta} \mathcal{F}^{-1}_{\eta_2} F_2(\eta_2,\tau_3) d\tau_3 \Bigg \Vert_{L^2 _{\tau_2} L^{6/5}_x} \lesssim \delta \Vert \mathcal{F}^{-1}_{\eta_2} F_2(\eta_2) \Vert_{L^2 _{\tau_3} L^{6/5}_x} .
\end{align*}
\end{lemma}
\begin{proof}
Using H\"{o}lder's inequality and Strichartz estimates we write that
\begin{align*}
& \bigg \Vert V(x) \int_{\tau_3 \leqslant \tau_2} e^{i(\tau_2-\tau_3)\Delta} \mathcal{F}^{-1}_{\eta_2} F_2(\eta_2,\tau_3) d\tau_3 \bigg \Vert_{L^2_{\tau_2} L^{6/5}_x} \\
& \lesssim \Vert V \Vert_{L^{3/2}_x} \bigg \Vert \int_{\tau_3 \leqslant \tau_2} e^{i(\tau_2-\tau_3)\Delta} \mathcal{F}^{-1}_{\eta_2} F_2(\eta_2,\tau_3) d\tau_3 \bigg \Vert_{L^2_{\tau_2} L^{6}_x} \\
& \lesssim \Vert V \Vert_{L^{3/2}_x} \Vert \mathcal{F}^{-1}_{\eta_2} F_2(\eta_2,\tau_3) \Vert_{L^{2}_{\tau_3}L^{6/5}_x} .
\end{align*}
\end{proof}
And we have the following related Lemma in the homogeneous case:
\begin{lemma} \label{potpotfin}
We have the following bound:
\begin{align*}
\big \Vert V(x) e^{i\tau_2\Delta} \mathcal{F}^{-1}_{\eta_2} F_2(\eta_2) \big \Vert_{L^2 _{\tau_2} L^{6/5}_x} \lesssim \delta \Vert \mathcal{F}^{-1}_{\eta_2} F_2(\eta_2) \Vert_{L^{2}_x} ,
\end{align*}
and 
\begin{align*}
\Bigg \Vert V(x) \int_{s \leqslant  \tau_2} e^{i(s-\tau_2)\Delta} \mathcal{F}^{-1}_{\eta_2} F_2 (s,\eta_2) ds \Bigg \Vert_{L^2 _{\tau_2} L^{6/5}_x} & \lesssim \delta \Vert \mathcal{F}^{-1}_{\eta_2} F_2 (\eta_2) \Vert_{L^{p'}_t L^{q'}_x},
\end{align*}
for every Strichartz admissible pair $(p,q).$
\end{lemma}
We will also need 
\begin{lemma} \label{mgnpot}
We have the following bound:
\begin{align*}
\bigg \Vert V(x) \int_{\tau_3 \leqslant \tau_2} e^{i(\tau_2-\tau_3)\Delta} \mathcal{F}^{-1}_{\eta_2} F_2(\eta_2,\tau_3) d\tau_3 \bigg \Vert_{L^1_{x_j} L^{2}_{\tau_2,\widetilde{x_j}}} \lesssim \delta \Vert \mathcal{F}^{-1}_{\eta_2} F_2(\eta_2) \Vert_{L^2_{\tau_3} L^{6/5}_x}  .
\end{align*}
\end{lemma}
\begin{proof}
We must bound
\begin{align*}
\bigg \Vert V (x) \int_{\tau_3 \leqslant \tau_2} e^{i(\tau_2-\tau_3)\Delta} \mathcal{F}^{-1}_{\eta_2} F_2(\eta_2,\tau_3) d\tau_3 \bigg \Vert_{L^1_{x_j} L^2_{\tau_2,\widetilde{x_j}}} \\
:= \Vert V (x)  e^{i \tau_2 \Delta} \big( \widetilde{F_2} \big) \Vert_{L^1_{x_j} L^2_{\tau_2,\widetilde{x_j}}}.
\end{align*}
The reasoning is similar to the one used for Lemma \ref{potmgn}. \\
We proceed by duality. Let $h(x,\tau_2) \in L^{\infty}_{x_j} L^2_{\tau_2,\widetilde{x_j}}.$ We pair the expression with $h$ and use H\"{o}lder inequality. The pairing is bounded by
\begin{align*}
& \int_{x_j} \int_{\widetilde{x_j}}  V(x) \Bigg( \int_{\tau_2} \bigg \vert  e^{i\tau_2 \Delta} \big(\widetilde{F_1} \big) \bigg \vert^2 d\tau_2 \Bigg)^{1/2} \bigg( \int_{\tau_2}  (h(x, \tau_2))^2 d\tau_2 \bigg)^{1/2} d\widetilde{x_j} dx_j \\
& \lesssim \int_{x_j} \Vert V \Vert_{L^{3}_{\widetilde{x_j}}} \big \Vert  e^{i\tau_2 \Delta} \big(\widetilde{F_1} \big)  \big \Vert_{L^{6}_{\widetilde{x_j}} L^2_{\tau_2}} \Vert h \Vert_{L^{2}_{\tau_2,\widetilde{x_j}}} dx_j \\
& \leqslant  \Vert h \Vert_{L^{\infty}_{x_j} L^{2}_{\tau_2,\widetilde{x_j}}}  \int_{x_j} \Vert V \Vert_{L^{3}_{\widetilde{x_j}}} \big \Vert  e^{i\tau_2 \Delta} \big(\widetilde{F_1} \big)  \big \Vert_{L^{6}_{\widetilde{x_j}} L^2_{\tau_2}} dx_j \\
& \leqslant  \Vert h \Vert_{L^{\infty}_{x_j} L^{2}_{\tau_2,\widetilde{x_j}}} \Vert V \Vert_{L^{6/5}_{x_j} L^{3}_{\widetilde{x_j}}} \big \Vert  e^{i\tau_2 \Delta} \big(\widetilde{F_1}\big)  \big \Vert_{L^{6}_{x} L^2_{\tau_2}}.
\end{align*} 
Now we use Minkowski's inequality and retarded Strichartz estimates from Lemma \ref{Strichartz} to write that
\begin{align*}
 \big \Vert  e^{i\tau_2 \Delta} \big(\widetilde{F_1}(\tau_2) \big)  \big \Vert_{L^{6}_{x} L^2_{\tau_2}} & \lesssim \big \Vert  e^{i\tau_2 \Delta} \big(\widetilde{F_1} \big)  \big \Vert_{ L^2_{\tau_2} L^{6}_{x}} \\
& =  \bigg \Vert \int_{\tau_3 \leqslant \tau_2} e^{i(\tau_2-\tau_3) \Delta} \mathcal{F}^{-1}_{\eta_2} \big( F_1 (\eta_2) \big) d\tau_3 \bigg \Vert_{ L^2_{\tau_2} L^{6}_{x}}  \\
& \lesssim  \big \Vert \mathcal{F}^{-1}_{\eta_2} F_1 (\eta_2) \big \Vert_{L^2 _{\tau_3} L^{6/5} _x}.
\end{align*}
\end{proof}
We have the similar analogous Lemma (for the homogeneous case)
\begin{lemma} \label{mgnpotfin}
We have
\begin{align*}
\Vert V(x) e^{i \tau_2 \Delta} \mathcal{F}^{-1}_{\eta_2} F_2(\eta_2) \Vert_{L^1_{x_j} L^{2}_{\tau_2,\widetilde{x_j}}} \lesssim \delta \Vert \mathcal{F}^{-1}_{\eta_2} F_2(\eta_2) \Vert_{L^{2}_x}
\end{align*}
and 
\begin{align*}
\Bigg \Vert V(x) \int_{s \leqslant \tau_2} e^{i(s-\tau_2)\Delta} \mathcal{F}^{-1}_{\eta_2} F_2(s,\eta_2) ds \Bigg \Vert_{L^1_{x_j} L^{\infty}_{\tau_2,\widetilde{x_j}}} \lesssim \delta \Vert \mathcal{F}^{-1}_{\eta_2} F_2 \Vert_{L^{p'}_t L^{q'}_x}
\end{align*}
for every admissible Strichartz pair $(p,q).$
\end{lemma}
\subsection{Basic bilinear lemmas}
We give an easy bilinear lemma 
\begin{lemma}\label{bilinit}
We have the bounds
\begin{align*}
\Bigg \Vert \mathcal{F}^{-1}_{\eta_1} \int_{\eta_2} \widehat{W}(\eta_1-\eta_2) m(\eta_2) F_2 (\eta_2,\tau_2) d\eta_2 \Bigg \Vert_{L^2_{\tau_2} L^{6/5}_x} \lesssim \Vert W \Vert_{L^p _x} \Vert \check{m} \Vert_{L^1} \Vert \mathcal{F}^{-1}_{\eta_2} F_2 \Vert_{L^2 _{\tau_2} L^q _x},
\end{align*}
where $\frac{5}{6} = \frac{1}{p} + \frac{1}{q}$ and 
\begin{align*}
\Bigg \Vert \mathcal{F}^{-1}_{\eta_1} \int_{\eta_2} \widehat{W}(\eta_1-\eta_2) m(\eta_2) F_2 (\eta_2,\tau_2) d\eta_2 \Bigg \Vert_{L^1_{x_j} L^2_{\tau_2,\widetilde{x_j}}} \lesssim \Vert W \Vert_{L^{\infty}_x} \Vert \check{m} \Vert_{L^1} \Vert \mathcal{F}^{-1}_{\eta_2} F_2 \Vert_{L^1_{x_j} L^2_{\tau_2,\widetilde{x_j}}},
\end{align*}
and for $c$ small number
\begin{align*}
\Bigg \Vert \mathcal{F}^{-1}_{\eta_1} \int_{\eta_2} \widehat{W}(\eta_1-\eta_2) m(\eta_2) F_2 (\eta_2,\tau_2) d\eta_2 \Bigg \Vert_{L^1_{x_j} L^2_{\tau_2,\widetilde{x_j}}} \lesssim \Vert W \Vert_{L^{\infty}_{x_j} L^{\frac{2+2c}{c}}_{\widetilde{x_j}}} \Vert \check{m} \Vert_{L^1} \Vert \mathcal{F}^{-1}_{\eta_2} F_2 \Vert_{L^1_{x_j}  L^2_{\tau_2} L^{2(1+c)}_{\widetilde{x_j}}}.
\end{align*}
Similarly
\begin{align*}
\Bigg \Vert \mathcal{F}^{-1}_{\eta_1} \int_{\eta_2} \widehat{W}(\eta_1-\eta_2) m(\eta_2) F_2 (\eta_2,\tau_2) d\eta_2 \Bigg \Vert_{L^2_{\tau_2} L^{6/5}_{x}} \lesssim \Vert W \Vert_{L^{\frac{6+6c}{5c}}_{x}} \Vert \check{m} \Vert_{L^1} \Vert \mathcal{F}^{-1}_{\eta_2} F_2 \Vert_{L^2_{\tau_2} L^{\frac{6}{5}(1+c)}_x}.
\end{align*}
Finally
\begin{align*}
\Bigg \Vert \mathcal{F}^{-1}_{\eta_1} \int_{\eta_2} \widehat{W}(\eta_1-\eta_2) m(\eta_2) F_2 (\eta_2) d\eta_2 \Bigg \Vert_{L^2_{x}} \lesssim \Vert W \Vert_{L^{\infty}_x} \Vert \check{m} \Vert_{L^1} \Vert \mathcal{F}^{-1}_{\eta_2} F_2 \Vert_{L^2_{x}}.
\end{align*}
\end{lemma}
\begin{proof}
The proofs are almost the same as that of Lemma \ref{bilin} therefore they are omitted.
\end{proof}
We end this section with a key quantity used in the estimation of the iterates.
\begin{definition}
Let $C_0$ be the largest of all the implicit constants that appear in the inequalities from Lemmas \ref{mgnmgn}, \ref{mgnmgnfin}, \ref{potmgn}, \ref{potmgnfin}, \ref{potpot}, \ref{potpotfin}, \ref{mgnpot}, \ref{mgnpotfin}, and \ref{bilinit}. We denote $C = 10^{10} C_0^{10}.$ The choice of this constant is arbitrary: we just need a large number to account for the numerical constants that appear in the iteration below.
\end{definition}

\section{Expansion of the solution as a series} \label{expansion}
In this section and the next two, our goal is to prove \eqref{goal1}. To do so we start as in \cite{L} by expanding $\partial_{\xi_l} \widehat{f}$ as a power series. This is done through integrations by parts in time. The full details are presented in this section.

\subsection{First expansion} \label{firstexpansion}
The Duhamel formula for \eqref{NLSmagn} reads:
\begin{align}\label{Duhamel}
\hat{f}(t,\xi) &= \hat{f}_0(\xi) - \frac{i}{(2 \pi)^3} \sum_{i=1}^3 \int_1 ^t \int_{\mathbb{R}^3} e^{is(\vert \xi \vert^2 - \vert \eta_1 \vert^2)} \widehat{a_i}(\xi-\eta_1) \eta_{1,i} \widehat{f}(s,\eta_1) d\eta_1 ds \\
\notag&-\frac{i}{(2 \pi)^3}\int_1 ^t e^{is \vert \xi \vert^2} \int_{\mathbb{R}^3} \widehat{V}(\xi - \eta_1) e^{-is \vert \eta_1 \vert ^2} \widehat{f}(s,\eta_1) d\eta_1 ds \\
\notag& -\frac{i}{(2 \pi)^3} \int_1 ^t e^{is \vert \xi \vert ^2} \int_{\mathbb{R}^3} e^{-is \vert \xi-\eta_1 \vert^2} e^{-i s \vert \eta_1 \vert^2} \widehat{f}(s,\eta_1) \widehat{f}(s,\xi-\eta_1) d\eta_1 ds .
\end{align}
We start by localizing in $\xi$ and taking a derivative in $\xi_l$: \\
\begin{align} \label{estimationX}
\partial_{\xi_l} \widehat{f_k}(t,\xi)&= \partial_{\xi_l} \big(e^{i \vert \xi \vert^2} \widehat{u_{1,k}}(\xi) \big) \\
\label{M2}& -\frac{i}{(2 \pi)^3}  \sum_{k_1 \in \mathbb{Z}} \sum_{i=1}^3 P_k(\xi) \int_{1} ^t  \int_{\mathbb{R}^3} 2 i s \xi_l e^{is (\vert \xi \vert^2 -\vert \eta_1 \vert^2)} \eta_{1,i} \widehat{f_{k_1}}(s,\eta_1) \widehat{a_i}(s,\xi-\eta_1) d\eta_1 ds \\
\label{M6}& -\frac{i}{(2 \pi)^3} P_k(\xi) \sum_{k_1 \in \mathbb{Z}} \int_{1}^t  \int_{\mathbb{R}^3} 2 i s \xi_l e^{is (\vert \xi \vert^2 -\vert \eta_1 \vert^2)} \widehat{f_{k_1}}(s,\eta_1) \widehat{V}(s,\xi-\eta_1) d\eta_1 ds \\
\notag & + \lbrace \textrm{remainder terms} \rbrace,
\end{align}
where the remainder terms are given by:
\begin{align}
\notag &  \lbrace \textrm{remainder terms} \rbrace \\
\label{M1}& = - \frac{i}{(2 \pi)^3} \sum_{k_1 \in \mathbb{Z}} \sum_{i=1}^3 P_k(\xi) \int_{1} ^t  \int_{\mathbb{R}^3}  e^{is (\vert \xi \vert^2 -\vert \eta_1 \vert^2)} \eta_{1,i} \widehat{f_{k_1}}(s,\eta_1) \partial_{\xi_l} \widehat{a_i}(s,\xi-\eta_1) d\eta_1 ds  \\
\label{M3}&-\frac{i}{(2 \pi)^3} 1.1^{-k} \phi'(1.1^{-k}\xi) \frac{\xi_l}{\vert \xi \vert} \sum_{k_1 \in \mathbb{Z}}  \sum_{i=1}^3  \int_{1} ^t  \int_{\mathbb{R}^3}  e^{is (\vert \xi \vert^2 -\vert \eta_1 \vert^2)} \eta_{1,i} \widehat{f_{k_1}}(s,\eta_1)  \widehat{a_l}(s,\xi-\eta_1) d\eta_1 ds \\
\label{M4}& - \frac{i}{(2 \pi)^3} P_k(\xi) \sum_{k_1 \in \mathbb{Z}} \int_{1}^t \int_{\mathbb{R}^3}  e^{is (\vert \xi \vert^2 -\vert \eta_1 \vert^2)} \widehat{f_{k_1}}(s,\eta_1) \partial_{\xi_l} \widehat{V}(s,\xi-\eta_1) d\eta_1 ds 
\end{align}
\begin{align}
\label{M5}& - \frac{i}{(2 \pi)^3} 1.1^{-k} \phi'(1.1^{-k} \xi) \frac{\xi_l}{\vert \xi \vert} \sum_{k_1 \in \mathbb{Z}} \int_{1}^t \int_{\mathbb{R}^3}  e^{is (\vert \xi \vert^2 -\vert \eta_1 \vert^2)} \widehat{f_{k_1}}(\eta_1) \widehat{V}(s,\xi-\eta_1) d\eta_1 ds  \\
\label{M7}& -\frac{i}{(2 \pi)^3} P_k(\xi) \int_{1}^t  \int_{\mathbb{R}^3} 2 i s \eta_{1,l} e^{is (\vert \xi \vert^2 -\vert \xi -\eta_1 \vert^2 - \vert \eta_1 \vert^2)} \widehat{f}(s,\eta_1) \widehat{f}(s,\xi-\eta_1) d\eta_1 ds \\
\label{M8}& -\frac{i}{(2 \pi)^3} P_k(\xi) \int_{1}^t  \int_{\mathbb{R}^3} e^{is (\vert \xi \vert^2 -\vert \xi -\eta_1 \vert^2 - \vert \eta_1 \vert^2)} \widehat{f}(s,\eta_1) \partial_{\xi_l} \widehat{f}(s,\xi-\eta_1) d\eta_1 ds .
\end{align}
We will estimate these remainder terms directly. More precisely we will prove the following bounds in Section \ref{firsterm}:
\begin{proposition}
We will prove that
\begin{align*}
\Vert \eqref{M1}, \eqref{M3}, \eqref{M4}, \eqref{M5} \Vert_{L^2_x} & \lesssim \delta \varepsilon_1 \\
\Vert \eqref{M7},\eqref{M8} \Vert_{L^2_x} & \lesssim  \varepsilon_1 ^2 .
\end{align*}
\end{proposition}
The remaining two terms \eqref{M2} and \eqref{M6} cannot be estimated directly: they will be expanded as series by repeated integrations by parts in time. We explain this procedure in greater detail in the remainder of the section. \\
\\
To treat them in a unified way, we notice that they both have the form
\begin{align*}
2 \int_{1} ^t  \int_{\mathbb{R}^3} i s \xi_l e^{is (\vert \xi \vert^2 -\vert \eta_1 \vert^2)} \alpha_1( \eta_1) \widehat{f_{k_1}}(s,\eta_1) \widehat{W_1}(\xi-\eta_1) d\eta_1 ds ,
\end{align*}
where $W_1$ denotes either $a_i$ or $V, $ and $\alpha_1(\eta_1)= \eta_{1,i}$ if $W=a_i$ and $1$ if $W=V.$ \\
\\
We distinguish two cases:\\ 
\\
\underline{Case 1: $\vert k_1 - k \vert > 1$} \\
Then we integrate by parts in time using $\displaystyle \frac{1}{i(\vert \xi  \vert^2-\vert \eta_1 \vert^2)} \partial_s \big(e^{is(\vert \xi \vert^2 - \vert \eta_1 \vert^2)} \big) = e^{is(\vert \xi \vert^2 - \vert \eta_1 \vert^2)}$ and obtain for \eqref{M2}
\begin{align}
\label{B1}& 2 \int_{1} ^t  \int_{\mathbb{R}^3} i s \xi_l e^{is (\vert \xi \vert^2 -\vert \eta_1 \vert^2)} \alpha_1(\eta_1) \widehat{f_{k_1}}(s,\eta_1) \widehat{W_1}(\xi-\eta_1) d\eta_1 ds 
\\
\notag&= -  \int_{1} ^t  \int_{\mathbb{R}^3} \frac{2  s \xi_l}{\vert \xi \vert^2 - \vert \eta_1 \vert^2} e^{is (\vert \xi \vert^2 -\vert \eta_1 \vert^2)} \alpha_1(\eta_1) \partial_s \widehat{f_{k_1}}(s,\eta_1) \widehat{W_1}(\xi-\eta_1) d\eta_1 ds \\
\notag& + \int_{\mathbb{R}^3} \frac{2  t \xi_l}{\vert \xi \vert^2 - \vert \eta_1 \vert^2} e^{it (\vert \xi \vert^2 -\vert \eta_1 \vert^2)}  \alpha_1(\eta_1) \widehat{f_{k_1}}(t,\eta_1) \widehat{W_1}(\xi-\eta_1) d\eta_1 \\
\notag& -  \int_{\mathbb{R}^3} \frac{ 2 \xi_l}{\vert \xi \vert^2 - \vert \eta_1 \vert^2} e^{i (\vert \xi \vert^2 -\vert \eta_1 \vert^2)}  \alpha_1(\eta_1) \widehat{f_{k_1}}(1,\eta_1) \widehat{W_1}(\xi-\eta_1) d\eta_1 \\
\notag&  - \int_{1} ^t  \int_{\mathbb{R}^3} \frac{ 2 \xi_l}{\vert \xi \vert^2 - \vert \eta_1 \vert^2} e^{is (\vert \xi \vert^2 -\vert \eta_1 \vert^2)} \alpha_1(\eta_1)  \widehat{f_{k_1}}(s,\eta_1) \widehat{W_1}(\xi-\eta_1) d\eta_1 ds.
\end{align}
Now we use that $\partial_{s} f = e^{-is\Delta}(W_2 \widetilde{D_2} u + u^2)$ where $W_2$ stands for $V$ and the $a_i$ terms, and $\widetilde{D_2}$ stands for either $1$ if $W_2=V$ or $\partial_i$ if $W_2=a_i$. On the Fourier side, $\widetilde{D_2}$ is denoted $\alpha_2$ with a similar convention as for $\alpha_1.$ The summation on these different kinds of potentials is implicit in the following. \\ 
We get that
\begin{align}
\label{B2}&= -  \int_{1} ^t  \int_{\mathbb{R}^3} \frac{2  s \xi_l}{\vert \xi \vert^2 - \vert \eta_1 \vert^2} e^{is \vert \xi \vert^2} \alpha_1(\eta_1) \widehat{(W_2 \widetilde{D_2} u)_{k_1}}(s,\eta_1) \widehat{W_1}(\xi-\eta_1) d\eta_1 ds \\
\label{B2bis} &- \int_{1} ^t  \int_{\mathbb{R}^3} \frac{2  s \xi_l}{\vert \xi \vert^2 - \vert \eta_1 \vert^2} e^{is \vert \xi \vert^2} \alpha_1(\eta_1) \widehat{(u^2)_{k_1}}(s,\eta_1) \widehat{W_1}(\xi-\eta_1) d\eta_1 ds \\
\label{B3}& + \int_{\mathbb{R}^3} \frac{2  t \xi_l}{\vert \xi \vert^2 - \vert \eta_1 \vert^2} e^{it (\vert \xi \vert^2 -\vert \eta_1 \vert^2)}  \alpha_1(\eta_1) \widehat{f_{k_1}}(t,\eta_1) \widehat{W_1}(\xi-\eta_1) d\eta_1 \\
\label{B4}& -  \int_{\mathbb{R}^3} \frac{ 2 \xi_l}{\vert \xi \vert^2 - \vert \eta_1 \vert^2} e^{i (\vert \xi \vert^2 -\vert \eta_1 \vert^2)}  \alpha_1(\eta_1) \widehat{f_{k_1}}(1,\eta_1) \widehat{W_1}(\xi-\eta_1) d\eta_1 \\
\label{B5}&  - \int_{1} ^t  \int_{\mathbb{R}^3} \frac{ 2 \xi_l}{\vert \xi \vert^2 - \vert \eta_1 \vert^2} e^{is (\vert \xi \vert^2 -\vert \eta_1 \vert^2)} \alpha_1(\eta_1)  \widehat{f_{k_1}}(s,\eta_1) \widehat{W_1}(\xi-\eta_1) d\eta_1 ds.
\end{align}
All these terms, except for \eqref{B2}, will be estimated directly. We will prove the following bounds:
\begin{proposition}
We have
\begin{align*}
\Vert \eqref{B3}, \eqref{B4}, \eqref{B5} \Vert_{L^2_x} & \lesssim \delta \varepsilon_1 \\
\Vert \eqref{B2bis} \Vert_{L^2 _x} & \lesssim \delta \varepsilon_1 ^2.
\end{align*}
\end{proposition}
To deal with the remaining term \eqref{B2}, we will iterate the procedure presented here. Indeed if we write it as 
\begin{align*}
\int_{1} ^t  \int_{\mathbb{R}^3} \frac{\widehat{W_1}(\xi-\eta_1) P_{k_1}(\eta_1)}{\vert \xi \vert^2 - \vert \eta_1 \vert^2} \alpha_1(\eta_1) \int_{\mathbb{R}^3} 2i s \xi_l \alpha_2(\eta_2) e^{is(\vert \xi \vert^2 - \vert \eta_2 \vert^2)} \widehat{W_2}(\eta_1-\eta_2) \widehat{f}(s,\eta_2) d\eta_2 d\eta_1 ds.
\end{align*}
We see that the inner integral on $\eta_2$ is similar to the term we started with, namely \eqref{M2}, \eqref{M6}. The idea is then to iterate the procedure presented here for these terms.
\\
\\
\underline{Case 2: $\vert k-k_1 \vert \leqslant 1$} \\
In this case we integrate by parts in $\eta_1$ using $\displaystyle \frac{\eta_1 \cdot \nabla \big( e^{is \vert \eta_1 \vert^2} \big)}{2is \vert \eta_1 \vert^2} = e^{is \vert \eta_1 \vert^2},$ and obtain
\begin{align}
\notag \eqref{B1} &= \\ 
\label{c21}&\int_{1} ^t \int_{\mathbb{R}^3} \frac{\eta_{1,j} \xi_l  \alpha_1(\eta_1)}{\vert \eta_1 \vert^2}
e^{is(\vert \xi \vert^2 - \vert \eta_1 \vert^2)}\widehat{W_1}(\xi-\eta_1) \partial_{\eta_{1,j}} \widehat{f}(s,\eta_1) P_{k_1}(\eta_1) d\eta_1 ds \\
\label{c22}&+  \int_{1} ^t \int_{\mathbb{R}^3} \frac{\eta_{1,j} \xi_l  \alpha_1(\eta_1)}{\vert \eta_1 \vert^2}
e^{is(\vert \xi \vert^2 - \vert \eta_1 \vert^2)} \partial_{\eta_{1,j}} \widehat{W_1}(\xi-\eta_1) \widehat{f_{k_1}}(s,\eta_1) d\eta_1 ds \\
\label{c23}& + \int_{1} ^t \int_{\mathbb{R}^3} \partial_{\eta_{1,j}} \big( \frac{\eta_{1,j} \xi_l  \alpha_1(\eta_1)}{\vert \eta_1 \vert^2} \big) e^{is (\vert \xi \vert^2 - \vert \eta_1 \vert^2)} \widehat{W_1}(\xi-\eta_1) \widehat{f_{k_1}}(s,\eta_1) d\eta_1 ds \\
\label{c24}& + \int_{1} ^t \int_{\mathbb{R}^3}  \frac{\eta_{1,j} \xi_l  \alpha_1(\eta_1)}{\vert \eta_1 \vert^2}  e^{is (\vert \xi \vert^2 - \vert \eta_1 \vert^2)} \widehat{W_1}(\xi-\eta_1) \widehat{f_{k_1}}(s,\eta_1) 1.1^{-k_1} \phi'(1.1^{-k_1} \eta_1) \frac{\eta_{1,j}}{\vert \eta_1 \vert} d\eta_1 ds.
\end{align}
Note that there is an implicit sum on $j$ above. \\
For the easier terms we have the following estimates that will be proved in Section \ref{firsterm}
\begin{proposition}
We have the following bounds
\begin{align*}
\Vert \eqref{c22}, \eqref{c23}, \eqref{c24} \Vert_{L^2 _x} \lesssim \delta \varepsilon_1.
\end{align*}
\end{proposition}
For the main complicated term \eqref{c21} we integrate by parts in time in the inner integral. Since the denominator $\frac{1}{\vert \xi \vert^2 - \vert \eta_1 \vert^2}$ is singular, we must consider a regularization of that term. 
\\
We consider for $\beta>0$
\begin{align} \label{c21beta}
\int_{1} ^t \int_{\mathbb{R}^3} \frac{\eta_{1,j} \xi_l  \alpha_1(\eta_1)}{\vert \eta_1 \vert^2}
e^{-\beta s}e^{is(\vert \xi \vert^2 - \vert \eta_1 \vert^2)}\widehat{W_1}(\xi-\eta_1) \partial_{\eta_{1,j}} \widehat{f}(s,\eta_1) P_{k_1}(\eta_1) d\eta_1 ds.
\end{align}
Now we integrate by parts in time and obtain
\begin{align}
\notag \eqref{c21beta} &= \\
\label{iteration}& -\int_{1} ^t \int_{\mathbb{R}^3} \frac{\eta_{1,j} \xi_l \alpha_1(\eta_1)}{i(\vert \xi \vert^2 - \vert \eta_1 \vert^2 + i \beta) \vert \eta_1 \vert^2} \\
\notag & \times e^{is(\vert \xi \vert^2 - \vert \eta_1 \vert^2 + i \beta)}\widehat{W_1}(\xi-\eta_1) \partial_s \partial_{\eta_{1,j}} \widehat{f}(s,\eta_1) P_{k_1}(\eta_1) d\eta_1 ds \\
\label{restediff}&+ \int_{\mathbb{R}^3} \frac{\eta_{1,j} \xi_l \alpha_1(\eta_1)}{i(\vert \xi \vert^2 - \vert \eta_1 \vert^2 + i \beta) \vert \eta_1 \vert^2}
e^{it(\vert \xi \vert^2 - \vert \eta_1 \vert^2 + i \beta)}\widehat{W_1}(\xi-\eta_1) \partial_{\eta_{1,j}} \widehat{f}(t,\eta_1) P_{k_1}(\eta_1) d\eta_1 \\
\label{restefac}&- \int_{\mathbb{R}^3} \frac{\eta_{1,j} \xi_l \alpha_1(\eta_1)}{i(\vert \xi \vert^2 - \vert \eta_1 \vert^2 + i \beta)\vert \eta_1 \vert^2}
e^{i(\vert \xi \vert^2 - \vert \eta_1 \vert^2 + i \beta)}\widehat{W_1}(\xi-\eta_1) \partial_{\eta_{1,j}} \widehat{f}(1,\eta_1) P_{k_1}(\eta_1) d\eta_1 .
\end{align}
The terms \eqref{restediff} and \eqref{restefac} will be estimated in the same way in Section \ref{firsterm}. We will prove the following estimates:
\begin{proposition}
We have the bounds (that hold uniformly on $\beta$):
\begin{align*}
\Vert \eqref{restediff}, \eqref{restefac}  \Vert_{L^2_x} \lesssim \delta \varepsilon_1.
\end{align*}
\end{proposition}
Since the estimates hold uniformly on $\beta,$ we have, by lower semi-continuity of the norm:
\begin{align*}
\big \Vert \lim_{\beta \to 0, \beta>0} \eqref{restediff}, \eqref{restefac} \big \Vert_{L^2 _x} \leqslant \liminf_{\beta \to 0, \beta>0} \Vert \eqref{restediff}, \eqref{restefac}  \Vert_{L^2_x} \lesssim \delta \varepsilon_1.
\end{align*}
This is how all regularized terms will be handled since all our estimates are uniform on $\beta$. Therefore we will drop the regularizing factors $\beta$ to simplify notations.
\\
\\
For the remaining term \eqref{iteration} we use the following expression for $\partial_t \partial_{\eta_{1,j}} \widehat{f}$ (obtained by differentiating \eqref{Duhamel})
\begin{align} \label{identder}
& \partial_t \partial_{\xi_j} \widehat{f}(t,\xi) \\
\label{id1}& = 2 t \xi_j e^{it \vert \xi \vert^2} \int_{\mathbb{R}^3} \widehat{a_i}(\xi - \eta_1) e^{-it \vert \eta_1 \vert ^2} \eta_{1,i} \widehat{f}(t,\eta_1) d\eta_1 \\
\label{id2}&-i e^{it \vert \xi \vert^2} \int_{\mathbb{R}^3} \widehat{x_j a_i}(\xi - \eta_1) e^{-it \vert \eta_1 \vert ^2} \eta_{1,i} \widehat{f}(s,\eta_1) d\eta_1 \\
\label{id3}&+2 t \xi_j e^{it \vert \xi \vert^2} \int_{\mathbb{R}^3} \widehat{V}(\xi - \eta_1) e^{-it \vert \eta_1 \vert ^2} \widehat{f}(t,\eta_1) d\eta_1 \\
\label{id4}&-i e^{it \vert \xi \vert^2} \int_{\mathbb{R}^3} \widehat{x_j V}(\xi - \eta_1) e^{-it \vert \eta_1 \vert ^2} \widehat{f}(s,\eta_1) d\eta_1 \\
\label{id5}&+2  \int_{\mathbb{R}^3} t \eta_{1,j} e^{it (\vert \xi \vert^2 - \vert \xi-\eta_1 \vert^2 - \vert \eta_1 \vert^2)} \widehat{f}(t,\eta_1) \widehat{f}(t,\xi- \eta_1) d\eta_1 \\
\label{id6}&-i  \int_{\mathbb{R}^3} e^{it (\vert \xi \vert^2 - \vert \xi-\eta_1 \vert^2 - \vert \eta_1 \vert^2)} \widehat{f}(t,\eta_1) \partial_{\xi_j} \widehat{f}(t,\xi- \eta_1) d\eta_1.
\end{align}
Of all the terms that appear subsequently, the two that we will not be able to estimate directly come from \eqref{id1} and \eqref{id3} will be of the form
\begin{align} \label{model2}
& \int_{1} ^t \int_{\mathbb{R}^3} \frac{\eta_{1,j} \xi_l \alpha_1(\eta_1)}{(\vert \xi \vert ^2 - \vert \eta_1 \vert^2 + i \beta)\vert \eta_1 \vert^2}
\widehat{W_1}(\xi-\eta_1)  P_{k_1}(\eta_1) \\
\notag & \times \int_{\mathbb{R}^3} 2 s \eta_{1,j} \widehat{W_2}(\eta_1 - \eta_2) \alpha_2(\eta_2) e^{is(\vert \xi \vert^2 - \vert \eta_2 \vert^2 + i\beta)} \widehat{f}(s,\eta_2) d\eta_2 d\eta_1 ds \\
\notag & = \int_{1} ^t \int_{\mathbb{R}^3} \frac{ \alpha_1(\eta_1)}{(\vert \xi \vert ^2 - \vert \eta_1 \vert^2 + i \beta)}
e^{is(\vert \xi \vert^2 - \vert \eta_1 \vert^2)}\widehat{W_1}(\xi-\eta_1)  P_{k_1}(\eta_1) \\
\notag & \times \int_{\mathbb{R}^3} 2 s \xi_l  \widehat{W_2}(\eta_1 - \eta_2) e^{is(\vert \xi \vert^2 - \vert \eta_2 \vert^2 + i\beta)} \alpha_2(\eta_2) \widehat{f}(s,\eta_2) d\eta_2 d\eta_1 ds ,
\end{align}
where the simplification here is due to the summation in $j.$ 
\\
We explain in the next section how to deal with these terms.

\subsection{Further expansions}

Note that for expressions of the type \eqref{model2} the inner integral is the same as the terms we started with (that is \eqref{M2} and \eqref{M6}). \\
Therefore we use the exact same strategy:
\\
We start by localizing in the $\eta_2$ variable ($k_2$ denotes the corresponding exponent)
\begin{itemize}
\item If $\vert k-k_2 \vert > 1$ then we integrate by parts in time.
\item If $\vert k-k_2 \vert \leqslant 1$ then we integrate by parts in $\eta_2.$ Then integrate by parts in time in the worse term (that is the term for which the derivative in $\eta_2$ falls on the profile).
\end{itemize}
Then we repeat the procedure iteratively.
\\
\\
\underline{Case 1: $\vert k-k_n \vert >1 $}
\\
At the $n-$th step of the iteration we obtain the following terms:
\begin{align*}
&\mathcal{F} I_1 ^n f := \int_1 ^t \int \prod_{\gamma=1}^{n-1} \frac{\alpha_{\gamma}(\eta_{\gamma})\widehat{W_{\gamma}}(\eta_{\gamma-1}-\eta_{\gamma})P_{k_{\gamma}}(\eta_{\gamma}) P_k(\xi)}{\vert \xi \vert^2 - \vert \eta_{\gamma} \vert^2} d\eta_1 ... d\eta_{n-2} \\
& \times \int_{\eta_n} s \xi_l e^{is(\vert \xi \vert^2 - \vert \eta_n \vert^2 - \vert \eta_{n-1}- \eta_n \vert^2)} \widehat{f}(s,\eta_{n-1}-\eta_n) \widehat{f}(s,\eta_n) d\eta_n d\eta_{n-1} ds,
\end{align*}
where $W$ denotes either $V$ or one of $a_i$'s. The function $\alpha_{\gamma}(\eta)$ is equal to 1 if $W=1$ and $\eta_i$ if $W = a_i.$ 
\\
We also use the convention that $\eta_0 = \xi.$ 
\\
With similar notations, we also have the analog of \eqref{B3}: 
\begin{align*}
&\mathcal{F} I_2 ^n f := \int \prod_{\gamma=1}^{n-1} \frac{\alpha_{\gamma}(\eta_{\gamma})\widehat{W_{\gamma}}(\eta_{\gamma-1}-\eta_{\gamma})P_{k_{\gamma}}(\eta_{\gamma}) P_k(\xi)}{\vert \xi \vert^2 - \vert \eta_{\gamma} \vert^2} d\eta_1 ... d\eta_{n-1}  e^{it(\vert \xi \vert^2-\vert \eta_n \vert^2)} t \xi_l \widehat{f_{k_n}}(t,\eta_n) d\eta_n
\end{align*}
and similarly:
\begin{align*}
&\mathcal{F} I_3 ^n f := \int \prod_{\gamma=1}^{n-1} \frac{\alpha_{\gamma}(\eta_{\gamma})\widehat{W_{\gamma}}(\eta_{\gamma-1}-\eta_{\gamma})P_{k_{\gamma}}(\eta_{\gamma}) P_k(\xi)}{\vert \xi \vert^2 - \vert \eta_{\gamma} \vert^2} d\eta_1 ... d\eta_{n-1} \\
& \times \int_1 ^t \xi_l e^{is(\vert \xi \vert^2-\vert \eta_n \vert^2)} \frac{\widehat{W_n}(\eta_{n-1}-\eta_n) \alpha_n (\eta_n)}{\vert \xi \vert^2 - \vert \eta_n \vert^2} \widehat{f_{k_n}}(t,\eta_n) d\eta_n ds.
\end{align*}
\underline{Case 2: $\vert k-k_n \vert \leqslant 1$} \\
In this case we get more terms. More precisely we have
\begin{align*}
&\mathcal{F} I_4 ^n f := \int \prod_{\gamma=1}^{n-1} \frac{\alpha_{\gamma}(\eta_{\gamma})\widehat{W_{\gamma}}(\eta_{\gamma-1}-\eta_{\gamma})P_{k_{\gamma}}(\eta_{\gamma}) P_k(\xi)}{\vert \xi \vert^2 - \vert \eta_{\gamma} \vert^2} d\eta_1 ... d\eta_{n-1} \\
& \times e^{it(\vert \xi \vert^2-\vert \eta_n \vert^2)}  \frac{\xi_l \eta_{n,j}}{\vert \eta_n \vert^2} \partial_{\eta_n,j} \widehat{f}(t,\eta_n) P_{k_n}(\eta_n) d\eta_n.
\end{align*}
There are the easier terms of the same type: 
\begin{align*}
&\mathcal{F} I_5 ^n f := \int \prod_{\gamma=1}^{n-1} \frac{\alpha_{\gamma}(\eta_{\gamma})\widehat{W_{\gamma}}(\eta_{\gamma-1}-\eta_{\gamma})P_{k_{\gamma}}(\eta_{\gamma}) P_k(\xi)}{\vert \xi \vert^2 - \vert \eta_{\gamma} \vert^2} d\eta_1 ... d\eta_{n-2} \\
& \times \int_1 ^t \int_{\mathbb{R}^3}  e^{is(\vert \xi \vert^2-\vert \eta_n \vert^2)} \alpha_n(\eta_n) \partial_{\eta_{n,j}} \widehat{W_n}(\eta_{n-1}-\eta_n) \frac{\xi_l \eta_{n,j}}{\vert \eta_n \vert^2} \widehat{f_{k_n}}(s,\eta_n) d\eta_n ds d\eta_{n-1}.
\end{align*}
We also get an somewhat similar term
\begin{align*}
&\mathcal{F} I_6 ^n f := \int \prod_{\gamma=1}^{n-1} \frac{\alpha_{\gamma}(\eta_{\gamma})\widehat{W_{\gamma}}(\eta_{\gamma-1}-\eta_{\gamma})P_{k_{\gamma}}(\eta_{\gamma}) P_k(\xi)}{\vert \xi \vert^2 - \vert \eta_{\gamma} \vert^2} d\eta_1 ... d\eta_{n-2} \\
& \times \int_1 ^t \int_{\eta_n}  e^{is(\vert \xi \vert^2-\vert \eta_n \vert^2)} \alpha_{n}(\eta_n) \widehat{W_n}(\eta_{n-1}-\eta_n) \partial_{\eta_n,j} \big( \frac{\xi_l  \eta_{n,j}}{\vert \eta_n \vert^2} \big) \widehat{f_{k_n}}(s,\eta_n) d\eta_n ds d\eta_{n-1},
\end{align*}
and also
\begin{align*}
&\mathcal{F} I_7 ^n f :=  \int \prod_{\gamma=1}^{n-1} \frac{\alpha_{\gamma}(\eta_{\gamma})\widehat{W_{\gamma}}(\eta_{\gamma-1}-\eta_{\gamma})P_{k_{\gamma}}(\eta_{\gamma}) P_k(\xi)}{\vert \xi \vert^2 - \vert \eta_{\gamma} \vert^2} d\eta_1 ... d\eta_{n-2} \\
& \times \int_1 ^t  \int_{\eta_n}  e^{is(\vert \xi \vert^2 - \vert \eta_n \vert^2} \frac{\alpha_{n}(\eta_n)\eta_{n,j}\xi_l}{\vert \eta_n \vert^2} 1.1^{-k_1} \phi'(1.1^{-k_1}\eta_n) \frac{\eta_{n,j}}{\vert \eta_n \vert} \widehat{W_n}(\eta_{n-1}-\eta_n) \widehat{f_{k_n}}(s,\eta_n) d\eta_n d\eta_{n-1} ds .
\end{align*}
Finally we get the following terms coming from \eqref{identder}
\begin{align*}
&\mathcal{F} I_8 ^n f :=  \int \prod_{\gamma=1}^{n-1} \frac{\alpha_{\gamma}(\eta_{\gamma})\widehat{W_{\gamma}}(\eta_{\gamma-1}-\eta_{\gamma})P_{k_{\gamma}}(\eta_{\gamma}) P_k(\xi)}{\vert \xi \vert^2 - \vert \eta_{\gamma} \vert^2} d\eta_1 ... d\eta_{n-2} \frac{\xi_l \eta_{n-1,j}}{\vert \eta_{n-1} \vert^2} \\
& \times \int_1 ^t  \int_{\eta_n}  e^{is(\vert \xi \vert^2 - \vert \eta_n \vert^2} \alpha_{n}(\eta_n) \widehat{x_{n,j} W_n}(\eta_{n-1}-\eta_n) \widehat{f}(s,\eta_n) d\eta_n d\eta_{n-1} ds .
\end{align*}
More importantly we also have the bilinear terms
\begin{align*}
&\mathcal{F} I_9 ^n f :=  \int \prod_{\gamma=1}^{n-1} \frac{\alpha_{\gamma}(\eta_{\gamma})\widehat{W_{\gamma}}(\eta_{\gamma-1}-\eta_{\gamma})P_{k_{\gamma}}(\eta_{\gamma}) P_k(\xi)}{\vert \xi \vert^2 - \vert \eta_{\gamma} \vert^2} d\eta_1 ... d\eta_{n-2} \frac{\xi_l \eta_{n-1,j}}{\vert \eta_{n-1} \vert^2} \\
& \times \int_1 ^t s \int_{\eta_n} \eta_{n,j}  e^{is(\vert \xi \vert^2 - \vert \eta_n \vert^2 - \vert \eta_{n-1}- \eta_n \vert^2)}  \widehat{f}(\eta_{n-1}-\eta_n)  \widehat{f}(s,\eta_n) d\eta_n d\eta_{n-1} ds
\end{align*}
and
\begin{align*}
&\mathcal{F} I_{10} ^n f :=  \int \prod_{\gamma=1}^{n-1} \frac{\alpha_{\gamma}(\eta_{\gamma})\widehat{W_{\gamma}}(\eta_{\gamma-1}-\eta_{\gamma})P_{k_{\gamma}}(\eta_{\gamma}) P_k(\xi)}{\vert \xi \vert^2 - \vert \eta_{\gamma} \vert^2} d\eta_1 ... d\eta_{n-2} \frac{\xi_l \eta_{n-1,j}}{\vert \eta_{n-1} \vert^2} \\
& \times \int_1 ^t \int_{\eta_n}   e^{is(\vert \xi \vert^2 - \vert \eta_n \vert^2 - \vert \eta_{n-1}- \eta_n \vert^2)} \partial_{n,j} \widehat{f}(\eta_{n-1}-\eta_n) \widehat{f}(s,\eta_n) d\eta_n d\eta_{n-1} ds.
\end{align*}
Heuristically, one has the following correspondance: $I^{n+1}_1 f$ is the analog of \eqref{B2bis}, $I_2^n f$ is the analog of both \eqref{B3} and \eqref{B4}, $I_3 ^n f$ is the analog of \eqref{B5}, $I_4 ^n f$ is the analog of both \eqref{restediff} and \eqref{restefac}, $I_5 ^n f$ is the analog of \eqref{c22}, $I_6 ^n f$ is the analog of \eqref{c23}, $I_7 ^n f$ is the analog of \eqref{c24}, $I_8^{n+1} f$ is the analog of both \eqref{id2} and \eqref{id4}, $I_9^{n+1} f$ is the analog of \eqref{id5} and $I_{10}^{n+1} f$ is the analog of \eqref{id6}.
\\
\\
We will prove the following estimates in Section \ref{Boundingiterates}:
\begin{proposition} \label{nthestimate}
We have the bound
\begin{align*}
\sum_{k_1,...,k_n} \Vert I_j ^n f \Vert_{L^2_x} \lesssim C^{n} \delta^{n} \varepsilon_1 \ \ \ \ \ j \in \lbrace 1; ... ;10 \rbrace .
\end{align*}
The implicit constant in the above inequality does not depend on $n. $
\end{proposition}
We now explain how this implies the desired bound \eqref{goal1}.
\\
\begin{proof}[Proof of \eqref{goal1}]
From the iteration procedure explained above, we arrive at the following expression for $\partial_{\xi_l} \widehat{f}(t,\xi)$ (where we set all the numerical constants such as $i, 2 \pi$ equal to one for better legibility, since they do not matter in the estimates):
\begin{align*}
\partial_{\xi_l} \widehat{f}(t,\xi) & = \partial_{\xi_l} \big( e^{i \vert \xi \vert^2} \widehat{u_{1,k}} (\xi)
 \big) + \eqref{M1} + \eqref{M3} + \eqref{M4} + \eqref{M5}+ \eqref{M7} + \eqref{M8} \\
& +\mathcal{F} \sum_{n=1}^{\infty} \sum_{\gamma=1} ^{n+1} \sum_{W_{\gamma} \in \lbrace a_1, a_2,a_3, V \rbrace} \sum_{j=1}^3 \Bigg( \sum_{k_1,...,k_n}  I^{n+1}_1 f + I_2 ^n f + \widetilde{I_2}^n f + I_3 ^n f + I_4 ^n f + \widetilde{I_4}^n f \\ & +I_5 ^n f + I_6 ^n f + I_7 ^n f + I_8 ^{n+1} f + I_9 ^{n+1} f + I_{10}^{n+1} f \Bigg),
\end{align*}
where $\widetilde{I_2}^n f, \widetilde{I_4}^n f$ denote the same expressions as $I_2^n f, I_4^n f$ but with $t=1.$ \\
Since at each step of the iteration there are $O(4^n)$ terms that appear (that is the three middle sums above contribute $O(4^n)$ terms), there exists some large constant $D$ such that (using \eqref{estimationX} and Proposition \ref{nthestimate})
\begin{align*}
\Vert f \Vert_{X} \leqslant \varepsilon_0 + D \sum_{n=1}^{+\infty} 4^n C^n \delta^n \varepsilon_1 \leqslant \frac{\varepsilon_1}{2},
\end{align*}
provided $\delta$ is small enough.
\end{proof}

\section{Bounding the terms from the first expansion} \label{firsterm}
In this section we bound the terms from the first expansion (see the various estimates announced in Section \ref{firstexpansion}). \\ 
We distinguish in the first subsection the estimates that are done without the use of smoothing estimates, from the ones that require recovering derivatives (terms of potential type) in the second subsection.

\subsection{Easier terms}
We start with terms that appear directly after taking a derivative in $\xi_l$ (so that do not strictly speaking arise from the iteration procedure)
\begin{lemma}
We have the bounds
\begin{align*}
\Vert \eqref{M1}, \eqref{M4} \Vert_{L^2_x} \lesssim \delta \varepsilon_1.
\end{align*}
\end{lemma}
\begin{proof}
We start with \eqref{M1}. We use Strichartz estimates, Lemma \ref{bilin} and we have if $k_1<0$
\begin{align*}
\Vert \eqref{M1} \Vert_{L^2_x} & \lesssim  \Bigg \Vert \mathcal{F}^{-1} \int_{\eta_1} \eta_{1,i} e^{-is \vert \eta_1 \vert^2} \widehat{f_{k_1}}(s, \eta_1) \partial_{\xi_l} \widehat{a_i}(\xi-\eta_1) d\eta_1 \Bigg \Vert_{L^{2}_t L^{6/5}_x} \\
& \lesssim 1.1^{k_1} \Vert e^{it\Delta} f_{k_1} \Vert_{L^{2}_t L^{6}_x}  \Vert x_l a_i \Vert_{L^{3/2}_x},
\end{align*}
which can be summed using Lemma \ref{dispersive}.
\\
When $k_1 \geqslant 0$ we use Lemma \ref{summation} with $c=1/4$ to obtain
\begin{align*}
\Vert \eqref{M1} \Vert_{L^2} & \lesssim 1.1^{k_1} \Vert e^{it\Delta} f_{k_1} \Vert_{L^{2}_t L^{6}_x}  \Vert x_l a_i \Vert_{L^{3/2}_x} \\ 
& \lesssim 1.1^{-k_1} \delta \varepsilon_1.
\end{align*}
In the case where $W=V,$ we use Strichartz estimates to obtain
\begin{align*}
& \Bigg \Vert P_k (\xi) \int_1 ^t \int_{\mathbb{R}^3} e^{is (\vert \xi \vert^2 - \vert \eta_1 \vert^2)} \widehat{f}_{k_1}(s,\eta_1) \partial_{\xi_l} \widehat{V}(s,\xi-\eta_1) d\eta_1 ds \Bigg \Vert_{L^{\infty}_t L^2 _x} \\
& \lesssim  \Vert e^{it \Delta} f_{k_1} (xV) \Vert_{L^{2}_t L^{6/5}_x} \\
& \lesssim  \Vert e^{it \Delta} f_{k_1} \Vert_{L^{2}_t L^6 _x} \Vert x V \Vert_{L^{3/2} _x}, \\
\end{align*}
and then we can sum over $k_1$ to deduce the result thanks to Lemmas \ref{dispersive} and \ref{summation}.
\end{proof}
Now we prove the following:
\begin{lemma}
We have the bounds
\begin{align*}
\Vert \eqref{M3},\eqref{M5} \Vert_{L^2_x} \lesssim \delta \varepsilon_1 .
\end{align*}
\end{lemma}
\begin{proof}
Here we essentially reproduce the proof from \cite{L}, Lemma 5.2. Since the proofs for the magnetic part and the potential part are similar, we do them simultaneously. \\
We split the proof into several cases: \\
\underline{Case 1: $k>0$} \\
In this case we use Strichartz estimates to write that
\begin{align*}
& \Bigg \Vert  1.1^{-k} \phi'(1.1^{-k} \xi) \frac{\xi_l}{\vert \xi \vert} \int_{1}^t \int_{\mathbb{R}^3}  e^{is (\vert \xi \vert^2 -\vert \eta_1 \vert^2)} \widehat{f_{k_1}}(s,\eta_1) \widehat{V}(\xi-\eta_1) d\eta_1 ds \Bigg \Vert_{L^{\infty}_t L^2 _x} \\
& \lesssim \Bigg \Vert \mathcal{F}^{-1} \int_1 ^t \int_{\mathbb{R}^3} e^{is (\vert \xi \vert^2 - \vert \eta_1 \vert^2)} \widehat{f}_{k_1}(s,\eta_1)  \widehat{V}(\xi-\eta_1) d\eta_1 ds \Bigg \Vert_{L^{\infty}_t L^2 _x} \\
& \lesssim \Vert \big( e^{it \Delta} f_{k_1} \big)V \Vert_{L^{2}_t L^{6/5}_x} \\
& \lesssim \Vert e^{it \Delta} f_{k_1} \Vert_{L^{2}_t L^{6}_x} \Vert V \Vert_{ L^{3/2}_x} \\
& \lesssim \Vert e^{it \Delta} f_{k_1} \Vert_{L^{2}_t L^{6}_x} \delta ,
\end{align*}
and we can sum over $k_1$ using Lemma \ref{summation} and \ref{dispersive}.
\\
In the magnetic case, there is an extra $1.1^{k_1}$ term. To deal with it when $k_1>0,$ we use Lemma \ref{summation} as in the proof of the previous lemma.\\
\\
\underline{Case 2: $k \leqslant 0$}
\\
We distinguish three subcases:\\
\underline{Case 2.1: $k>k_1+1$} \\
In this case the frequency $\vert \xi- \eta_1 \vert$ is localized at $1.1^{k},$ therefore we can write, using Strichartz estimates and Bernstein's inequality:
\begin{align*}
& \Bigg \Vert  1.1^{-k} \phi'(1.1^{-k} \xi) \frac{\xi_l}{\vert \xi \vert} \int_{1}^t \int_{\mathbb{R}^3}  e^{is (\vert \xi \vert^2 -\vert \eta_1 \vert^2)} \widehat{f_{k_1}}(s,\eta_1) \widehat{V}(\xi-\eta_1) d\eta_1 ds \Bigg \Vert_{L^{\infty}_t L^2 _x} \\
& \lesssim 1.1^{-k} \Bigg \Vert \mathcal{F}^{-1} \int_{\mathbb{R}^3}  e^{is (\vert \xi \vert^2 -\vert \eta_1 \vert^2)} \widehat{f_{k_1}}(s,\eta_1) \widehat{V_{k}}(\xi-\eta_1) d\eta_1 ds \Bigg \Vert_{L^{2}_t L^{6/5}_x} \\
& \lesssim 1.1^{-k} \Vert V_{k} \Vert_{ L^{3/2} _x}  \Vert e^{it\Delta} f_{k_1} \Vert_{L^{2} _t L^{6}_x} \\
& \lesssim \Vert V_k \Vert_{L^{1}} \Vert e^{it\Delta} f_{k_1} \Vert_{L^{2} _t L^{6}_x},
\end{align*}
which can be summed using Lemmas \ref{summation} and \ref{dispersive}.
\\
In the magnetic case, the proof is simpler. We directly obtain the bound 
\begin{align*}
\Vert \eqref{M3} \Vert_{L^2_x} \lesssim 1.1^{k_1-k} \Vert {a}_{i,k} \Vert_{ L^{3/2} _x}  \Vert e^{it\Delta} f_{k_1} \Vert_{L^{2}_t L^{6}_x},
\end{align*}
which can be summed directly since $k_1>k.$
\\
\\
\underline{Case 2.2: $\vert k-k_1 \vert \leqslant 1$}\\
Then we split the $\xi-\eta_1$ frequency dyadically and denote $k_2$ the corresponding exponent. \\
Note that $\vert \xi-\eta_1 \vert \leqslant \vert \xi \vert + \vert \eta_1 \vert \leqslant 1.1^{k+1} + 1.1^{k+2} \leqslant 1.1^{k+10}.$ \\
As a result $k_2 \leqslant k +10.$ \\
Now we can write, using Strichartz estimates, Bernstein's inequality:
\begin{align*}
 & \Bigg \Vert  1.1^{-k} \phi'(1.1^{-k} \xi) \frac{\xi_l}{\vert \xi \vert} \int_{1}^t \int_{\mathbb{R}^3}  e^{is (\vert \xi \vert^2 -\vert \eta_1 \vert^2)} \widehat{f_{k_1}}(s,\eta_1) \widehat{V_{k_2}}(\xi-\eta_1) d\eta_1 ds \Bigg \Vert_{L^{\infty}_t L^2 _x} \\
 & \lesssim 1.1^{-k} \Vert V_{k_2} \Vert_{L^{3/2}_x} \Vert e^{it\Delta} f_{k_1} \Vert_{L^{2}_t L^{6}_x}  \\
 & \lesssim 1.1^{k_2-k} \Vert V \Vert_{ L^{1} _x} \Vert e^{it\Delta} f_{k_1} \Vert_{L^{2}_t L^{6}_x}.
\end{align*}
Now since $k_2 \leqslant k + 10$ the factor in front allows us to sum over $k_2.$ The result follows.
\\
The magnetic case is simpler. There is no need for the additional localization since
\begin{align*}
\Vert \eqref{M5} \Vert_{L^2 _x}  \lesssim 1.1^{k_1-k} \Vert a_i \Vert_{L^{3/2} _x} \Vert e^{it\Delta} f_{k_1} \Vert_{L^{2}_t L^{6}_x} .
\end{align*}
\underline{Case 2.3: $k_1 > k+1$}\\
In this case we split the time variable dyadically. Let's denote $m$ the corresponding exponent. \\
We must estimate
\begin{align*}
I_{m,k_1,k} :=  1.1^{-k} \phi'(1.1^{-k} \xi) \frac{\xi_l}{\vert \xi \vert} \int_{1.1^m}^{1.1^{m+1}} \int_{\mathbb{R}^3}  e^{is (\vert \xi \vert^2 -\vert \eta_1 \vert^2)} \widehat{f_{k_1}}(s,\eta_1) \widehat{V_{k_1}}(\xi-\eta_1) d\eta_1 ds ,
\end{align*}
where the extra localization can be placed on $V$ since $\xi - \eta_1$ has magnitude roughly $1.1^{k_1}$ given that $k_1>k+1.$ \\
\\
\underline{Subcase 2.3.1: $k \leqslant -\epsilon m$ ($\epsilon$ some small number)} \\
Then we write, using Bernstein's inequality, H\"{o}lder's inequality and Lemma \ref{dispersive}, that
\begin{align*}
\Vert I_{m,k_1,k} \Vert_{L^{\infty}_t L^2 _x} & \lesssim 1.1^{-k} 1.1^{m} \sup_{t \simeq 1.1^m} \Vert \big( e^{it \Delta} f_{k_1} V_{k_1} \big)_k \Vert_{L^2 _x} \\
& \lesssim 1.1^{m} 1.1^{\epsilon k} \sup_{t \simeq 1.1^m} \Vert e^{it \Delta} f_{k_1} V_{k_1}(t)  \Vert_{L^{6/5-}_x} \\
                                              & \lesssim 1.1^{m} 1.1^{\varepsilon k} \sup_{t \simeq 1.1^m} \Vert e^{it \Delta} f_{k_1} \Vert_{L^6 _x} \Vert V_{k_1}  \Vert_{L^{3/2-}_x} \\
& \lesssim 1.1^m 1.1^{-m} 1.1^{\epsilon k} \varepsilon_1 \Vert V_{k_1} \Vert_{L^{\infty}_t L^{3/2-}_x} \\
& \lesssim 1.1^{-m \epsilon} \varepsilon_1 \Vert V_{k_1} \Vert_{L^{\infty}_t L^{3/2-}_x}.
\end{align*}
We can sum over $k_1$ and $m$ given the factors that appear.
\\
\\
In the magnetic case we get the same bound with $V$ replaced by $\nabla a. $
\\
\\
\underline{Subcase 2.3.2: $-\epsilon m < k \leqslant 0$} \\
In this case we use Strichartz estimates as well as Lemma \ref{dispersive}:
\begin{align*}
\Vert I_{m,k_1,k} \Vert_{L^{\infty}_t L^2 _x} & \lesssim 1.1^{\epsilon m} \Bigg \Vert \int_{1.1^m}  ^{1.1^{m+1}} e^{-is \Delta} \bigg( \big( e^{is \Delta} f_{k_1} \big) V_{k_1} \bigg) ds \Bigg \Vert_{L^2 _x} \\
& \lesssim 1.1^{\epsilon m} \Vert \textbf{1}_{t \simeq 1.1^m} \big( e^{it \Delta} f_{k_1} \big) V_{k_1} \Vert_{L^{2}_t L^{6/5}_x} \\
& \lesssim 1.1^{\epsilon m} \Vert \textbf{1}_{t \simeq 1.1^m} \big( e^{it \Delta} f_{k_1} \big) \Vert_{L^{2}_t L^{6} _x} \Vert V_{k_1} \Vert_{L^{3/2}_x} \\
& \lesssim 1.1^{\epsilon m} 1.1^{-m/2} \varepsilon_1 \Vert V_{k_1} \Vert_{L^{3/2}_x}.
\end{align*}
Now notice that we can sum that bound over $m$ and $k_1$.
\\
In the magnetic case we get the same bound with $V$ replaced by $\nabla a. $
\end{proof}
We now come to the terms that appear in the iterative procedure in the case where $\vert k-k_1 \vert>1.$
\begin{lemma}
We have the bounds
\begin{align*}
\sum_{\vert k-k_1 \vert>1} \Vert \eqref{B3}, \eqref{B4}, \eqref{B5} \Vert_{L^2 _x} &\lesssim  \delta \varepsilon_1 \\
\sum_{\vert k-k_1 \vert>1} \Vert \eqref{B2bis} \Vert \lesssim \delta \varepsilon_1 ^2.
\end{align*}
\end{lemma}
\begin{proof}
Note that the bound on \eqref{B3} implies the one on \eqref{B4} (take $t=1$). Therefore we only prove the first bound. \\
First assume that $k_1 > k+1.$ \\
We use Lemma \ref{bilin}, the dispersive estimate from Lemma \ref{dispersive} and Bernstein's inequality
\begin{align*}
\Vert \eqref{B3} \Vert_{L^2 _x} & \lesssim  1.1^{k} 1.1^{k_1} 1.1^{-2k_1} \Vert (a_i)_{k_1} \Vert_{L^3_x} \Vert t e^{it \Delta} f_{k_1} \Vert_{L^6_x} \\
& \lesssim 1.1^{k-k_1} \delta \varepsilon_1.
\end{align*}
This last expression can be summed over $k_1$ given the condition $k_1>k+1.$ \\
The reasoning is similar if $k > k_1 + 1.$ We obtain the inequality
\begin{align*}
\Vert \eqref{B3} \Vert_{L^2 _x} & \lesssim  1.1^{k} 1.1^{k_1} 1.1^{-2k} \Vert (a_i)_{k} \Vert_{L^3_x} \Vert t e^{it \Delta} f_{k_1} \Vert_{L^6_x} ,
\end{align*}
and we use Lemmas \ref{summation} and \ref{dispersive} to sum over $k_1$ in that last inequality.
\\
For the last term we use Strichartz estimates we get when $k>k_1+1$
\begin{align} 
\notag  \Vert \eqref{B5} \Vert_{L^2_x} & \lesssim \Bigg \Vert \mathcal{F}^{-1} \int_{\mathbb{R}^3} \frac{\xi_l \alpha_{1}(\eta_1)}{\vert \xi \vert^2 - \vert \eta_1 \vert^2} \widehat{f_{k_1}}(s,\eta_1) \widehat{W}(\xi-\eta_1) d\eta_1 \Bigg \Vert_{L^{2}_t L^{6/5}_x} \\
\label{aux} & \lesssim 1.1^{-k}  \Vert e^{it\Delta} f_{k_1} \Vert_{L^{2} _t L^{6}_x} \Vert \mathcal{F}^{-1} \big( \alpha(\eta_1) P_{k_1}(\eta_1) \big) \Vert_{L^1} \Vert W_{k} \Vert_{L^{3/2}_x}.
\end{align}
If $W = a_i$ then the bound above reads
\begin{align*}
\Vert \eqref{aux} \Vert_{L^2_x} & \lesssim 1.1^{k_1-k} \Vert e^{it\Delta} f_{k_1} \Vert_{L^{2} _t L^{6}_x} \Vert a_i  \Vert_{L^{3/2}_x}
\end{align*}
and this can be summed when $k_1 + 1<k.$
\\
\\
If $W = V$ then using Bernstein's inequality we obtain
\begin{align*}
\Vert \eqref{aux} \Vert_{L^2_x} & \lesssim 1.1^{-k} \Vert e^{it\Delta} f_{k_1} \Vert_{L^{2} _t L^{6}_x} \Vert V_{k} \Vert_{L^{3/2}_x} \\
& \lesssim \Vert e^{it\Delta} f_{k_1} \Vert_{L^{2} _t L^{6}_x} \Vert V_{k} \Vert_{L^{1}_x}.
\end{align*}
Now we consider the case $k_1 >k +1:$
\\
Using Strichartz estimates, Lemma \ref{bilin} and Bernstein's inequality as above yields 
\begin{align*}
\Vert \eqref{B5} \Vert_{L^2_x} & \lesssim 1.1^{k} 1.1^{-2k_1} \Vert e^{it\Delta} f_{k_1} \Vert_{L^{2} _t L^{6}_x} \Vert \mathcal{F}^{-1} \big( \alpha(\eta_1) P_{k_1}(\eta_1) \big) \Vert_{L^1} \Vert W_{k_1} \Vert_{L^{3/2}_x}
\end{align*}
If $W = a_i$ then the bound above reads
\begin{align*}
\Vert \eqref{aux} \Vert_{L^2_x} & \lesssim 1.1^{k-k_1} \Vert e^{it\Delta} f_{k_1} \Vert_{L^{2} _t L^{6}_x} \Vert a_i  \Vert_{L^{3/2}_x},
\end{align*}
and this can be summed when $k_1>k+1.$\\
\\
If $W=V$ then using Bernstein's inequality we obtain
\begin{align*}
\Vert \eqref{aux} \Vert_{L^2_x} & \lesssim 1.1^k 1.1^{-2k_1} \Vert e^{it\Delta} f_{k_1} \Vert_{L^{2} _t L^{6}_x} \Vert V_{k_1} \Vert_{L^{3/2}_x} \\
& \lesssim 1.1^{k-k_1} \Vert e^{it\Delta} f_{k_1} \Vert_{L^{2} _t L^{6}_x} \Vert V_{k_1} \Vert_{L^{1}_x},
\end{align*}
which can be summed.
\\
\\
Now we prove the bound on \eqref{B2bis}. The reasoning is similar, therefore we only sketch it here. \\
We treat the case $k_1>k+1,$ the other case being similar. \\
Using Strichartz estimates and a standard bilinear estimate, we obtain if $W=a_i:$
\begin{align*}
\Vert \eqref{B2bis} \Vert_{L^2_x} & \lesssim 1.1^k 1.1^{-2k_1} 1.1^{k_1} \Vert a_i \Vert_{L^2_x} \Vert tu \Vert_{L^{\infty}_t L^6_x} \Vert u \Vert_{L^2_t L^6 _x} \\
                                  & \lesssim 1.1^{k-k_1} \delta \varepsilon_1^2.
\end{align*}
This bound can be summed. \\
In the case where $W = V,$ we write an extra line using Bernstein's inequality:
\begin{align*}
\Vert \eqref{B2bis} \Vert_{L^2_x} & \lesssim 1.1^k 1.1^{-2k_1} \Vert V_{k_1} \Vert_{L^2_x} \Vert tu \Vert_{L^{\infty}_t L^6_x} \Vert u \Vert_{L^2_t L^6 _x} \\
                                  & \lesssim 1.1^{k-k_1} \Vert V_{k_1} \Vert_{L^{6/5}_x} \Vert tu \Vert_{L^{\infty}_t L^6_x} \Vert u \Vert_{L^2_t L^6 _x} \\
                                  & \lesssim 1.1^{k-k_1} \delta \varepsilon_1^2.
\end{align*}
\end{proof}
Now we come to the terms that appear in the iteration procedure in the case $\vert k-k_1 \vert \leqslant 1:$
\begin{lemma}
We have the estimate
\begin{align*}
\Vert \eqref{c22} \Vert_{L^2_x} & \lesssim \delta \varepsilon_1 \\
\Vert \eqref{c23} \Vert_{L^2_x} & \lesssim \delta \varepsilon_1 \\
\Vert \eqref{c24} \Vert_{L^2_x} & \lesssim \delta \varepsilon_1.
\end{align*}
\end{lemma}
\begin{proof}
In the case where $W_1=V$ these estimates have been proved in \cite{L}, Lemma 5.6. Therefore we only give the proof in the case where $W_1=a_i$ here. 
\\
We use Strichartz estimates, the bilinear Lemma \ref{bilin} and Lemma \ref{summation} (with $c=1/4$) to write that
\begin{align*}
\Vert \eqref{c22} \Vert & \lesssim \Bigg \Vert \mathcal{F}^{-1}_{\xi} \int_{\eta_1} \frac{\xi_l \eta_{1,j}\eta_{i}}{\vert \eta_1 \vert^2} \partial_{\eta_{1,j}} \widehat{a_i} (\xi - \eta_1) e^{-is \vert \eta_1 \vert^2} \widehat{f_{k_1}}(t,\eta_1)  d\eta_1 \Bigg \Vert_{L^{2}_t L^{6/5}_x} \\
& \lesssim 1.1^{k} \Vert e^{it \Delta} f_{k_1} \Vert_{L^{2}_t L^{6}_x} \Vert x_j a_i(x) \Vert_{L^{3/2}_x} \\
& \lesssim 1.1^{-k_1} \delta \varepsilon_1,
\end{align*}
which is good enough if $k_1>0.$ \\
Otherwise if $k_1 \leqslant 0$ we write, using that $\vert k-k_1 \vert \leqslant 1,$
\begin{align*}
\Vert \eqref{c22} \Vert & \lesssim \Bigg \Vert \mathcal{F}^{-1}_{\xi} \int_{\eta_1} \frac{\xi_l \eta_{1,i}\eta_{1,j}}{\vert \eta_1 \vert^2} \partial_{\eta_{1,j}} \widehat{a_i} (\xi - \eta_1) e^{-is \vert \eta_1 \vert^2} \widehat{f_{k_1}}(t,\eta_1)  d\eta_1 \Bigg \Vert_{L^{2}_t L^{6/5}_x} \\
& \lesssim 1.1^{k} \Vert e^{it \Delta} f_{k_1} \Vert_{L^{2}_t L^{6}_x} \Vert x_j a_i(x) \Vert_{L^{3/2}_x} .
\end{align*}
The proofs for the terms \eqref{c23}, \eqref{c24} are the similar, therefore details are omitted.
\end{proof}
Finally we recall two bounds on the bilinear terms \eqref{M7}, \eqref{M8} that were proved in \cite{L}, Lemmas 5.14, 5.15 and 5.16.
\begin{lemma}
We have the estimates
\begin{align*}
\Vert \eqref{M7},\eqref{M8} \Vert_{L^2} \lesssim \varepsilon_1 ^2.
\end{align*}

\end{lemma}
\subsection{Potential terms}
There remains to estimate \eqref{restediff}. Note that the estimate on \eqref{restefac} follows directly from this one by taking $t=1.$ \\
Since in the case where the potential is magnetic there is a derivative loss to deal with, we must use smoothing estimates. \\
The higher order iterates might involve both types of potentials, therefore we need a unified way to deal with such terms. Hence we also give a proof in the case where $W=V$, although the bound has already been established in \cite{L}. That is the content of the following lemma.

\begin{lemma}
We have
\begin{align*}
\Vert \eqref{restediff} \Vert_{L^2} \lesssim \delta \varepsilon_1
\end{align*}
when $W = V.$ 
\end{lemma}
\begin{proof}
We use the following identity:
\begin{align}\label{identit}
\frac{1}{\vert \xi \vert^2 - \vert \eta \vert^2 + i \beta} =(-i) \int_0 ^{\infty} e^{i \tau_1 ( \vert \xi \vert^2 - \vert \eta \vert^2 + i \beta) } d\tau_1
\end{align}
and plug it back into \eqref{restediff}. \\
We bound the outcome using Strichartz estimates, bilinear estimates from Lemma \ref{bilin}:
\begin{align*}
& \Bigg \Vert P_k(\xi) \int_{0}^{\infty} e^{i \tau_1 \vert \xi \vert^2} \int_{\mathbb{R}^3} \frac{\eta_{1,j}\xi_l}{\vert \eta_1 \vert^2} e^{-i\tau_1  \vert \eta_1 \vert^2} e^{-it\vert \eta_1 \vert^2} e^{-\beta \tau} \widehat{V}(\xi-\eta_1) \partial_{\eta_{1,j}} \widehat{f}(t,\eta_1) P_{k_1}(\eta_1) d\eta_1 d\tau_1  \Bigg \Vert_{L^2_x} \\
& \lesssim 1.1^{k} \Bigg \Vert e^{-\beta \tau_1} \mathcal{F}^{-1} \int_{\mathbb{R}^3} \frac{\eta_{1,j}}{\vert \eta_1 \vert^2} \widehat{V}(\xi-\eta_1)  e^{-i (t+\tau) \vert \eta_1 \vert^2} \big( \partial_{\eta_{1,j}} \widehat{f}(t,\eta_1) \big) P_{k_1}(\eta_1) d\eta_1 \Bigg \Vert_{L^{2}_{\tau_1} L^{6/5}_x} \\
& \lesssim 1.1^k \Bigg \Vert \mathcal{F}^{-1} \int_{\mathbb{R}^3} \frac{\eta_{1,j}}{\vert \eta_1 \vert^2} \widehat{V}(\xi-\eta_1)  e^{-i (t+\tau_1) \vert \eta_1 \vert^2} \big( \partial_{\eta_{1,j}} \widehat{f}(t,\eta_1) \big) P_{k_1}(\eta_1) d\eta_1 \Bigg \Vert_{L^{2}_{\tau_1} L^{6/5}_x} \\
& \lesssim  \Vert V \Vert_{L^{3/2}_x} \bigg \Vert e^{i\tau_1 \Delta} \mathcal{F}^{-1} \big(e^{-it\vert \eta_1 \vert^2} \partial_{\eta_{1,j}} \widehat{f}(t,\eta_1) P_{k_1}(\eta_1) \big) \bigg \Vert_{L^{2}_{\tau_1} L^{6}_x}. 
\end{align*}
We can conclude with Strichartz estimates from Lemma \ref{Strichartz}, Lemma \ref{X'} and the fact that the Schr\"{o}dinger semi-group is an isometry on $L^2 _x$:
\begin{align*}
\bigg \Vert e^{i\tau_1 \Delta} \mathcal{F}^{-1} \big(e^{-it \vert \eta_1 \vert^2} \partial_{\eta_{1,j}} \widehat{f}(t,\eta_1) P_{k_1}(\eta_1) \big) \bigg \Vert_{L^{2}_{\tau_1} L^{6}_x} & \lesssim \bigg \Vert \mathcal{F}^{-1} \big(e^{-it \vert \eta_1 \vert^2} \partial_{\eta_{1,j}} \widehat{f}(t,\eta_1) P_{k_1}(\eta_1) \big) \bigg \Vert_{L^{2}_x}     \\
& \lesssim \Vert f \Vert_{X'} \\
& \lesssim \varepsilon_1.
\end{align*}
Note that the bound is uniform on $\beta.$ 
\end{proof}
Now we explain how to bound the magnetic Schr\"{o}dinger term. This is the main difficulty of the paper. This is where a new ingredient has to be introduced compared to the earlier paper \cite{L}. We replace the boundedness of wave operators by smoothing estimates from Section \ref{recallsmoothing}.

\begin{lemma}
We have the following bound:
\begin{align*}
\Vert \eqref{restediff} \Vert_{L^2 _x} \lesssim \delta \varepsilon_1
\end{align*}
when $W =  a_i$.
\end{lemma}
\begin{proof}
We use the identity \ref{identit} and plug it in \eqref{restediff}. \\
We split the integral according to the dominant direction of $\xi$ using Lemma \ref{direction}. \\
We will therefore estimate
\begin{align} 
\label{restediffb} & \chi_{j}(\xi) \xi_l P_k (\xi) \int_{0} ^{\infty} e^{-\tau_1 \beta} e^{i\tau_1 \vert \xi \vert^2} \int_{\mathbb{R}^3} e^{-i\tau_1 \vert \eta_1 \vert^2} \eta_{1,i} \widehat{a_i}(\xi-\eta_1) \frac{\eta_{1,j} \eta_{1,i}}{\vert \eta_1 \vert^2} e^{-it\vert \eta_1 \vert^2} \partial_{\eta_{1,j}} \widehat{f}(t,\eta_1) P_{k_1}(\eta_1) d\eta_1 d\tau_1 \\
\notag & =  \chi_{j}(\xi) \frac{\xi_l}{\vert \xi_j \vert^{1/2}} P_k(\xi) \vert \xi_j \vert^{1/2} \int_{0} ^{\infty} e^{-\tau_1 \beta} e^{i\tau_1 \vert \xi \vert^2} \int_{\mathbb{R}^3} e^{-i\tau_1 \vert \eta_1 \vert^2} \vert \eta_{1,i} \vert^{1/2} \widehat{a_i}(\xi-\eta_1) \widehat{g_{k_1}}(\eta_1) d\eta_1 d\tau_1  ,
\end{align} 
where
\begin{align*}
\widehat{g_{k_1}}(\eta_1) = \frac{\eta_{1,j} \vert \eta_{1,i}\vert ^{1/2}}{\vert \eta_1 \vert^2} \frac{\eta_{1,i}}{\vert \eta_{1,i} \vert} e^{-it\vert \eta_1 \vert^2} \partial_{\eta_{1,j}} \widehat{f}(t,\eta_1) P_{k_1}(\eta_1).
\end{align*}
Therefore using Young's inequality, \eqref{smo2} from Lemma \ref{smoothing}, Lemma \ref{mgnmgnfin} and Lemma \ref{X'} we find that
\begin{align*}
\Vert \eqref{restediffb} \Vert_{L^2_x} & \lesssim \big \Vert \mathcal{F}^{-1}_{\xi} \big( \chi_{j}(\xi) \frac{\xi_l}{\vert \xi_j \vert^{1/2}} P_k(\xi) \big) \big \Vert_{L^1_x} \\
& \times \Bigg \Vert D_{x_j}^{1/2} \int_{0}^{\infty} e^{-i \tau_1 \Delta}   \bigg( e^{-\tau_1 \beta} a_i (x) D_{x_i}^{1/2} e^{i\tau_1 \Delta} \big( g_{k_1} \big) \bigg) d\tau_1 \Bigg \Vert_{L^2 _x} \\
& \lesssim 1.1^{k/2} \bigg \Vert a_i (x) D_{x_i}^{1/2} e^{i\tau_1 \Delta} \big( g_{k_1} \big) \bigg \Vert_{L^1_{x_j} L^{2}_{\tau_1,\widetilde{x_j}}} \\
& \lesssim 1.1^{k/2} \delta \Vert g_{k_1} \Vert_{L^2_x}  \\
& \lesssim 1.1^{\frac{k-k_1}{2}} \delta \varepsilon_1  .
\end{align*}
\end{proof}

\section{Multilinear terms} \label{boundmultilinear}
In this section we prove Proposition \ref{nthestimate}. The estimates are based on key multilinear lemmas proved in the first subsection. We then use them to bound the iterates in the following subsection.

\subsection{Multilinear lemmas} \label{multilinkey}
We will start by proving multilinear lemmas that essentially allow us to reduce estimating the $n-$th iterates to estimating the first iterates. We distinguish between low and high output frequencies. We note that the case of main interest is that of high frequencies, since otherwise the loss of derivative is not a threat. However the proofs are slightly different in both cases, hence the need to separate the two. Besides, the case of low output frequencies is essentially analogous to the case of non-magnetic potentials. Therefore it is possible to see this part of the argument as an alternative proof of the result of \cite{L}.

\subsubsection{The case of large output frequency}
We start with the main multilinear Lemma of the paper. The other lemmas of this section will be variations of it.  
\begin{lemma} \label{multilinear}
Assume that $k>0.$ \\
We have the bound 
\begin{align} \label{multimain}
&\Bigg \Vert \int \prod_{\gamma=1}^{n-1} \frac{\alpha_{\gamma}(\eta_{\gamma})\widehat{W_{\gamma}}(\eta_{\gamma-1}-\eta_{\gamma})P_{k_{\gamma}}(\eta_{\gamma}) P_k(\xi)}{\vert \xi \vert^2 - \vert \eta_{\gamma} \vert^2} d\eta_1 ... d\eta_{n-1} \widehat{g_{k_n}}(\eta_n) d\eta_n \Bigg \Vert_{L^2} \\
& \lesssim q(\max K) C^n \delta^{n} \Vert g \Vert_{L^2} \prod_{\gamma \in J^{+}} 1.1^{-k_{\gamma}} \times \prod_{\gamma \in J^{-}} 1.1^{-k} 1.1^{\epsilon k_{\gamma}} ,
\end{align}
where $J^{+}= \lbrace j \in [[1;n ]]; k_j > k+1  \rbrace, \ \ J^{-}= \lbrace j \in [[1;n ]]; k > k_j+1  \rbrace, \ \ J = J^{+} \cup J^{-}.$ \\
$K$ denotes the complement of $J,$ $\epsilon$ denotes a number strictly between 0 and 1 and 
\begin{displaymath} q(\max K) = \begin{cases}
1 &  \textrm{if} ~~~~ \alpha_{\max K} = 1  \\
1.1^{k_{\max K}/2} & \textrm{otherwise}. 
\end{cases}
\end{displaymath}
Finally the implicit constant in the inequality does not depend on $n.$
\end{lemma}
\begin{remark}
Recall that by convention $\eta_0 = \xi.$
\end{remark}
\begin{remark}
The role of the products on elements of $J$ is to ensure that we can sum over $k_j, j \in J.$
\end{remark}
\begin{remark} \label{beta}
We will sometimes denote in the proof
\begin{displaymath} \beta_{\gamma} = \begin{cases}
0 &  \textrm{if} ~~~~ W_{\gamma} = V  \\
1 & \textrm{otherwise}.
\end{cases}
\end{displaymath}
\end{remark}
\begin{proof}
If $\gamma \in K$ then we write (recall that such terms have been regularized, see \eqref{c21beta})
\begin{align}\label{identite}
\frac{1}{\vert \xi \vert^2 - \vert \eta_{\gamma} \vert^2 + i\beta} =(-i) \int_{0} ^{\infty} e^{i \tau_{r(\gamma)}(\vert \xi \vert^2 - \vert \eta_{\gamma} \vert^2 + i \beta)} d\tau_{r({\gamma})} 
\end{align}
where $r(\gamma)$ denotes the number given to $\gamma$ in the enumeration of the elements of $K$ (if $\gamma$ is the second smallest element of $K$ then $r(\gamma) = 2$ for example).
\\
Our goal is to prove the bound that is uniform on $\beta$. Therefore for legibility we drop the terms $ e^{-\tau_{r(\gamma)} \beta} $ in the expressions above (they are systematically bounded by 1 in the estimates).
\\
We obtain the following expression 
\begin{align*}
\eqref{multimain} &=(-i)^{\vert K \vert} \int \prod_{ \gamma \in J} \alpha_{\gamma}(\eta_{\gamma})\widehat{W_{\gamma}}(\eta_{\gamma-1}-\eta_{\gamma}) m_{\gamma}(\xi,\eta_{\gamma}) 
\\
& \times \prod_{\gamma \in K} \int_{0}^{\infty} e^{i \tau_{r(\gamma)}(\vert \xi \vert^2 - \vert \eta_{\gamma} \vert^2)} d \tau_{r(\gamma)} \alpha_{\gamma}(\eta_{\gamma}) \widehat{W_{\gamma}}(\eta_{\gamma} - \eta_{\gamma-1})
\widehat{g_{k_n}}(\eta_n) d\eta_1.. d\eta_n
\end{align*} 
where 
\begin{align*}
m_{\gamma}(\xi,\eta) = \frac{P_k(\xi) P_{k_{\gamma}}(\eta)}{\vert \xi \vert^2 - \vert \eta \vert^2}.
\end{align*}
Now we perform the following change of variables:
\begin{align*}
\tau_1 & \leftrightarrow \tau_1 + \tau_2 + ... \\
\tau_2 & \leftrightarrow \tau_2+ \tau_3 + ... \\
...
\end{align*}
The expression becomes
\begin{align*}
\eqref{multimain} &=(-i)^{\vert K \vert} \int_{\tau_1} e^{i\tau_1 \vert \xi \vert^2} \Bigg( \int \prod_{ \gamma \in J} \alpha_{\gamma}(\eta_{\gamma})\widehat{W_{\gamma}}(\eta_{\gamma-1}-\eta_{\gamma}) m_{\gamma}(\xi,\eta_{\gamma}) 
\\
& \times \prod_{\gamma \in K, \gamma \neq \max K} \int_{\tau_{r(\gamma)+1} \leqslant \tau_{r(\gamma)}} e^{-i (\tau_{r(\gamma)} - \tau_{r(\gamma)+1} )\vert \eta_{\gamma} \vert^2} \alpha_{\gamma}(\eta_{\gamma}) \widehat{W_{\gamma}}(\eta_{\gamma-1} - \eta_{\gamma}) \\
& \times e^{-i \tau_{r(\max K)} \vert \eta_{\max K} \vert^2} \alpha_{\max K}(\eta_{\max K}) \widehat{W_{\max K}}(\eta_{\max K-1} -\eta_{\max K}) \widehat{g_{k_n}}(\eta_n) d\eta \Bigg) d\tau_1.
\end{align*} 
\bigskip
\underline{Case 1: $1 \in K,$ $\alpha_1 (\eta_1) = \eta_{1,i}, W_1 = a_i$}\\
We isolate the first term in the product:
\begin{align*}
\eqref{multimain} & = (-i)^{\vert K \vert} \int_{\tau_1} e^{i\tau_1 \vert \xi \vert^2} \int_{\eta_1} \widehat{a_i}(\xi-\eta_1) \eta_{1,i} e^{-i\tau_1 \vert \eta_1 \vert^2} P_{k_1}(\eta_1) \int_{\tau_2 \leqslant \tau_1} e^{i \tau_2 \vert \eta_1 \vert^2} \Bigg( \int \prod_{ \gamma \in J} \alpha_{\gamma}(\eta_{\gamma}) \\
& \times \widehat{W_{\gamma}}(\eta_{\gamma-1}-\eta_{\gamma}) m_{\gamma}(\xi,\eta_{\gamma}) \prod_{\gamma \in K, \gamma \neq 1, \gamma \neq \max K} \int_{\tau_{r(\gamma)+1}\leqslant \tau_{r(\gamma)}} e^{-i (\tau_{r(\gamma)} - \tau_{r(\gamma)+1} )\vert \eta_{\gamma} \vert^2} \alpha_{\gamma}(\eta_{\gamma}) \widehat{W_{\gamma}}(\eta_{\gamma-1} - \eta_{\gamma}) \\
& \times e^{-i \tau_{r(\max K)} \vert \eta_{\max K} \vert^2} \alpha_{\max K}(\eta_{\max K}) \widehat{W_{\max K}}(\eta_{\max K-1} -\eta_{\max K})\widehat{g_{k_n}}(t,\eta_n) d\eta \Bigg) d\tau_2 d\eta_1 d\tau_1.
\end{align*}
Now we take an inverse Fourier transform in $\xi.$ The terms in the expression above that contain $\xi$ are the first $a_i,$ and all the $m_{\gamma}$ for $\gamma \in J.$ 
\\
To simplify notations we denote
\begin{align*}
F_1(y_1,...,y_r,\eta_1)& =\int \prod_{ \gamma \in J} \alpha_{\gamma}(\eta_{\gamma})\widehat{W_{\gamma}}(\eta_{\gamma-1}-\eta_{\gamma}) \check{m}_{\gamma}(y_{r(\gamma)},\eta_{\gamma}) 
\\
\notag & \times \prod_{\gamma \in K, \gamma \neq 1, \gamma \neq \max K} \int_{\tau_{r(\gamma)+1}\leqslant \tau_{r(\gamma)}} e^{-i (\tau_{r(\gamma)} - \tau_{r(\gamma)+1} )\vert \eta_{\gamma} \vert^2 } \alpha_{\gamma}(\eta_{\gamma}) \widehat{W_{\gamma}}(\eta_{\gamma} - \eta_{\gamma-1})
\\
& \times e^{-i \tau_{r(\max K)} \vert \eta_{\max K} \vert^2} \alpha_{\max K}(\eta_{\max K}) \widehat{W_{\max K}}(\eta_{\max K} -\eta_{\max K-1})\widehat{g_{k_n}}(t,\eta_n) d\widetilde{\eta_1} \Bigg) d\tau_2 d\eta_1 d\tau_1,
\end{align*}
where $r = \vert J \vert.$ \\
Here we abusively wrote $\check{m}$ for the the Fourier transform with respect to the first variable only. \\
Also, as for element of $K,$ $r(\gamma)$ for $\gamma \in J$ denotes the number that $\gamma $ is assigned in the enumeration of the elements of $J.$
\\
We can then write
\begin{align*}
\mathcal{F}^{-1}_{\xi} \eqref{multimain} &=(2 \pi)^{3r} \int_{\tau_1} e^{i \tau_1 \Delta} \Bigg( \int_{y_r, r \in J}    a_i(x-y_1-...-y_r) \int_{\eta_1}  \eta_{1,i} e^{i\eta_1 \cdot(x-y_1-...-y_r)} e^{-i\tau_1 \vert \eta_1 \vert^2} \\
& \times P_{k_1}(\eta_1) \int_{\tau_2 \leqslant \tau_1} e^{i\tau_2 \vert \eta_1 \vert^2} F_1(y_1, ...,y_r,\eta_1) d\tau_2 d\eta_1 dy_1 ... dy_r \Bigg) d\tau_1.
\end{align*}
Using Strichartz estimates from Lemma \ref{Strichartz} we write
\begin{align*}
\Vert \eqref{multimain} \Vert_{L^2_x} & \lesssim (2 \pi)^{3r} \Bigg \Vert \int_{y_r, r \in J}    a_i(x-y_1-...-y_r) \int_{\tau_2 \leqslant \tau_1}\int_{\eta_1}  \eta_{1,i} e^{i\eta_1 \cdot(x-y_1-...-y_r)} e^{-i\tau_1 \vert \eta_1 \vert^2} \\
& \times P_{k_1}(\eta_1)  e^{i\tau_2 \vert \eta_1 \vert^2} F_1(y_1, ...,y_r,\eta_1) d\tau_2 d\eta_1 dy_1 ... dy_r \Bigg \Vert_{L^2_{\tau_1} L^{6/5}_{x}}.
\end{align*}
Now we estimate the right-hand side by duality. Consider $h(x,\tau_1) \in L^{2}_{\tau_1} L^{6}_{x} .$  \\
To simplify notations further, we denote
\begin{align*}
\widetilde{F_1}(\tau_1, y_1,...,y_r,\eta_1) = \int_{\tau_2 \leqslant \tau_1} e^{i\tau_2 \vert \eta_1 \vert^2} F_1(y_1, ...,y_r,\eta_1) d\tau_2.
\end{align*}
We pair the expression above against $h$, put the $x$ integral inside, change variables ($x \leftrightarrow x-y_1-... - y_r$) and use the Cauchy-Schwarz inequality in $\tau_1$ and then in $x_j:$
\begin{align} 
\label{multi1} & \Bigg \vert \int_x \int_{\tau_1} \int_{y_r, r \in J}    a_i(x-y_1-...-y_r) \int_{\eta_1} e^{i\eta_1 \cdot(x-y_1-...-y_r)} \\
\notag& \times e^{-i\tau_1 \vert \eta_1 \vert^2} \eta_{1,i} P_{k_1}(\eta_1) \widetilde{F_1}(y_1,...,y_r,\eta_1)  d\eta_1 h(x, \tau_1) dx d\tau_1 \Bigg \vert \\
\notag&=(2\pi)^3 \Bigg \vert \int_{y_r, r \in J}  \int_x a_i(x) D_{x_i} \int_{\tau_1}  e^{-i\tau_1 \Delta} \mathcal{F}^{-1}_{\eta_1} \big(P_{k_1}(\eta_1) \widetilde{F_1}(y_1,...,y_r,\eta_1) \big) h(x+y_1+...+y_r, \tau_1) d\tau_1 dx \Bigg \vert.
\end{align}
Now we can reproduce the proof of Lemma \ref{potmgn} (replacing $h$ by $h$ translated in space by a fixed vector). We find that
\begin{align*}
\vert \eqref{multi1} \vert & \lesssim \delta \int \Vert h \Vert_{L^2_{\tau_1} L^6_x} \Vert D_{x_i} e^{i\tau_1 \Delta} \mathcal{F}^{-1} \big(P_{k_1}(\eta_1) \widetilde{F_1}(y_1;...;y_r,\eta_1) \big) \Vert_{L^{\infty}_{x_j}L^2_{\tau_2,\widetilde{x_j}}} dy_1 ... dy_r.
\end{align*}
Now note that
\begin{align*}
D_{x_i} e^{i\tau_1 \Delta} \mathcal{F}^{-1}_{\eta_1} \big(P_{k_1}(\eta_1) \widetilde{F_1}(y_1,...,y_r,\eta_1) \big) & = D_{x_j} \int_{\tau_2 \leqslant \tau_1} e^{i(\tau_1-\tau_2)\Delta} \mathcal{F}^{-1}_{\eta_1} \big( P_{k_1}(\eta_1) F_1 \big) d\tau_2.
\end{align*}
Hence using the inhomogeneous smoothing estimate from Lemma \ref{smoothing} we obtain
\begin{align*}
\vert \eqref{multi1} \vert &\lesssim \delta \int \bigg \Vert \mathcal{F}^{-1}_{\eta_1} F_1(y_1,...,y_r,\eta_1) \bigg \Vert_{L^1_{x_j} L^2_{\tau_2, \widetilde{x_j}}} dy_1 ... dy_r.
\end{align*} 
Now we consider several subcases:
\\
\underline{Subcase 1.1: $2 \in K, \alpha_2(\eta_2)=\eta_{2,k}$} \\
Let's first assume that $2 \neq \max K.$ \\
Then notice that with similar notations as before, 
\begin{align*}
F_1(\tau_2, \eta_1) = \int_{\eta_2} \widehat{a_k}(\eta_1-\eta_2) \eta_{2,k} \int_{\tau_3 \leqslant \tau_2} e^{-i(\tau_2-\tau_3)\vert \eta_2 \vert^2}P_{k_2}(\eta_2) F_{2}(\eta_2,\tau_3) d\tau_3 d\eta_2.
\end{align*}
Hence 
\begin{align*}
\mathcal{F}^{-1}_{\eta_1} F_1 = (2\pi)^3 a_k (x) D_{x_k} \int_{\tau_3 \leqslant \tau_2} e^{i(\tau_2-\tau_3)\Delta} \mathcal{F}^{-1}_{\eta_2}\big( P_{k_2}(\eta_2) F_2(\eta_2,\tau_3) \big) d\tau_3,
\end{align*}
therefore we can use Lemma \ref{mgnmgn} to conclude that
\begin{align*}
\Vert \mathcal{F}^{-1}_{\eta_1} F_1 \Vert_{L^1_{x_j} L^2_{\tau_2,\widetilde{x_j}}} & \leqslant C \delta \Vert \mathcal{F}^{-1}_{\eta_2} F_2 \Vert_{L^1_{x_k} L^2_{\tau_3,\widetilde{x_k}}}.
\end{align*}
Now in the case where $2 = \max K$ then 
\begin{align*}
F_1 (\tau_2,\eta_1) = \int_{\eta_2} \widehat{a_k}(\eta_1-\eta_2) \eta_{2,k} e^{-i\tau_2 \vert \eta_2 \vert^2} \big(P_{k_2}(\eta_2) F_{2}(\eta_2) \big) d\eta_2.
\end{align*}
By a similar reasoning, we can write, using Lemma \ref{mgnmgnfin}
\begin{align*}
\Vert \mathcal{F}^{-1}_{\eta_1} F_1 \Vert_{L^1_{x_j} L^2_{\tau_2,\widetilde{x_j}}} & \leqslant 1.1^{k/2} C \delta \Vert \mathcal{F}^{-1}_{\eta_2} F_2 \Vert_{L^2_{x}}.
\end{align*}
This is where the $q$ term in the result comes from. \\
\\
\underline{Subcase 1.2: $2 \in K, \alpha_2(\eta_2) = 1 $} \\
Assume first that $2 \neq \max K.$ \\
Here we have
\begin{align*}
\mathcal{F}^{-1} F_1 (\tau_2) =(2\pi)^3 V (x) \int_{\tau_3 \leqslant \tau_2} e^{i(\tau_2-\tau_3)\Delta} \mathcal{F}^{-1}_{\eta_2} \big( P_{k_2}(\eta_2) F_2(\eta_2,\tau_3) \big) d\tau_3,
\end{align*}
therefore using Lemma \ref{mgnpot} we obtain
\begin{align*}
\bigg \Vert V (x) \int_{\tau_3 \leqslant \tau_2} e^{i(\tau_2-\tau_3)\Delta} \mathcal{F}^{-1}_{\eta_2} \big(P_{k_2}(\eta_2) F_2(\eta_2,\tau_3) \big) d\tau_3 \bigg \Vert_{L^1_{x_j} L^2_{\tau_2,\widetilde{x_j}}} \leqslant C \delta \Vert \mathcal{F}^{-1} _{\eta_2} F_2 \Vert_{L^2_{\tau_3} L^{6/5}_x}.
\end{align*}
Now if $2=\max K,$ we have
\begin{align*}
\mathcal{F}^{-1} F_1 (\tau_2) = V (x) e^{i \tau_2\Delta} \bigg( \mathcal{F}^{-1}_{\eta_2} \big(P_{k_2}(\eta_2) F_2(\eta_2)\big) \bigg).
\end{align*}
Hence we obtain, using Lemma \ref{mgnpotfin}
\begin{align*}
\bigg \Vert  V (x) e^{i \tau_2\Delta} \bigg( \mathcal{F}^{-1}_{\eta_2} \big(P_{k_2}(\eta_2) F_2(\eta_2)\big) \bigg) \bigg \Vert_{L^1_{x_j} L^2_{\tau_2,\widetilde{x_j}}} \leqslant C \delta \Vert \mathcal{F}^{-1}_{\eta_2} F_2 \Vert_{L^{2}_x}.
\end{align*}
\underline{Case 1.3: $2 \in J$} \\
We consider two subcases: \\
\underline{Subcase 1.3.1: $2 \in J^{+}$} \\
In this case we conclude using Lemma \ref{bilinit} that
\begin{align*}
\int \Vert \mathcal{F}^{-1} F_1 \Vert_{L^1 _{x_j} L^{2}_{\tau_2,\widetilde{x_j}}} dy_1 & \leqslant C_0 \Vert \check{m_2} \Vert_{L^1} \Vert W_2 \Vert_{L^{\infty}} \Vert \mathcal{F}^{-1}_{\eta_2} F_2 \Vert_{L^1 _{x_j} L^{2}_{\tau_2,\widetilde{x_j}}} \\
& \leqslant C 1.1^{(\beta_2-2) k_2} \Vert W_2 \Vert_{L^{\infty}}  \Vert \mathcal{F}^{-1}_{\eta_2} F_2 \Vert_{L^1 _{x_j} L^{2}_{\tau_2,\widetilde{x_j}}},
\end{align*}
where as defined in Remark \ref{beta}, $\beta_2 = 0$ or $\beta_2 = 1$ depending on whether $W_2 = V$ or $W_2 = a_i.$ \\
\underline{Subcase 1.3.2: $2 \in J^{-} $} \\
In this case we use the third inequality in Lemma \ref{bilinit} and Bernstein's inequality to write 
\begin{align*}
\int \Vert \mathcal{F}^{-1}_{\eta_1} F_1 \Vert_{L^1 _{x_j} L^{2}_{\tau_2,\widetilde{x_j}}} dy_1 & \leqslant C_0 1.1^{\beta_2 k_2} 1.1^{-2k} \Vert W \Vert_{L^{\infty}_{x_j} L^{\frac{2+2c}{c}}_{\widetilde{x_j}} } \Vert \mathcal{F}^{-1} (F_2)_{k_2} \Vert_{L^1 _{x_j} L^{2(1+c)}_{\widetilde{x_j}} L^{2}_{\tau_2}} \\
& \leqslant 1.1^{-k} C \delta 1.1^{\epsilon k_2} \Vert \mathcal{F}^{-1} F_2 \Vert_{L^1 _{x_j} L^{2}_{\tau_2,\widetilde{x_j}}}.
\end{align*} 
\underline{Case 2: $1 \in K, $ $\alpha_1 (\eta_1)=1$} \\
This case is similar to the previous one, but we use Strichartz estimates instead of smoothing effects.
\\
By Plancherel's theorem, Minkowski's inequality and H\"{o}lder's inequality, we have
\begin{align*}
\Vert \eqref{multimain} \Vert_{L^2 _x} & \lesssim \Bigg \Vert \mathcal{F}^{-1}  \bigg(\int \prod_{ \gamma \in J} \alpha_{\gamma}(\eta_{\gamma}) \ \ ... \ \ d\eta_n d\eta_{n-1} ds \bigg) \Bigg \Vert_{L^2_x} \\
& =(2 \pi)^{3r} \Bigg \Vert \int_{\tau_1} e^{i \tau_1 \Delta} \bigg( \int_{y_1,...,y_r} V(x-y_1-...-y_r) \\
& \times \int_{\eta_1} e^{i\eta_1 \cdot(x-y_1-...-y_r)} e^{-i\tau_1 \vert \eta_1 \vert^2} P_{k_1}(\eta_1) \widetilde{F_1}(y_1,...,y_r,\eta_1)  d\eta_1 \bigg) d\tau_1 \Bigg \Vert_{L^2_x} \\
& \lesssim (2 \pi)^{3r} \int_{y_1,...,y_r} \Bigg  \Vert V(x) e^{i\tau_1 \Delta} \mathcal{F}^{-1}\big( P_{k_1}(\eta_1) \widetilde{F_1} \big)(x) \Bigg \Vert_{L^{2}_{\tau_1} L ^{6/5}_x} dy_1 ... dy_r \\
& \lesssim (2 \pi)^{3r} \int_{y_1,...,y_r} \Vert V \Vert_{L^{3/2}_x} \Vert e^{i \tau_1 \Delta} \mathcal{F}^{-1}_{\eta_1} \big( P_{k_1}(\eta_1) \widetilde{F_1} \big) \Vert_{L^2_{\tau_1} L^{6} _x} dy_1 ... dy_r \\
& \lesssim (2 \pi)^{3r} \delta \int_{y_1,...,y_r} \bigg \Vert \int_{\tau_2 \leqslant \tau_1} e^{i(\tau_1-\tau_2)\Delta} \big( \mathcal{F}^{-1}_{\eta_1} F_1 \big)_{k_1} \bigg \Vert_{L^2_{\tau_1}L^{6}_x} dy_1 ... dy_r \\
& \lesssim (2 \pi)^{3r} \delta \int_{y_1,...,y_r} \Vert \mathcal{F}^{-1}_{\eta_1} F_1 \Vert_{L^2_{\tau_2} L^{6/5}_x} dy_1 ... dy_r.
\end{align*}
Now distinguish several subcases: 
\\
\underline{Subcase 2.1: $2 \in K, \alpha_2(\eta_2)=1$ } \\
Assume that $2 \neq \max K.$ \\
In this case we can use Lemma \ref{potpot} and obtain
\begin{align*}
\Vert \mathcal{F}^{-1}_{\eta_1} F_1 \Vert_{L^2_{\tau_2} L^{6/5}_x} \leqslant C \delta \Vert \mathcal{F}^{-1}_{\eta_2} F_2 \Vert_{L^2_{\tau_3} L^{6/5}_x}.
\end{align*}
In the case where $2 = \max K$ then we use Lemma \ref{potpotfin} to obtain
\begin{align*}
\Vert \mathcal{F}^{-1}_{\eta_1} F_1 \Vert_{L^2_{\tau_2} L^{6/5}_x} \leqslant C \delta \Vert \mathcal{F}^{-1}_{\eta_2} F_2 \Vert_{L^{2}_x}.
\end{align*}
\underline{Subcase 2.2: $2 \in K, \alpha_2(\eta_2) = \eta_{2,l} $} \\
We assume first that $2 \neq \max K.$ \\
In this case we use Lemma \ref{potmgn} and obtain
\begin{align*}
\bigg \Vert a_l (x) D_{x_l} \int_{\tau_3 \leqslant \tau_2} e^{i(\tau_2-\tau_3)\Delta} \mathcal{F}^{-1}_{\eta_2} F_2(\eta_2,\tau_3) d\tau_3 \bigg \Vert_{L^2_{\tau_2} L^{6/5}_x} \leqslant C \delta \Vert \mathcal{F}^{-1}_{\eta_2} F_2(\eta_2,\tau_3) \Vert_{L^1_{x_l} L^2_{\tau_3,\widetilde{x_l}}}.
\end{align*}
Now we treat the case where $2 = \max K.$ \\
We use Lemma \ref{mgnpotfin} and write
\begin{align*}
\Vert a_l (x) D_{x_l} e^{i\tau_2 \Delta} \mathcal{F}^{-1}_{\eta_2} F_2(\eta_2) \Vert_{L^2_{\tau_2} L^{6/5}_x} \leqslant C 1.1^{k/2} \delta \Vert \mathcal{F}^{-1}_{\eta_2} F_2(\eta_2,\tau_3) \Vert_{L^2_{x}}.
\end{align*}
\underline{Subcase 2.3: $ 2 \in J $} \\
\underline{Subcase 2.3.1: $2 \in J^{+} $} \\
In this case we use Lemma \ref{bilinit} to write that
\begin{align*}
\int_{y_1} \Vert \mathcal{F}^{-1}_{\eta_1} F_1 \Vert_{L^2 _{\tau_2} L^{6/5}_x} dy_1 & \leqslant C_0 \Vert \check{m_2} \Vert_{L^1} \Vert W_2 \Vert_{L^{\infty}} \Vert \mathcal{F}^{-1}_{\eta_2} F_2 \Vert_{L^2 _{\tau_2} L^{6/5} _x} \\
& \leqslant C 1.1^{- k_2} \Vert W_2 \Vert_{L^{\infty}} \Vert \mathcal{F}^{-1}_{\eta_2} F_2 \Vert_{L^2 _{\tau_2} L^{6/5} _x} .
\end{align*} 
\underline{Subcase 2.3.2: $2 \in J^{-} $}
In this case we use Lemma \ref{bilinit} as well as Bernstein's inequality to write that
\begin{align*}
\int_{y_1} \Vert \mathcal{F}^{-1}_{\eta_1} F_1 \Vert_{L^2 _{\tau_2} L^{6/5}_x} dy_1 & \leqslant C_0 \Vert \check{m_2} \Vert_{L^1} \Vert W_2 \Vert_{L^{\frac{6+6c}{5c}}_x} \Vert \big( \mathcal{F}^{-1}_{\eta_2} F_2 \big)_{k_2} \Vert_{L^2 _{\tau_2} L^{\frac{6}{5}(1+c)} _x} \\
& \leqslant C 1.1^{-k} \Vert W_2 \Vert_{L^{\frac{6+6c}{5c}}_x} 1.1^{\epsilon k_2} \Vert \mathcal{F}^{-1}_{\eta_2} F_2 \Vert_{L^2 _{\tau_2} L^{6/5} _x}.
\end{align*} 
\underline{Conclusion in cases 1 and 2:} In all subcases we reduced the problem to estimating $\mathcal{F}^{-1}_{\eta_2} F_2$ in either $L^2_{\tau_2} L^{6/5}_x$ or $L^1_{x_j} L^2_{\tau_2,\widetilde{x_j}}.$ Since $\mathcal{F}^{-1}_{\eta_2} F_2$ has the exact same form as $\mathcal{F}^{-1}_{\eta_1} F_1$ but with one less term in the product, and that its $L^2_{\tau_2} L^{6/5}_x$ or $L^1_{x_j} L^2_{\tau_2,\widetilde{x_j}}$ norms have already been estimated, we can conclude that we have the desired bound by induction. \\
\\
\underline{Case 3: $1 \in J$} 
\\
In this case we can add a frequency localization on the first potential $W_1.$ Let's denote $k_{max} = \max \lbrace k,k_1 \rbrace.$ 
\\
In this case we write
\begin{align*}
\eqref{multimain} = \int_{\tau_1} e^{i \tau_1 \vert \xi \vert^2} \int_{\eta_1} \widehat{W_{1,k_{\max}}}(\xi-\eta_1) m_1(\xi,\eta_1) F_1 (y_1,...,y_r,\eta_1) d\eta_1,
\end{align*}
with the same notation as in the previous cases.
\\
Now we take the inverse Fourier transform in $\xi$ and obtain
\begin{align*}
\mathcal{F}^{-1}_{\xi} \eqref{multimain} &= (2\pi)^{3r} \int_{\tau_1} e^{-i \tau_1 \Delta} \Bigg( \int_{y_1,...,y_r} W_{1,k_{\max}}(x-y_1-...-y_r) \\
& \times \int_{\eta_1} e^{i \eta_1 \cdot (x-y_1-...)}  m(y_1,\eta_1) P_{k_1}(\eta_1) F_1 (y_1,...,y_r,\eta_1) d\eta_1 dy_1 ... dy_r \Bigg) d\tau_1 \\
& = (2\pi)^{3r} \int_{\tau_1} e^{i\tau_1 \Delta} \Bigg( \int_{y_1,...y_r} W_{1,k_{\max}} (x-y_1-...y_r) \\
& \times \mathcal{F}^{-1}_{\eta_1} \big(  m(y_1,\eta_1) P_{k_1}(\eta_1) F_1(\eta_1) \big)(x-y_1-...) dy_1 ... dy_r \Bigg) d\tau_1.
\end{align*}
Now we use Strichartz estimates, Minkowki's inequality and H\"{o}lder's inequality to write that
\begin{align*}
\Vert \mathcal{F}^{-1}_{\xi} \eqref{multimain} \Vert_{L^2_x} & \leqslant (2\pi)^{3r} C_0 ^3 \Bigg \Vert \int_{y_1,...y_r} W_{1,k_{\max}} (x-y_1-...y_r) \\
& \times \mathcal{F}^{-1}_{\eta_1} \big(m_1(y_1,\eta_1) F_1(\eta_1) \big)(x-y_1-...) dy_1 ... dy_r \Bigg \Vert_{L^{2}_{\tau_1} L^{6/5}_x } \\
& \leqslant (2\pi)^{3(r+1)} C_0^3 \Bigg \Vert \int_{y_1,...y_r} W_{1,k_{\max}} (x-y_1-...y_r) \int_{z} \check{m_1}(y_1,z) 
\\ & \times \bigg( \mathcal{F}^{-1}_{\eta_1} \big(F_1(\eta_1) \big) \bigg)_{k_1} (x-z-y_1-...) dz dy_1 ... dy_r  \Bigg \Vert_{L^{2}_{\tau_1} L^{6/5}_x } .
\end{align*}
Now we distinguish several subcases:
\\
\underline{Subcase 3.1: $1 \in J^{+} $} \\
Then we can conclude directly using Lemma \ref{bilinit} and Minkowski's inequality that
\begin{align*}
\Vert \mathcal{F}^{-1}_{\xi} \eqref{multimain} \Vert_{L^2_x} & \leqslant (2 \pi)^{3r} C \Vert W_{1} \Vert_{L^{\infty}} 1.1^{(\beta_1-2) k_1} \int_{y_2,...,y_r} \Vert \mathcal{F}^{-1}_{\eta_1} \big( F_1 (\eta_1) \big) \Vert_{L^2_{\tau_1} L^{6/5}_x} dy_2 ... dy_r.
\end{align*}
\underline{Subcase 3.2: $1 \in J^{-} $} \\
Then we use Lemma \ref{bilinit}, Minkowki's inequality and Bernstein's inequality to write that
\begin{align*}
\Vert \mathcal{F}^{-1}_{\xi} \eqref{multimain} \Vert_{L^2_x} & \leqslant (2 \pi)^{3r} C 1.1^{-k} \Vert W_{1,k} \Vert_{L^{\frac{6+6c}{5c}}_x}  \int_{y_2,...,y_r} \Vert \bigg( \mathcal{F}^{-1}_{\eta_1} \big( F_1 (\eta_1) \big) \bigg)_{k_1} \Vert_{L^2_{\tau_1} L^{6/5(1+c)}_x} dy_2 ... dy_r \\
& \leqslant (2 \pi)^{3r} C 1.1^{-k} 1.1^{\epsilon k_1} \delta \int_{y_2,...,y_r} \Vert \mathcal{F}^{-1}_{\eta_1} \big( F_1 (\eta_1) \big) \Vert_{L^2_{\tau_1} L^{6/5}_x} dy_2 ... dy_r,
\end{align*}
and then we can conclude by induction as in the previous cases.
\end{proof}
Now we give a similar lemma that contains a gain of $1/2$ of a derivative compared to the previous one. 
\begin{lemma} \label{multilinearautre}
Assume that $k>0.$ \\
We have the bound 
\begin{align} \label{multimainautre}
&\Bigg \Vert \int \prod_{\gamma=1}^{n-1} \frac{\alpha_{\gamma}(\eta_{\gamma})\widehat{W_{\gamma}}(\eta_{\gamma-1}-\eta_{\gamma})P_{k_{\gamma}}(\eta_{\gamma}) P_k(\xi)}{\vert \xi \vert^2 - \vert \eta_{\gamma} \vert^2} d\eta_1 ... d\eta_{n-1} \widehat{g_{k_n}}(\eta_n) d\eta_n \Bigg \Vert_{L^2_x} \\
& \lesssim 1.1^{-k/2} q(\max K) C^n \delta^{n} \Vert g \Vert_{L^2_x} \prod_{\gamma \in J^{+}} 1.1^{-k_{\gamma}} \times \prod_{\gamma \in J^{-}} 1.1^{-k} 1.1^{\epsilon k_{\gamma}} ,
\end{align}
where $J^{+}= \lbrace j \in [[1;n ]]; k_j > k+1  \rbrace ,\ \ J^{-}= \lbrace j \in [[1;n ]]; k > k_j+1  \rbrace, \ \ J = J^{+} \cup J^{-}.$ \\
$K$ denotes the complement of $J$, $\epsilon$ denotes a small strictly positive number and
\begin{displaymath} q(\max K) = \begin{cases}
1 &  \textrm{if} ~~~~ \alpha_{\max K} = 1  \\
1.1^{k_{\max K}/2} & \textrm{otherwise}   
\end{cases}
\end{displaymath}
Finally the implicit constant in the inequality does not depend on $n.$
\end{lemma}
\begin{proof}
The proof of this lemma is similar to Lemma \ref{multilinear}. The only difference is in the set-up. \\
We must bound
\begin{align}\label{multilinautre}
& \int_{\tau_1} e^{i\tau_1 \vert \xi \vert^2} \Bigg( \int \prod_{ \gamma \in J} \alpha_{\gamma}(\eta_{\gamma})\widehat{W_{\gamma}}(\eta_{\gamma-1}-\eta_{\gamma}) m_{\gamma}(\xi,\eta_{\gamma}) 
\\
\notag & \times \prod_{\gamma \in K,\gamma \neq \max K} \int_{\tau_{r(\gamma)+1} \leqslant \tau_{r(\gamma)}} e^{-i (\tau_{r(\gamma)} - \tau_{r(\gamma)+1} )\vert \eta_{\gamma} \vert^2)} \alpha_{\gamma}(\eta_{\gamma}) \widehat{W_{\gamma}}(\eta_{\gamma-1} - \eta_{\gamma}) \\
\notag & \times \widehat{W_{\max K}}(\eta_{\max K-1}- \eta_{\max K}) \alpha_{\max K}(\eta_{\max K}) e^{-i\tau_{r(\max K) }\vert \eta_{\max K} \vert^2}  \widehat{g_{k_n}}(\eta_n) d\eta \Bigg) d\tau_1.
\end{align} 
We localize the $\xi$ variable according to the dominant direction using Lemma \ref{direction}. Then isolate the first term in the product:
\begin{align} 
\label{obj} & P_k(\xi) \chi_{j}(\xi) \int_{\tau_1} e^{i\tau_1 \vert \xi \vert^2} \Bigg( \int \prod_{ \gamma \in J} \alpha_{\gamma}(\eta_{\gamma})\widehat{W_{\gamma}}(\eta_{\gamma-1}-\eta_{\gamma}) m_{\gamma}(\xi,\eta_{\gamma}) 
\\
\notag & \times \prod_{\gamma \in K,\gamma \neq \max K} \int_{\tau_{r(\gamma)}} e^{-i (\tau_{r(\gamma)} - \tau_{r(\gamma)+1} )\vert \eta_{\gamma} \vert^2} \alpha_{\gamma}(\eta_{\gamma}) \widehat{W_{\gamma}}(\eta_{\gamma-1} - \eta_{\gamma}) \\
\notag & \times \widehat{W_{\max K}}(\eta_{\max K-1}- \eta_{\max K}) \alpha_{\max K}(\eta_{\max K}) e^{-i\tau_{r(\max K)} \vert \eta_{\max K} \vert^2}  \widehat{g_{k_n}}(\eta_n) d\eta \Bigg) d\tau_1 \\
\notag & = P_k(\xi) \chi_j(\xi) \int_{\tau_1} e^{i\tau_1 \vert \xi \vert^2} \int_{\eta_1} \widehat{W_1}(\xi-\eta_1)\alpha_1(\eta_1) e^{-i\tau_1 \vert \eta_1 \vert^2} \\
\notag & \times P_{k_1}(\eta_1) \int_{\tau_2 \leqslant \tau_1} e^{i \tau_2 \vert \eta_1 \vert^2} \Bigg( \int \prod_{ \gamma \in J} \alpha_{\gamma}(\eta_{\gamma})\widehat{W_{\gamma}}(\eta_{\gamma-1}-\eta_{\gamma}) m(\xi,\eta_{\gamma}) 
\\
\notag & \times \prod_{\gamma \in K, \gamma \neq 1, \gamma \neq \max K} \int_{\tau_{r(\gamma)+1}\leqslant \tau_{r(\gamma)}} e^{-i (\tau_{r(\gamma)} - \tau_{r(\gamma)+1} )\vert \eta_{\gamma} \vert^2} \alpha_{\gamma}(\eta_{\gamma}) \widehat{W_{\gamma}}(\eta_{\gamma-1} - \eta_{\gamma}) \\
\notag & \times \widehat{W_{\max K}}(\eta_{\max K-1}- \eta_{\max K}) \alpha_{\max K}(\eta_{\max K}) e^{-i \tau_{r(\max K)} \vert \eta_{\max K} \vert^2}  \widehat{g_{k_n}}(\eta_n) d\eta \Bigg) d\tau_1.
\end{align}
Now we take an inverse Fourier transform in $\xi.$ The terms in the expression above that contain $\xi$ are the first $W,$ and all the $m_{\gamma}$ for $\gamma \in J.$ The complex exponential gives a Schr\"{o}dinger semi-group, and the other terms give a convolution. \\
To simplify notations we denote (using similar notations as above)
\begin{align*}
F_1(y_1,...,y_r,\eta_1) &=\int \prod_{ \gamma \in J} \alpha_{\gamma}(\eta_{\gamma})\widehat{W_{\gamma}}(\eta_{\gamma-1}-\eta_{\gamma}) \check{m}(y_{r(\gamma)},\eta_{\gamma}) 
\\
\notag & \times \prod_{\gamma \in K, \gamma \neq 1, \gamma \neq \max K} \int_{\tau_{r(\gamma)+1}\leqslant \tau_{r(\gamma)}} e^{-i (\tau_{r(\gamma)} - \tau_{r(\gamma)+1} )\vert \eta_{\gamma} \vert^2)} \alpha_{\gamma}(\eta_{\gamma}) \widehat{W_{\gamma}}(\eta_{\gamma-1} - \eta_{\gamma}) \\
\notag & \times \widehat{W_{\max K}}(\eta_{\max K-1}- \eta_{\max K}) \alpha_{\max K}(\eta_{\max K}) e^{-i\tau_{r(\max K)} \vert \eta_{\max K} \vert^2}  \widehat{g_{k_n}}(\eta_n) d\eta.
\end{align*}
Hence
\begin{align*}
\mathcal{F}^{-1}_{\xi} \eqref{obj} &=(2 \pi)^{3} \mathcal{F}^{-1}_{\xi} \big(P_k(\xi) \chi_j(\xi) \frac{1}{\vert \xi_j \vert ^{1/2}} \big) \ast \mathcal{F}^{-1} _{\xi}  \Bigg( \vert \xi_j^{1/2} \vert \int \prod_{ \gamma \in J} \alpha_{\gamma}(\eta_{\gamma}) \ \ ... \ \ d\eta_n d\eta_{n-1} ds \Bigg) 
\\
&=(2 \pi)^{3} \bigg(\mathcal{F}^{-1}_{\xi}\big(P_k(\xi) \chi_j(\xi)  \frac{1}{\vert \xi_j \vert ^{1/2}} \big) \bigg) \\
& \ast \Bigg( \int_{\tau_1} D_{x_j}^{1/2}  e^{-i \tau_1 \Delta} \Bigg( \int_{y_r, r \in J}    W_1(x-y_1-...-y_r) \int_{\eta_1}  \alpha_1(\eta_1) e^{i\eta_1 \cdot(x-y_1-...-y_r)} e^{-i\tau_1 \vert \eta_1 \vert^2} \\
& \times P_{k_1}(\eta_1) \int_{\tau_2 \leqslant \tau_1} e^{i\tau_2 \vert \eta_1 \vert^2} F_1(y_1, ...,y_r) d\tau_2 d\eta_1 \Bigg) d\tau_1 \Bigg).
\end{align*}
Using Lemma \ref{smoothing} we get
\begin{align*}
\Vert \eqref{obj} \Vert_{L^2_x} & \lesssim 1.1^{-k/2} \Bigg \Vert \int_{y_r, r \in J}    W_1(x-y_1-...-y_r) \int_{\tau_2 \leqslant \tau_1}\int_{\eta_1}  \alpha_1(\eta_1) e^{i\eta_1 \cdot(x-y_1-...-y_r)} e^{-i\tau_1 \vert \eta_1 \vert^2} \\
& \times P_{k_1}(\eta_1)  e^{i\tau_2 \vert \eta_1 \vert^2} F_1(y_1, ...,y_r) d\tau_2 d\eta_1 \Bigg \Vert_{L^1_{x_j} L^{2}_{\tau_1,\widetilde{x_j}}}.
\end{align*}
If $\alpha_1(\eta)=\eta_{1,i}$ then using a similar proof to that of Lemma \ref{mgnmgn} we find
\begin{align*}
\Vert \eqref{obj} \Vert_{L^2_x} & \lesssim 1.1^{-k/2} \int_{y_r, r \in J}  \Vert \mathcal{F}^{-1}_{\eta_1} F_1 \Vert_{L^1_{x_i} L^{2}_{\tau_2,\widetilde{x_i}}} dy_1 ... dy_r.
\end{align*}
If $\alpha_1(\eta_1)=1$ then using a similar proof to that of Lemma \ref{mgnpot} we have
\begin{align*}
\Vert \eqref{obj} \Vert_{L^2_x} & \lesssim 1.1^{-k/2} \int_{y_r, r \in J}  \Vert \mathcal{F}^{-1}_{\eta_1} F_1 \Vert_{L^{2}_{\tau_2} L^{6/5}_x} dy_1 ... dy_r.
\end{align*}
From that point the same proof as Lemma \ref{multilinear} can be carried out to prove the desired result.
\end{proof}
We have the following straightforward corollary which will be useful in the next section. 
\begin{corollary} \label{multilinearbis}
Assume that $k>0.$ \\
We have the bound
\begin{align} \label{multimainbis}
&\Bigg \Vert \int \prod_{\gamma=1}^{n-1} \frac{\alpha_{\gamma}(\eta_{\gamma})\widehat{W_{\gamma}}(\eta_{\gamma-1}-\eta_{\gamma})P_{k_{\gamma}}(\eta_{\gamma}) P_k(\xi)}{\vert \xi \vert^2 - \vert \eta_{\gamma} \vert^2} d\eta_1 ... d\eta_{n-1} \widehat{g_{k_n}}(\xi,\eta_n) d\eta_n \Bigg \Vert_{L^2_x} \\
& \lesssim C^n \delta^{n} \Vert \mathcal{F}^{-1}_{\xi,\eta_n} g \Vert_{L^1 _{y_r} L^2_x} \prod_{\gamma \in J^{+}} 1.1^{-k_{\gamma}} \times \prod_{\gamma \in J^{-}} 1.1^{-k} 1.1^{\epsilon k_{\gamma}} ,
\end{align}
where the notations are the same as in previous lemmas. \\
Finally the implicit constant in the inequality does not depend on $n.$
\end{corollary}
\begin{proof}
The proof of this lemma is either identical to Lemma \ref{multilinear} or \ref{multilinearautre} depending on the value of $\alpha_{\max K}.$ The only minor difference is that the dependence of $g$ on $\xi$ adds a convolution in the physical variable.
\end{proof}
Finally the following version of the above lemma will be useful to use Strichartz estimates for the multilinear expressions. 
\begin{corollary} \label{multiin}
Assume that $k>0.$ \\
We have the bound
\begin{align*}
&\Bigg \Vert \int \prod_{\gamma=1}^{n-1} \frac{\alpha_{\gamma}(\eta)\widehat{W_{\gamma}}(\eta_{\gamma-1}-\eta_{\gamma})P_{k_{\gamma}}(\eta_{\gamma}) P_k(\xi)}{\vert \xi \vert^2 - \vert \eta_{\gamma} \vert^2} d\eta_1 ... d\eta_{n-1} \int_{1}^t e^{i s \vert \xi \vert^2} \widehat{g_{k_n}}(s,\xi,\eta_n) ds d\eta_n \Bigg \Vert_{L^2_x} \\
& \lesssim C^n \delta^{n} \Vert g \Vert_{L^{p'}_t L^{q'}_x} \prod_{\gamma \in J^{+}} 1.1^{-k_{\gamma}} \times \prod_{\gamma \in J^{-}} 1.1^{-k} 1.1^{\epsilon k_{\gamma}} ,
\end{align*}
where $(p,q)$ is a Strichartz admissible pair of exponents and $\epsilon$ denotes a small (strictly less than 1) strictly positive number. \\
The implicit constant in the inequality above does not depend on $n.$ 
\end{corollary}
\begin{proof}
Since the proof is similar to Lemma \ref{multilinearautre}, we only sketch it here. \\
First note that we can extend the domain of integration of $s$ to $(0;+\infty)$ by multiplying $g$ by $\textbf{1}_{(1;t)}.$ 
\\
After replacing the singular denominators by their integral expression using \eqref{identite} and doing the change of variables
\begin{align*}
\tau_1 & \leftrightarrow \tau_1 + \tau_2 + ... +s \\
\tau_2 & \leftrightarrow \tau_2+ \tau_3 + ... +s \\
... & \\
s & \leftrightarrow s
\end{align*}
The expression becomes
\begin{align*}
\eqref{multimain} &=(-i)^{\vert K \vert} \int_{\tau_1} e^{i\tau_1 \vert \xi \vert^2} \Bigg( \int \prod_{ \gamma \in J} \alpha_{\gamma}(\eta_{\gamma})\widehat{W_{\gamma}}(\eta_{\gamma-1}-\eta_{\gamma}) m_{\gamma}(\xi,\eta_{\gamma}) 
\\
& \times \prod_{\gamma \in K, \gamma \neq \max K} \int_{\tau_{r(\gamma)+1} \leqslant \tau_{r(\gamma)}} e^{-i (\tau_{r(\gamma)} - \tau_{r(\gamma)+1} )\vert \eta_{\gamma} \vert^2} \alpha_{\gamma}(\eta_{\gamma}) \widehat{W_{\gamma}}(\eta_{\gamma-1} - \eta_{\gamma}) \\
& \times \int_{s \leqslant \tau_{\vert K \vert}} e^{i(s-\tau_{\vert K \vert}) \vert \eta_{\max K} \vert^2} \widehat{W_{\max K}}(\eta_{\max K-1}-\eta_{\max K}) \alpha_{\max K} (\eta_{\max K})  
\widehat{g_{k_n}}(s,y_r,\eta_n) d\eta \Bigg) d\tau_1.
\end{align*} 
Now distinguish two cases:
\\
\underline{Case 1: $\alpha_{\max K} = 1 $} \\
Then we bound all the terms as in the proof of Lemma \ref{multilinear} until the last one: \\
To bound $\mathcal{F}^{-1} F_{\max K-1}$ we write using retarded Strichartz estimates that
\begin{align*}
\bigg \Vert V(x) \int_{s \leqslant \tau_{\vert K \vert}} e^{i(s-\tau_{\vert K \vert})\Delta} \mathcal{F}^{-1} \big(P_{k_{\max K}}(\eta_{\max K}) F_{\max K }(s,\cdot) \big) ds \bigg \Vert_{L^1_{x_j} L^2 _{\tau_{\vert K \vert},\widetilde{x_j}}} \leqslant C \delta \Vert  \mathcal{F}^{-1}  F_{\max K} \Vert_{L^{p'}_t L^{q'}_x },
\end{align*}
and 
\begin{align*}
\bigg \Vert V(x) \int_{s \leqslant \tau_{\vert K \vert}} e^{i(s-\tau_{\vert K \vert})\Delta}  \mathcal{F}^{-1} \big( P_{k_{\max K}}(\eta_{\max K}) F_{\max K}(s,\cdot) \big) ds \bigg \Vert_{L^2 _{\tau_{\vert K \vert}} L^{6/5}_x} \leqslant C \delta \Vert  \mathcal{F}^{-1} F_{\max K} \Vert_{L^{p'}_t L^{q'}_x }.
\end{align*}
Now in the expression of $F_{\max K}$ there are only terms in $J$ therefore we can conclude the proof using bilinear lemmas.
\\
\\
\underline{Case 2: $\alpha_{\max K}(\eta_{\max K}) = \eta_{\max K,i} $} \\
In this case we repeat the proof of Lemma \ref{multilinearautre} except for the last term in the product on elements of $K.$ \\
Using either Lemma \ref{mgnmgnfin} or \ref{potmgnfin} we obtain
\begin{align*}
& \bigg \Vert a_i(x) D_{x_i} \int_{s \leqslant \tau_{\vert K \vert}} e^{i(s-\tau_{\vert K \vert})\Delta} \mathcal{F}^{-1} \big( P_{k_{\max K}}(\eta_{\max K}) F_{\max K}(s,\cdot) \big) ds \bigg  \Vert_{L^1_{x_j} L^2 _{\tau_{\vert K \vert},\widetilde{x_j}}} \\ & \leqslant 1.1^{k/2} C\delta \Vert  \mathcal{F}^{-1} F_{\max K} \Vert_{L^{p'}_t L^{q'}_x },
\end{align*}
and 
\begin{align*}
& \bigg \Vert a_i(x) D_{x_i} \int_{s \leqslant \tau_{\vert K \vert}} e^{i(s-\tau_{\vert K \vert})\Delta} \mathcal{F}^{-1} \big( P_{k_{\max K}}(\eta_{\max K}) F_{\max K}(s,\cdot) \big) ds \bigg  \Vert_{L^2 _{\tau_n} L^{6/5}_x} \\
& \leqslant 1.1^{k/2} C \delta \Vert \mathcal{F}^{-1} F_{\max K} \Vert_{L^{p'}_t L^{q'}_x }.
\end{align*}
We deduce the result in this case. 
\end{proof}
We end with the following simpler version of the above lemma:
\begin{corollary} \label{multiinbis}
We have the bound
\begin{align*}
&\Bigg \Vert \int \prod_{\gamma=1}^{n-1} \frac{\alpha_{\gamma}(\eta)\widehat{W_{\gamma}}(\eta_{\gamma-1}-\eta_{\gamma})P_{k_{\gamma}}(\eta_{\gamma}) P_k(\xi)}{\vert \xi \vert^2 - \vert \eta_{\gamma} \vert^2} d\eta_1 ... d\eta_{n-1} \int_{1}^t e^{i s \vert \xi \vert^2} \widehat{g_{k_n}}(s,\xi,\eta_n) ds d\eta_n \Bigg \Vert_{L^2} \\
& \lesssim t \delta^{\vert K \vert} C^n \Vert g \Vert_{L^{\infty}_s ([1,t]) L^{2}_x} \prod_{\gamma \in J^{+}} 1.1^{-k_{\gamma}} \times \prod_{\gamma \in J^{-}} 1.1^{-k} 1.1^{\epsilon k_{\gamma}} .
\end{align*}
The implicit constant in the inequality does not depend on $n.$
\end{corollary}

\subsubsection{The case of small output frequency}
Now we write analogs of the previous Lemmas for $k<0.$ In this case the loss of derivative is helpful, and therefore we do not need to use the smoothing effect. We only resort to Strichartz estimates, which makes this case simpler. 
\\
We start with the analog of Lemma \ref{multilinear}:

\begin{lemma} \label{multilinearneg}
Assume that $k<0.$ \\
We have the bound 
\begin{align} \label{multimainneg}
&\Bigg \Vert \int \prod_{\gamma=1}^{n-1} \frac{\alpha_{\gamma}(\eta_{\gamma})\widehat{W_{\gamma}}(\eta_{\gamma-1}-\eta_{\gamma})P_{k_{\gamma}}(\eta_{\gamma}) P_k(\xi)}{\vert \xi \vert^2 - \vert \eta_{\gamma} \vert^2} d\eta_1 ... d\eta_{n-1} \widehat{g_{k_n}}(\eta_n) d\eta_n \Bigg \Vert_{L^2_x} \\
& \lesssim C^n \delta^{n} \Vert g \Vert_{L^2_x} \prod_{\gamma \in J} \min \lbrace 1.1^{0.5 k_{\gamma-1}} ; 1 \rbrace \prod_{\gamma \in J^{+}} \min \lbrace 1.1^{-k_{\gamma}} ; 1.1^{\epsilon k_{\gamma}} \rbrace \times \prod_{\gamma \in J^{-}} 1.1^{k_{\gamma}-k}.
\end{align}
The implicit constant in the inequalities does not depend on $n.$
\end{lemma}
\begin{proof}
Since the proof is almost identical to that of Lemma \ref{multilinear} where only case 1 is considered. The main difference appears when we deal with terms for which $\gamma \in J.$
\\
Therefore we only consider these terms here. They are of the form
\begin{align} \label{gammans}
\big( \mathcal{F}^{-1} F_{\gamma-1}  \big)(x)= \int_{y,z} \check{m_{\gamma}}(y,z) W_{\gamma}(x-y) \big( \mathcal{F}^{-1} F_{\gamma} \big)_{k_{\gamma}} (x-y-z) dy dz.
\end{align}
We must estimate the $L^2 _t L^{6/5}_x$ norm of this expression.  \\
We distinguish two cases: 
\\
\underline{Case 1: $k_{\gamma} >k +1$} \\
The inequalities in this case give the $J^{+}$ terms in the product. \\
\underline{Subcase 1.1: $k_{\gamma}>0$} \\
Then we write using Lemma \ref{bilin} that 
\begin{align*}
\Vert \eqref{gammans} \Vert_{L^2 _t L^{6/5}_x} \leqslant C \Vert W \Vert_{L^{\infty}_x} 1.1^{-k_{\gamma}} \Vert \big( \mathcal{F}^{-1} F_{\gamma} \big)\Vert_{L^2 _t L^{6/5}_x}.
\end{align*}
\underline{Subcase 1.2: $k_{\gamma}<0 $} \\
In this case, we use Lemma \ref{bilin} again as well as Bernstein's inequality to write that
\begin{align*}
\Vert \eqref{gammans} \Vert_{L^2 _t L^{6/5}_x} & \leqslant C_0 \Vert W \Vert_{L^{3/2-}_x} 1.1^{-k_{\gamma}} \big \Vert \big( \mathcal{F}^{-1} F_{\gamma} \big)_{k_{\gamma}} \big \Vert_{L^2 _t L^{6+}_x} \\
& \leqslant C \delta 1.1^{\epsilon k_{\gamma}} \Vert \big( \mathcal{F}^{-1} F_{\gamma} \big)_{k_{\gamma}} \Vert_{L^2 _t L^{6/5}_x}
\end{align*}
which can be summed. 
\\
\underline{Case 2: $k>k_{\gamma}+1$} \\
The inequalities in this case give the $J^{-}$ terms in the product. \\
In this case we write as in the previous case that
\begin{align*}
\Vert \eqref{gammans} \Vert_{L^2 _t L^{6/5}_x} & \leqslant C_0 \Vert W \Vert_{L^{3/2}_x} 1.1^{\beta k_{\gamma}} 1.1^{-2k} \big \Vert \big( \mathcal{F}^{-1} F_{\gamma} \big)_{k_{\gamma}} \big \Vert_{L^2 _t L^{6}_x} \\
& \leqslant C \delta 1.1^{2(k_{\gamma}-k)} \Vert \big( \mathcal{F}^{-1} F_{\gamma} \big)_{k_{\gamma}} \Vert_{L^2 _t L^{6/5}_x}.
\end{align*}
Similarly, we now show how to obtain the extra factor $\prod_{\gamma \in J} \min \lbrace 1 ; 1.1^{0.5 k_{\gamma-1}} \rbrace$. We bound
\begin{align} \label{gammansprime}
(\mathcal{F}^{-1} F_{\gamma-1})_{k_{\gamma-1}} =\big( \int_{y,z} \check{m}(y,z) W_{\gamma}(x-y) \big( \mathcal{F}^{-1} F_{\gamma}\big)_{k_{\gamma}} (x-y-z) dy dz \big)_{k_{\gamma-1}}.
\end{align}
We must, for the the same reason as above, bound the $L^2 _t L^{6/5}_x$ norm of that expression. \\
We start by using Bernstein's inequality and obtain
\begin{align*}
\Vert (\mathcal{F}^{-1} F_{\gamma-1})_{k_{\gamma-1}} \Vert_{L^2 _t L^{6/5}_x} \leqslant C_0 1.1^{0.5 k_{\gamma-1}} \Vert (\mathcal{F}^{-1} F_{\gamma-1})_{k_{\gamma-1}} \Vert_{L^2 _t L^{1}_x},
\end{align*}
and then the proof is identical to the previous inequality: \\
We distinguish two cases: 
\\
\underline{Case 1: $k_{\gamma} >k +1$} \\
The inequalities in this case give the $J^{+}$ terms in the product. \\
\underline{Subcase 1.1: $k_{\gamma}>0$} \\
Then we write using Lemma \ref{bilin} that 
\begin{align*}
\Vert \eqref{gammansprime} \Vert_{L^2 _t L^{1}_x} \leqslant C \Vert W \Vert_{L^{6}_x} 1.1^{-k_{\gamma}}  \Vert  \mathcal{F}^{-1} F_{\gamma}  \Vert_{L^2 _t L^{6/5}_x}.
\end{align*}
\underline{Subcase 1.2: $k_{\gamma}<0 $} \\
In this case we use Lemma \ref{bilin} again as well as Bernstein's inequality to write that
\begin{align*}
\Vert \eqref{gammansprime} \Vert_{L^2 _t L^{1}_x} & \leqslant C_0 \Vert W \Vert_{L^{6/5-}_x} 1.1^{-2k_{\gamma}} \Vert \big( \mathcal{F}^{-1} F_{\gamma} \big)_{k_{\gamma}} \Vert_{L^2 _t L^{6+}_x} \\
& \leqslant C \delta 1.1^{\epsilon k_{\gamma}} \Vert \mathcal{F}^{-1} F_{\gamma} \Vert_{L^2 _t L^{6/5}_x},
\end{align*}
which can be summed. 
\\
\underline{Case 2: $k>k_{\gamma}+1$} \\
The inequalities in this case give the $J^{-}$ terms in the product. \\
In this case we write as in the previous case that
\begin{align*}
\Vert \eqref{gammansprime} \Vert_{L^2 _t L^{1}_x} & \leqslant C_0 \Vert W \Vert_{L^{6/5}_x} 1.1^{-2k} \bigg \Vert \big( \mathcal{F}^{-1} F_{\gamma} \big)_{k_{\gamma}} \bigg \Vert_{L^2 _t L^{6}_x} \\
& \leqslant C \delta 1.1^{2(k_{\gamma}-k)} \Vert \mathcal{F}^{-1} F_{\gamma} \Vert_{L^2 _t L^{6/5}_x}.
\end{align*}
\end{proof}
We keep recording analogs of the previous section.
\begin{corollary} \label{multilinearbisneg}
Assume that $k<0.$ \\
We have the bound
\begin{align} \label{multimainbisneg}
&\Bigg \Vert \int \prod_{\gamma=1}^{n-1} \frac{\alpha_{\gamma}(\eta_{\gamma})\widehat{W_{\gamma}}(\eta_{\gamma-1}-\eta_{\gamma})P_{k_{\gamma}}(\eta_{\gamma}) P_k(\xi)}{\vert \xi \vert^2 - \vert \eta_{\gamma} \vert^2} d\eta_1 ... d\eta_{n-1} \widehat{g_{k_n}}(\xi,\eta_n) d\eta_n \Bigg \Vert_{L^2_x} \\
& \lesssim C^n \delta^{n} \Vert \mathcal{F}^{-1}_{\xi,\eta_n} g \Vert_{L^1 _{y_r} L^2_x} \prod_{j \in J} \min \lbrace 1.1^{0.5 k_{\gamma-1}} ; 1 \rbrace \prod_{\gamma \in J^{+}} \min \lbrace 1.1^{-k_{\gamma}} ; 1.1^{0.5 k_{\gamma}} \rbrace \times \prod_{\gamma \in J^{-}} 1.1^{k_{\gamma}-k},
\end{align}
where the notations are the same as in previous lemmas. \\
Finally the implicit constant in the inequalities do not depend on $n.$
\end{corollary}
Finally for Strichartz estimates we need:
\begin{corollary} \label{multiinneg}
Assume that $k<0.$ \\
We have the bound
\begin{align*}
&\Bigg \Vert \int \prod_{\gamma=1}^{n-1} \frac{\alpha_{\gamma}(\eta_{\gamma})\widehat{W_{\gamma}}(\eta_{\gamma-1}-\eta_{\gamma})P_{k_{\gamma}}(\eta_{\gamma}) P_k(\xi)}{\vert \xi \vert^2 - \vert \eta_{\gamma} \vert^2} d\eta_1 ... d\eta_{n-1} \int_{1}^t e^{i s \vert \xi \vert^2} \widehat{g_{k_n}}(s,\xi,\eta_n) ds d\eta_n \Bigg \Vert_{L^2_x} \\
& \lesssim  C^n \delta^{n} \Vert g \Vert_{L^{p'}_t L^{q'}_x} \prod_{\gamma \in J} \min \lbrace 1.1^{0.5 k_{\gamma-1}};1 \rbrace \prod_{\gamma \in J^{+}} \min \lbrace 1.1^{-k_{\gamma}} ; 1.1^{0.5 k_{\gamma}} \rbrace \times \prod_{\gamma \in J^{-}} 1.1^{k_{\gamma}-k}.
\end{align*}
The implicit constant in the inequality does not depend on $n.$
\end{corollary}
And we also have the following easier version:
\begin{corollary} \label{multiinneg}
Assume that $k<0.$ \\
We have the bound
\begin{align*}
&\Bigg \Vert \int \prod_{\gamma=1}^{n-1} \frac{\alpha_{\gamma}(\eta_{\gamma})\widehat{W_{\gamma}}(\eta_{\gamma-1}-\eta_{\gamma})P_{k_{\gamma}}(\eta_{\gamma}) P_k(\xi)}{\vert \xi \vert^2 - \vert \eta_{\gamma} \vert^2} d\eta_1 ... d\eta_{n-1} \int_{1}^t e^{i s \vert \xi \vert^2} \widehat{g_{k_n}}(s,\xi,\eta_n) ds d\eta_n \Bigg \Vert_{L^2_x} \\
& \lesssim  C^n \delta^{n} t \Vert g \Vert_{L^{\infty}_t L^2_x} \prod_{\gamma \in J} \min \lbrace 1.1^{0.5 k_{\gamma-1}};1 \rbrace \prod_{\gamma \in J^{+}} \min \lbrace 1.1^{-k_{\gamma}} ; 1.1^{0.5 k_{\gamma}} \rbrace \times \prod_{\gamma \in J^{-}} 1.1^{k_{\gamma}-k}.
\end{align*}
The implicit constant in the inequality does not depend on $n.$
\end{corollary}

\subsection{Bounding $n-$th iterates} \label{Boundingiterates}
In this section we prove the bounds announced in proposition \ref{nthestimate}. In spirit they all follow from the previous multilinear lemmas and the bounds for the first iterates (Section \ref{firsterm}). However we cannot always localize the potential in frequency as easily: say for example that $k_1 \ll k$ then the potential term $\widehat{W_1}(\xi-\eta_1)$ in the first iterate is localized at frequency $1.1^k$. In the multilinear setting however we cannot conclude that $\widehat{W_n}(\eta_{n-1}-\eta_n)$ will be localized at $1.1^{k_{n-1}}$ from the fact that $k_n \ll k.$ Therefore we need to adjust slightly some of the proofs. 
\\
\\
We introduce the following notation to simplify the expressions that appear:
\begin{definition}
In this section and the next $G(\textbf{k})$ will denote either 
\begin{align*}
\prod_{j \in J^{+}} \min \lbrace 1.1^{-k_{\gamma}} ; 1.1^{\epsilon k_{\gamma}}  \rbrace \times \prod_{j \in J^{-}} 1.1^{-k}1.1^{\epsilon k_{\gamma}} 
\end{align*}
or 
\begin{align*}
\prod_{j \in J^{+}} \min \lbrace 1.1^{-k_{\gamma}} ; 1.1^{\epsilon k_{\gamma}} \rbrace \times \prod_{j \in J^{-}} 1.1^{k_{\gamma}-k} 
\end{align*}
depending on whether $k>0$ or $k \leqslant 0.$ 
\end{definition}
We start by estimating the terms that come up in the iteration in the case where $\vert k_n - k \vert >1.$ 
\begin{lemma}
We have the bound
\begin{align*}
\Vert I_1 ^n f \Vert_{L^2_x} & \lesssim C^n G(\textbf{k}) \delta^n \varepsilon_1 .
\end{align*}
\end{lemma}
\begin{proof}
We start by splitting the $\eta_n$ variable dyadically. We denote $k_n$ the corresponding exponent.
\\
We can apply Corollary \ref{multiin} with 
\begin{align*}
\widehat{g_{k_{n-1}}}(s,\eta_{n-1},\xi) = \int_{\eta_{n}} s \frac{P_k(\xi) P_{k_{n}}(\eta_n)\alpha_{n}(\eta_n) \xi_l \widehat{W_n}(\eta_{n-1}-\eta_{n})}{\vert \xi \vert^2 - \vert \eta_{n} \vert^2} \widehat{u^2} (s,\eta_n) d\eta_n.
\end{align*} 
Note that in this case $n \in J,$ meaning in the last term in the product on $\gamma$ the denominator is not singular.
\\
\underline{Case 1: $W_n = V, k_n -1 > k $} \\
In this case the $\xi_l$ in front of the expression $I_1 ^n$ contributes $1.1^k$ and the symbol with denominator $\frac{1}{\vert \xi \vert^2 - \vert \eta_n \vert^2}$ contributes $1.1^{-2 k_n}.$ \\
In this case we use Bernstein's inequality and Lemma \ref{multilinear} (with $(p,q)=(2,6)$) and obtain
\begin{align*}
\Vert I_1 ^n f \Vert_{L^2} & \lesssim C^n G(\textbf{k}) \delta^n 1.1^{k-2k_n} \Vert V \Vert_{L^{6/5}_x} \Vert t (u^2)_{k_n} \Vert_{L^2_t L^{\infty}_x} \\
& \lesssim C^n G(\textbf{k}) \delta^n 1.1^{k-k_n} \Vert t (u^2)_{k_n} \Vert_{L^2_t L^{3}_x} \\
& \lesssim C^n G(\textbf{k}) \delta^n 1.1^{k-k_n} \bigg \Vert \Vert t u \Vert_{L^{6}_x} \Vert u \Vert_{L^6_x} \bigg \Vert_{L^2 _t} \\
& \lesssim C^n G(\textbf{k}) \delta^n 1.1^{k-k_n} \varepsilon_1^2 ,
\end{align*}
which can be summed over $k_n$ using Lemma \ref{summation} and the fact that $k_n>k.$ 
\\
\underline{Case 2: $W_n = V, k_n + 1 < k $}
\\
This case is handled as case 1. 
\\
\underline{Case 3: $W_n = a_i, k_n -1 > k$} \\
In this case we obtain a $1.1^{k-k_n}$ factor in front which is directly summable.
\\
\underline{Case 4: $W_n = a_i, k_n +1 < k$} \\
In this case we obtain a $1.1^{k-2k+k_n}$ factor in front which is directly summable.
\end{proof}
The following term also appears when $\vert k-k_n \vert >1.$
\begin{lemma}
We have the bound:
\begin{align*}
\Vert I_2 ^n f \Vert_{L^2_x} \lesssim C^n G(\textbf{k}) \delta^n \varepsilon_1.
\end{align*}
\end{lemma}
\begin{proof}
\underline{Case 1: $k>0.$} \\
We can write, using Lemma \ref{bilin} and Corollary \ref{multilinearbis} with 
\begin{align*}
\widehat{g_{k_{n-1}}}(t,\xi,\eta_{n-1}) = \int_{\eta_{n}} \frac{ P_k(\xi) \xi_l \widehat{W_n}(\eta_{n-1}-\eta_n) \alpha_{n}(\eta_n)}{\vert \xi \vert^2 - \vert \eta_n \vert^2} e^{it \vert \eta_n \vert^2} t \widehat{f_{k_n}}(t,\eta_n) d\eta_n
\end{align*}
that we have the bound
\begin{align*}
\Vert I_2 ^n f \Vert_{L^2_x} & \lesssim C^{n} G(\textbf{k}) \delta^{n} \Bigg \Vert \mathcal{F}^{-1}_{\eta_{n-1}} \int_{\eta_{n}} \frac{ P_k(\xi) \xi_l \widehat{W_n}(\eta_{n-1}-\eta_n) \alpha_{n}(\eta_n)}{\vert \xi \vert^2 - \vert \eta_n \vert^2} e^{it \vert \eta_n \vert^2} t \widehat{f_{k_n}}(t,\eta_n) d\eta_n \Bigg \Vert_{L^2_x} .
\end{align*}
\underline{Subcase 1.1: $k_n > k+1$} \\
Then using Bernstein's inequality, we obtain (we denote $\beta_n=0$ if $\alpha_n = 1$ and $\beta_n = 1$ otherwise):
\begin{align*}
\Vert I_2 ^n f \Vert_{L^2_x} & \lesssim \delta^{n} C^{n} G(\textbf{k}) 1.1^{-2k_n} 1.1^k 1.1^{\beta_n k_n} \Vert W_n \Vert_{L^3 _x} \Vert t e^{it\Delta} f_{k_n} \Vert_{L^6 _x},
\end{align*}
and this bound is directly summable over $k_n. $ 
\\
\underline{Subcase 1.2: $k_n <k-1$} \\
If $\beta_n = 1,$ we can conclude as in the previous case. \\
If $\beta_n = 0, $ then we use Bernstein's inequality and obtain
\begin{align*}
\Vert I_2 ^n f \Vert_{L^2_x} & \lesssim \delta^{n} C^{n} G(\textbf{k}) 1.1^{-2k} 1.1^k \Vert W \Vert_{L^{2}_x} \Vert t e^{it\Delta} f_{k_n} \Vert_{L^{\infty} _x} \\
& \lesssim \delta^{n} C^{n} G(\textbf{k}) 1.1^{0.5(k_n-k)} \Vert W \Vert_{L^{2}_x} \Vert t e^{it\Delta} f_{k_n} \Vert_{L^{6} _x}.
\end{align*}
We can conclude using Lemma \ref{dispersive}.
\\
\\
\underline{Case 2: $k<0.$} \\
We distinguish several subcases:
\\
\underline{Subcase 2.1: $k>k_n+1$} \\
\underline{Subcase 2.1.1: $k_{n-1} \leqslant k+1$} \\
Then we use Lemma \ref{multilinearneg} as well as Bernstein's inequality and write that
\begin{align*}
\Vert I_2 ^n f \Vert_{L^2} & \lesssim 1.1^k C^n \delta^n G(\textbf{k}) 1.1^{(\beta_n-2) k} \Vert W_{\leqslant k+10} \Vert_{L^{3-}_x} \Vert t e^{it\Delta} f_{k_n} \Vert_{L^{6+}_x} \\
& \lesssim C^n \delta^n G(\textbf{k}) 1.1^{\beta_n k}  \Vert W \Vert_{L^{3/2-}_x} 1.1^{\epsilon k_n}\Vert t e^{it\Delta} f_{k_n} \Vert_{L^{6}_x} .
\end{align*}
\underline{Subcase 2.1.2: $k_{n-1}> k+1$} \\
If $\beta_n = 1$ we can conclude as above
\\
If $\beta_n = 0$ then consider the largest integer $\gamma \in J$ such that $\gamma-1 \in K.$ In this case we use Lemma \ref{multilinearneg} and at least one of the terms in the factor
\begin{align*}
\prod_{\gamma \in J} 1.1^{0.5 k_{\gamma-1}}
\end{align*}
is equal to $1.1^{0.5 k}$. We use Bernstein's inequality to write that 
\begin{align*}
\Vert I_2 ^n f \Vert_{L^2_x} & \lesssim 1.1^k 1.1^{0.5 k} C^n G(\textbf{k}) \delta^n 1.1^{- 2 k} \Vert W \Vert_{L^{2}_x} \Vert t e^{it\Delta} f_{k_n} \Vert_{L^{\infty}_x} \\
& \lesssim 1.1^{0.5(k_n-k)} C^n \delta^n G(\textbf{k}) \Vert W \Vert_{L^{2}_x} \Vert t e^{it\Delta} f_{k_n} \Vert_{L^{6}_x}, 
\end{align*}
which can be summed over $k_n$ since $k_n-k<0.$
\\
\underline{Subcase 2.2: $k<k_n-1$} \\
The proof is similar in this case, therefore it is omitted.
%\underline{Subcase 2.1: $k_{n-1} \leqslant k+1$}
%Then use Bernstein on $W$ to gain an additional $1.1^k$ as in subcase 1.1
%\begin{align*}
%\Vert I_2 ^n f \Vert_{L^2} & \lesssim 1.1^k C^n \delta^n 1.1^{-2 \beta k_n} \Vert W_{\leqslznt k_n +10} \Vert_{L^{3}} \Vert t e^{it\Delta} f_{k_n} \Vert_{L^{6}_x} \\
%& \lesssim  C^n \delta^n 1.1^{k} 1.1^{-2 \beta k_n} 1.1^{(2\beta -1) k_n} \delta \Vert t e^{it\Delta} f_{k_n} \Vert_{L^{6}_x} 
%\end{align*} 
%which can be summed over $k_n$ since $k_n> k+1$. \\
%\\
%\underline{Subcase 2.2: $k_{n-1} > k+1 $} \\
%Then for the same reason as in subcase 1.2 we gain an additional $1.1^{k/2}$ and we can conclude in this case using Lemma \ref{bilin} and Bernstein's inequality. 
%\begin{align*}
%\Vert I_2 ^n f \Vert_{L^2} & \lesssim 1.1^k 1.1^{0.5 k} C^n \delta^n 1.1^{-2 k_n} \Vert W \Vert_{L^{2}} \Vert t e^{it\Delta} f_{k_n} \Vert_{L^{\infty}_x} \\
%& \lesssim C^n \delta^n 1.1^{1.5 k} 1.1^{-2k_n} \delta 1.1^{0.5 k_n} \Vert t e^{it\Delta} f_{k_n} \Vert_{L^{6}_x} 
%\end{align*}
%where here we assume $W=V.$ The case $W=a_i$ is treated in a similar fashion. (in that case the additional term is not needed)
\end{proof}
Now we bound $I_3 ^n f$. This term is always such that $\vert k-k_n \vert >1.$
\begin{lemma}
We have the bound
\begin{align*}
\Vert I_3 ^n f \Vert_{L^2_x} \lesssim \delta^{n} C^n G(\textbf{k}) \varepsilon_1.
\end{align*}
The implicit constant does not depend on $n$ here. 
\end{lemma}
\begin{proof}
\underline{Case 1: $k>0$}
\\
We use Corollary \ref{multiin} with
\begin{align*}
\widehat{g_{k_{n-1}}} (s,\xi,\eta_{n-1}) & = \int_{\eta_n} \frac{\xi_l P_{k}(\xi) W_n(\eta_{n-1}-\eta_n) \alpha_n (\eta_n)}{\vert \xi \vert - \vert \eta_n \vert^2} e^{-is\vert \eta_n \vert^2} \widehat{f_{k_n}}(s,\eta_n) d\eta_n.
\end{align*}
\\
\underline{Subcase 1.1: $k_n > k+1$} \\
In this case we obtain
\begin{align*}
\Vert I_3 ^n f \Vert_{L^2_x} & \lesssim C^n  \delta^{n} G(\textbf{k}) \Vert g \Vert_{L^2_s L^{6/5}_x}.
\end{align*}
We estimate that last term using Lemma \ref{bilin} (as usual $\beta_n=0$ if $\alpha_n = 1$ and 1 otherwise)
\begin{align*}
\Vert g \Vert_{L^2_s L^{6/5}_x} & \lesssim 1.1^{k} 1.1^{-2k_n} 1.1^{\beta_n k_n} \Vert W_n \Vert_{L^{3/2}_x} \Vert e^{is \Delta} f_{k_n} \Vert_{L^2_s L^6_x} \\
& \lesssim 1.1^{k-k_n} 1.1^{(\beta_n-1)k_n}  \Vert W_n \Vert_{L^{3/2}_x} \Vert e^{is \Delta} f_{k_n} \Vert_{L^2_s L^6_x},
\end{align*}
% and the last term $1.1^{(\beta_n-1)k_n} \leqslant 1$ since $k_n>k>0.$
which can be summed using Lemma \ref{summation}. \\
\underline{Subcase 1.2: $k>k_n+1$} \\
This case is similar to the previous one. 
\\
\\
\underline{Case 2: $k<0$} \\
Now we assume that $k<0.$ We only treat the worse case ($W=V$).
\\
\underline{Subcase 2.1: $k_n > k+1$} \\
\underline{Subcase 2.1.1: $k_{n-1} \leqslant k+1 $} \\
Then we can use the fact that $W$ is then localized at frequency less that $1.1^{k_n+10}$ and we use Lemma \ref{multiinneg} Bernstein's inequality to write that
\begin{align*}
\Vert I_3 ^n f \Vert_{L^2} & \lesssim C^n G(\textbf{k}) \delta^n 1.1^{k} 1.1^{-2k_n} \Vert W_{n,\leqslant k_n} \Vert_{L^{3/2}_x} \Vert e^{it\Delta} f_{k_n} \Vert_{L^2_t L^6 _x} \\
                           & \lesssim 1.1^{k-k_n} C^n G(\textbf{k}) \delta^n \Vert W_n \Vert_{L^{1}_x} \Vert e^{it\Delta} f_{k_n} \Vert_{L^2 _t L^6 _x}
\end{align*}
and we can conclude using Lemma \ref{summation}. \\
\underline{Subcase 2.1.2: $k_{n-1} > k+1$} \\
In this case, as in the previous proof, we use Lemma \ref{multiinneg} to gain an additional $1.1^{k/2}$ factor. Overall we get the bound
\begin{align*}
\Vert I_3 ^n f \Vert_{L^2 _x} & \lesssim C^n G(\textbf{k}) \delta^n  1.1^{1.5 k} 1.1^{-2k_n} \Vert W \Vert_{L^{6/5}_x} \Vert e^{it\Delta} f_{k_n}  \Vert_{L^2 _t L^{\infty}_x} \\
& \lesssim C^n G(\textbf{k}) \delta^n 1.1^{1.5(k-k_n)} \delta  \Vert e^{it\Delta} f_{k_n}  \Vert_{L^2 _t L^{6}_x},
\end{align*} 
and we can conclude by Lemma \ref{summation}.
\\
\underline{Subcase 2.2: $k_n< k-1$} \\
This case is treated similarly to the previous one.
\end{proof}
Now we come to the terms that arise in the case $\vert k-k_n \vert \leqslant 1.$ We start with $I_n ^4 f:$ 
\begin{lemma}
We have the bound
\begin{align*}
\Vert I_4 ^n f \Vert \lesssim C^n G(\textbf{k}) \delta^n \varepsilon_1 .
\end{align*}
\end{lemma}
\begin{proof}
This is a direct consequence of Lemma \ref{multilinear} (or \ref{multilinearautre}) and Lemma \ref{X'}:
\begin{align*}
\Vert I_4 ^n f \Vert_{L^2 _x} & \lesssim C^n G(\textbf{k}) \delta^n \bigg \Vert \mathcal{F}^{-1} \big( e^{-it \vert \eta_n \vert^2} \partial_{\eta_n,j} f P_{k_n}(\eta_n) \big) \bigg \Vert_{L^2 _x} \\
& \lesssim C^n G(\textbf{k}) \delta^n \Vert f \Vert_{X'} \\
& \lesssim C^n G(\textbf{k}) \delta^n \varepsilon_1.
\end{align*}
\end{proof}

Now we estimate the next few terms similarly, therefore all the estimates are grouped in the same lemma. Recall that all these terms appear when $\vert k-k_n \vert \leqslant 1.$

\begin{lemma}
We have the bounds
\begin{align*}
\Vert I_5 ^n f, I_6 ^n f, I_7 ^n f, I_8^n f \Vert_{L^2} \lesssim C^n G(\textbf{k}) \delta^n \varepsilon_1.
\end{align*}
\end{lemma}
\begin{proof}
We do the proof for $I_6 ^n f,$ since the other terms are easier to deal with. 
\\
\underline{Case 1: $k_{n-1} \leqslant k_n +10$} \\
In this case the potential $W_n$ is localized at frequency less than $1.1^{k+10}.$ Therefore we can use Lemma \ref{multiin} (or \ref{multiinneg}) for
\begin{align*}
\widehat{g_{k_{n-1}}}(s,\eta_{n-1}) = \int_{\eta_n} \alpha_n(\eta_n) \widehat{W_{n, \leqslant k+10}}(\eta_{n-1}-\eta_n) \partial_{\eta_{n,j}} \big( \frac{\xi_l \eta_{n,j}}{\vert \eta_n \vert^2} \big)  e^{-is \vert \eta_n \vert^2} \widehat{f_{k_n}}(s, \eta_n) d\eta_n,
\end{align*}
and obtain (we denote $\beta_n =0$ or $1$ depending on whether $\alpha_n = 1$ or not) 
\begin{align*}
\Vert I_6 ^n f  \Vert_{L^2} & \lesssim C^n G(\textbf{k}) \delta^n 1.1^{-k} 1.1^{\beta_n k} \Vert W_{\leqslant k} \Vert_{L^{3/2} _x} \Vert e^{it\Delta} f_{k_n} \Vert_{L^2 _t L^6 _x} \\
                            & \lesssim C^n G(\textbf{k}) \delta^n \Vert e^{it\Delta} f_{k_n} \Vert_{L^2 _t L^6 _x},
\end{align*}
and we can conclude using Lemma \ref{summation}
\\
\underline{Case 2: $k_{n-1}> k_n +10.$} \\
\underline{Subcase 2.1: $k>0$ } \\
This subcase can be treated as case 1.
\\
\underline{Subcase 2.2: $k \leqslant 0$} \\
In this case we use Lemma \ref{multiinneg}.\\
We can consider the largest $j$ such that $j-1 \in K.$ That term gives us an additional $1.1^{0.5k}$ factor. We obtain, by the same reasoning as in case 1, the bound (we only treat the worse case here, that is $\beta =0$, see case 1):
\begin{align*}
\Vert I_6 ^n f  \Vert_{L^2} & \lesssim C^n G(\textbf{k}) \delta^n 1.1^{-k} 1.1^{0.5k} \Vert W \Vert_{L^{6/5} _x} \Vert e^{it\Delta} f_{k_n} \Vert_{L^2 _t L^{\infty} _x} \\
                            & \lesssim C^n G(\textbf{k}) \delta^n \Vert e^{it\Delta} f_{k_n} \Vert_{L^2 _t L^6 _x},
\end{align*}
where to obtain the last line we used Bernstein's inequality. 
\end{proof}
Finally we have the expected bounds on the iterates of the bilinear terms:
\begin{lemma}
We have the bounds
\begin{align*}
\Vert I_9 ^n f, I_{10} ^n f \Vert_{L^2_x} \lesssim C^n G(\textbf{k}) \delta^n \varepsilon_1 .
\end{align*}
\end{lemma}
\begin{proof}
The proofs of these estimates can be straightforwardly adapted from \cite{L}, Lemmas 7.7 and 7.8. Therefore we will only treat the more complicated of the two terms, namely $I_9 ^n f.$ \\
\\
We split the frequencies $\eta_n$ and $\eta_{n-1}-\eta_n$ dyadically ($k_n$ and $k_{n+1}$ denote the corresponding exponents) as well as time ($m$ denotes the exponent): 
\begin{align*}
&i\int_1 ^t \int_{\eta_n} \eta_{n,j} e^{is(\vert \xi \vert^2 - \vert  \eta_n \vert^2 - \vert \eta_{n-1} \vert^2 )} s \widehat{f}(s,\eta_n) \widehat{f}(s,\eta_{n-1}-\eta_n) d\eta_n ds              \\
& =\sum_{m=0}^{\ln t} \sum_{k_1,k_2 \in \mathbb{Z}} \int_{1.1^m} ^{1.1^{m+1}} is \eta_{n,l} e^{is (\vert \xi \vert^2 - \vert \eta_n \vert ^2 - \vert \eta_{n-1} - \eta_n \vert ^2)} \widehat{f_{k_n}}(s,\eta_n) \widehat{f_{k_{n+1}}}(s,\eta_{n-1}-\eta_n) d\eta_n ds .
\end{align*}
\underline{Case 1: $\max \lbrace k_n ; k_{n+1} \rbrace \geqslant m$}
\\
We apply Lemma \ref{multiinbis} for %(or \ref{multiinbisneg}) for
\begin{align*}
\widehat{g}(s,\xi,\eta_{n-1}) = \textbf{1}_{(1.1^m;1.1^{m+1})}(s) \int_{\mathbb{R}^3} is \eta_{n,l} e^{is (\vert \xi \vert^2 - \vert \eta_n \vert ^2 - \vert \eta_{n-1} - \eta_n \vert ^2)} \widehat{f_{k_n}}(s,\eta_n) \widehat{f_{k_{n+1}}}(s,\eta_{n-1}-\eta_n) d\eta_n,
\end{align*}
as well as Lemma \ref{bilin} to write that
\begin{align*}
\Vert I_9 ^n f \Vert_{L^{\infty}_t L^2 _x} & \lesssim 1.1^{2m} 1.1^{\max \lbrace k_n ; k_{n+1} \rbrace} 1.1^{-10 \max \lbrace k_n ; k_{n+1} \rbrace} \\
& \times \min \big \lbrace 1.1^{-10 \min \lbrace k_n ; k_{n+1} \rbrace}; 1.1^{3 \min \lbrace k_n ; k_{n+1} \rbrace /2} \big \rbrace \delta^n C^n G(\textbf{k}) \varepsilon_1 ^2 \\
                                       & \lesssim 1.1^{-6m} 1.1^{-\max \lbrace k_n ; k_{n+1} \rbrace} \min \big \lbrace 1.1^{-10 \min \lbrace k_n ; k_{n+1} \rbrace}; 1.1^{3 \min \lbrace k_n ; k_{n+1} \rbrace /2} \big \rbrace G(\textbf{k}) \delta^n C^n \varepsilon_1 ^2,
\end{align*}
which we can sum over $k_n,k_{n+1}$ and $m.$ \\
\underline{Case 2: $\min \lbrace k_n ; k_{n+1} \rbrace \leqslant -2m $} \\
Similarly in this case we write that
\begin{align*}
\Vert I_9 ^n f \Vert_{L^{\infty}_t L^2 _x} & \lesssim  1.1^{2m} 1.1^{\max \lbrace k_n ; k_{n+1} \rbrace} 1.1^{3 \min \lbrace k_n ; k_{n+1} \rbrace /2} \\
& \times \min \big \lbrace 1.1^{-10 \max \lbrace k_n ; k_{n+1} \rbrace}; 1.1^{\max \lbrace k_n ; k_{n+1} \rbrace /3} \big \rbrace \delta^n C^n G(\textbf{k}) \varepsilon _1 ^2 \\
                                       & \lesssim  1.1^{-0.5 m} 1.1^{0.25 \min \lbrace k_n ; k_{n+1} \rbrace}  1.1^{\max \lbrace k_n ; k_{n+1} \rbrace} \\
                                       & \times \min \big \lbrace 1.1^{-10 \max \lbrace k_n ; k_{n+1} \rbrace}; 1.1^{\max \lbrace k_n ; k_{n+1} \rbrace /3} \big \rbrace \delta^n C^n G(\textbf{k}) \varepsilon _1 ^2,
\end{align*}
which can be summed.
\\
\underline{Case 3: $-2m \leqslant k_n , k_{n+1} \leqslant m$} \\
When the gradient of the phase is not too small, we can integrate by parts in $\eta_n$ to gain decay in time. To quantify this more precisely, we split dyadically in the gradient of the phase, namely $\eta_{n-1}-2\eta_n.$ We denote $k_{n}'$ the corresponding exponent.\\
\\
\underline{Case 3.1: $k_{n}' \leqslant -10 m$} \\
We apply Lemma \ref{multiinbis} for %(or \ref{multiinbisneg})
\begin{align*}
\widehat{g}(s,\eta_{n-1}) &= \textbf{1}_{(1.1^m;1.1^{m+1})}(s) \int_{\mathbb{R}^3} is \eta_{n,l} e^{is (\vert \xi \vert^2 - \vert \eta_n \vert ^2 - \vert \eta_{n-1} - \eta_n \vert ^2)} P_{k_{n}'}(2\eta_n- \eta_{n-1}) \\
&\times \widehat{f_{k_n}}(s,\eta_n) \widehat{f_{k_{n+1}}}(s,\eta_{n-1}-\eta_n) d\eta_n.
\end{align*}
As in Lemma 5.16 in \cite{L} we have 
\begin{align*}
\Vert g \Vert_{L^2_{\eta_{n-1}}} \lesssim 1.1^{-13m} 1.1^{0.1 k_{n}'} \varepsilon_1^2.
\end{align*}
Hence
\begin{align*}
\Vert I_9 ^n f \Vert_{L^{\infty}_t L^2 _x} & \lesssim 1.1^m 1.1^{0.1 k_{n}'} 1.1^{-13 m} \delta^n C^n G(\textbf{k}) \varepsilon_1 ^2,
\end{align*}
which can be summed over $k_{n}'$ and $k_n,k_{n+1}$ (there are only $O(m^2)$ terms in the sum) as well as over $m.$ 
\\
\\
\underline{Case 3.2: $k_{n}' \geqslant k_n-50, k_{n+2} \geqslant -10m,$ and $-2m \leqslant k_n,k_{n+1} \leqslant m$}\\
In this case we do an integration by parts in $\eta_n$. 
\\
Again, this case is similar to that of lemma 7.8, \cite{L}. All the terms that appear are treated following the same strategy, therefore we focus on the case where the $\eta_n$ derivative hits one of the profiles. \\
We can apply Lemma \ref{multiin} with $(p,q) = (4,3)$ and 
\begin{align*}
\widehat{g}(s,\xi,\eta_{n-1}) &= \textbf{1}_{(1.1^m;1.1^{m+1})}(s) \int_{\mathbb{R}^3} e^{is (\vert \xi \vert^2 - \vert \eta_n \vert ^2 - \vert \eta_{n-1} - \eta_n \vert ^2)} 
\frac{P_{k_{n}'}(2\eta_n- \eta_{n-1}) (2 \eta_n - \eta_{n-1})_j \eta_{n,l}}{\vert 2 \eta_n - \eta_{n-1} \vert^2} \\
&\times \widehat{f_{k_n}}(s,\eta_n) \partial_{\eta_{n,j}} \widehat{f_{k_{n+1}}}(s,\eta_{n-1}-\eta_n) d\eta_n,
\end{align*}
which yields the bound
\begin{align*}
\Vert I_9 ^n f \Vert_{L^{\infty}_t L^2 _x} \lesssim \delta^n  C^n G(\textbf{k}) 1.1^{k_n-k_{n}'} 1.1^{-m/4} \varepsilon_1 ^2.
\end{align*}
This expression can be summed given the assumptions on the indices in this case. 
\\
\\
\underline{Case 3.3: $-10 m \leqslant k_{n}' \leqslant k_n-10$ and $-2m \leqslant k_n,k_{n+1} \leqslant m$} \\
There is a slight difference in this case compared to the corresponding lemma in \cite{L} due to the presence of the magnetic potentials.
\\
Let's start with a further restriction: notice that $\eta_{n-1} - \eta_n = \eta_{n-1}- 2\eta_n + \eta_n$ therefore $ \vert \eta_{n-1}-\eta_n \vert \sim 1.1^{k_n} \sim 1.1^{k_{n+1}}.$ \\
Using Lemma \ref{multiinbis} as well as Bernstein's inequality and the fact that the $X$ norm of $f_{k_n}$ controls its $L^p -$ norms for $6/5<p\leqslant 2$ we get that 
\begin{align*}
\Vert I_9 ^n f \Vert_{L^2 _x} & \lesssim \delta^n C^n G(\textbf{k}) 1.1^{2m} 1.1^{k_n} \Vert u_{k_n} \Vert_{L^{\infty} _t L^{\infty}_x} \Vert u_{k_{n+1}} \Vert_{L^{\infty} _t L^{2}_x} \\
& \lesssim \delta^n C^n G(\textbf{k}) 1.1^{2m} 1.1^{k_n} 1.1^{2.49 k_n} 1.1^{0.99 k_n} \varepsilon_1 ^2 \\
& \lesssim 1.1^{2m} 1.1^{4.48 k_n} C^n \delta^n G(\textbf{k}) \varepsilon_1 ^2.
\end{align*}
If $k_n \leqslant - 101/224 m$ we can sum the expressions above. Indeed there are only $O(m^2)$ terms in the sums on $k_n, k_{n+1}$ therefore the decaying factor in $m$ is enough to ensure convergence.\\
As a result we can assume from now on that $k_n >-101/224 m.$
\\
\\
First, recall the following key symbol bound from \cite{L}, which was the reason for using a frequency localization at $1.1^k$ and not $2^k :$
\begin{align} \label{symbol1}
\bigg \Vert \mathcal{F}^{-1} \frac{P_{k}(\xi) P_{k_n}(\eta_n) P_{k_{n+1}}(\eta_{n+1}) P_{k_n'}(2\eta_n - \eta_{n-1}) }{\vert \xi \vert^2 - \vert \eta_n \vert^2 - \vert \eta_{n-1}-\eta_n \vert^2} \bigg \Vert_{L^1} \lesssim 1.1^{-2k_n} .
\end{align}
Now we integrate by parts in time. \\
Let's start with the easier boundary terms. They are both of the same form, therefore we only treat one of the terms. In this case we can apply Lemma \ref{multilinear} with
\begin{align*}
\widehat{g}(t,\xi,\eta_{n-1}) &= \textbf{1}_{(1.1^m;1.1^{m+1})}(t) \int_{\mathbb{R}^3} \frac{it P_{k_n '}(2\eta_n - \eta_{n-1}) \eta_{n,l}}{\vert \xi \vert^2 - \vert \eta_n \vert ^2 - \vert \eta_{n-1} - \eta_n \vert ^2} e^{it (\vert \xi \vert^2 - \vert \eta_n \vert ^2 - \vert \eta_{n-1} - \eta_n \vert ^2)} \\
& \times  \widehat{f_{k_n}}(t,\eta_n) \widehat{f_{k_{n+1}}}(t,\eta_{n-1}-\eta_n) d\eta_n,
\end{align*}
as well as Lemma \ref{bilin} and \ref{dispersive} to obtain the following bound:
\begin{align*}
\Vert I_9 ^n f \Vert_{L^{\infty}_t L^2_x } & \lesssim 1.1^m 1.1^{-k_n} 1.1^{-m} 1.1^{-m/2} \delta^n C^n G(\textbf{k}) \varepsilon_1 ^2
\\
                                          & \lesssim 1.1^m 1.1^{101/224 m} 1.1^{-3m/2} \delta^n C^n G(\textbf{k}) \varepsilon_1 ^2.
\end{align*}
This expression can be summed. \\
After the integration by parts in time we also obtain the following main terms: (here for better legibility we only write the last part of the integral)
\begin{align*}
& \int_{1.1^m}^{1.1^{m+1}} \int_{\eta_n} \frac{i s \eta_{n,l}P_{k_n'}(2\eta_n-\eta_{n-1})}{\vert \xi \vert^2 - \vert \eta_n \vert^2 - \vert \eta_{n-1}-\eta_n \vert^2} \\ 
& \times e^{is(\vert \eta_{n-1} \vert^2 - \vert \eta_n \vert^2 - \vert \eta_{n-1}- \eta_n \vert^2)} \partial_s \widehat{f_{k_n}} (s,\eta_n) \widehat{f_{k_{n+1}}}(s,\eta_{n-1}-\eta_n) d\eta_n ds  \\ 
&= \int_{1.1^m}^{1.1^{m+1}} \int_{\eta_n} \frac{i s \eta_{n,l}P_{k_n'}(2\eta_n-\eta_{n-1})}{\vert \xi \vert^2 - \vert \eta_n \vert^2 - \vert \eta_{n-1}-\eta_n \vert^2} \\
& \times e^{is(\vert \eta_{n-1} \vert^2 - \vert \eta_n \vert^2 - \vert \eta_{n-1}- \eta_n \vert^2)} \widehat{f_{k_{n+1}}}(s,\eta_{n-1}-\eta_n) P_{k_n}(\eta_n) \int_{\eta_{n+1} \in \mathbb{R}^3} \widehat{V}(s,\eta_n - \eta_{n+1}) \widehat{u}(s,\eta_{n+1}) d\eta_{n+1} d\eta_n ds \\
&+ \int_{1.1^m}^{1.1^{m+1}} \int_{\eta_n} \frac{i s \eta_{n,l}P_{k_n'}(2\eta_n-\eta_{n-1})}{\vert \xi \vert^2 - \vert \eta_n \vert^2 - \vert \eta_{n-1}-\eta_n \vert^2} \\
& \times e^{is(\vert \eta_{n-1} \vert^2 - \vert \eta_n \vert^2 - \vert \eta_{n-1}- \eta_n \vert^2)} \widehat{f_{k_{n+1}}}(s,\eta_{n-1}-\eta_n) P_{k_n}(\eta_n) \int_{\mathbb{R}^3} \widehat{u}(\eta_{n} - \eta_{n+1}) \widehat{u}(s,\eta_{n+1}) d\eta_{n+1} d\eta_n ds \\
& + \int_{1.1^m}^{1.1^{m+1}} \int_{\eta_n} \frac{i s \eta_{n,l}P_{k_n'}(2\eta_n-\eta_{n-1})}{\vert \xi \vert^2 - \vert \eta_n \vert^2 - \vert \eta_{n-1}-\eta_n \vert^2} e^{is(\vert \eta_{n-1} \vert^2 - \vert \eta_n \vert^2 - \vert \eta_{n-1}- \eta_n \vert^2)} \widehat{f_{k_{n+1}}}(s,\eta_{n-1}-\eta_n) \\
& \times P_{k_n}(\eta_n) \int_{\eta_{n+1} \in \mathbb{R}^3}  \eta_{n+1,i} \widehat{a_i}(s,\eta_n - \eta_{n+1}) \widehat{u}(s,\eta_{n+1}) d\eta_{n+1} d\eta_n ds \\
&:= I+II+III.
\end{align*}
The terms $I$ and $II$ are already present in \cite{L} and they can be dealt with following the exact same strategy here. Therefore we omit the details for these two terms and focus on $III$ which, although it is very close to $I$ in \cite{L}, was not present. \\
\\
Using the observation above \eqref{symbol1} we write that, using as usual our multilinear lemmas,
\begin{align} \label{dernier}
\Vert III \Vert_{L^2_x} & \lesssim G(\textbf{k}) \delta^n C^n \Bigg \Vert \mathcal{F}^{-1} \int_{\eta_n} P_{k_n} (\eta_n) \frac{it \eta_{n,l}P_{k_n'}(2\eta_n - \eta_{n-1})}{\vert \xi \vert^2 - \vert \eta_n \vert^2 - \vert \eta_{n-1}-\eta_n \vert^2} \\
\notag & \times \mathcal{F}(a_i \partial_{x_i} u)(t,\eta_n) \widehat{u_{k_{n+1}}}(t,\eta_{n-1}-\eta_n) d\eta_n \Bigg \Vert_{L^{4/3}_t L^{3/2}_x} \\
\notag & \lesssim 1.1^{-k_n} G(\textbf{k}) \delta^n C^n \Vert t (a_i \partial_{x_i} u)_{k_n} \Vert_{L^{\infty}_t L^2 _x} \Vert u_{k_{n+1}} \Vert_{L^{4/3}_t L^{6}_x} \\
\notag & \lesssim G(\textbf{k}) \delta^n C^n \Vert t (a_i \partial_{x_i} u)_{k_n} \Vert_{L^{\infty}_t L^{6/5} _x} \Vert u_{k_{n+1}} \Vert_{L^{4/3}_t L^{6}_x}.
\end{align}
Now we look at the term $a_i \partial_{x_i} u$ and decompose the frequency variable dyadically (we denote $k_{n+2}$ the corresponding exponent). This reads
\begin{align*}
a_i \partial_{x_i} u=(2\pi)^{-3} \sum_{k_{n+2} \in \mathbb{Z}} \mathcal{F}^{-1} \int_{\eta_{n+2}} \widehat{a_i}(\eta_n-\eta_{n+2}) \eta_{n+2,i} \widehat{u_{k_{n+2}}}(\eta_{n+2}) d\eta_{n+2} .
\end{align*}
\underline{Case 1: $\vert k_{n+2} - k_n \vert \leqslant 1 $} \\
There are $O(m)$ terms in that sum on $k_{n+2}.$ Then, using dispersive estimates, the bound yields
\begin{align*}
\Vert III \Vert_{L^2 _x} & \lesssim \sum_{k_{n+2}} 1.1^{k_{n+2}-k_n} G(\textbf{k}) \delta^n C^n \Vert t(a_i u) \Vert_{L^2_x} \Vert u_{k_{n+1}} \Vert_{L^{4/3}_t L^6 _x} \\
& \lesssim \sum_{k_{n+2}} \delta^n C^n G(\textbf{k}) \Vert a_i \Vert_{L^3_x} \Vert t u \Vert_{L^{\infty}_t L^6 _x} \Vert u_{k_{n+1}} \Vert_{L^{4/3}_t L^6 _x} \\
& \lesssim \sum_{k_{n+2}} 1.1^{-m/4+} \delta^n C^n G(\textbf{k}) \delta \varepsilon^2,
\end{align*}
and we are done in this case. \\
\underline{Case 2: $k_{n+2} > k_n +1$} \\
Then $a_i$ is localized at frequency roughly $1.1^{k_{n+2}}$ and we can write that 
\begin{align*}
\Vert III \Vert_{L^2_x} & \lesssim \sum_{k_{n+2}} 1.1^{k_{n+2}-k_n} \Vert t a_{i,k_{n+2}} u \Vert_{L^{2}_x} \Vert u_{k_{n+1}} \Vert_{L^{4/3}_t L^6 _x} \delta^n C^n G(\textbf{k}) \\
& \lesssim \sum_{k_{n+2}} 1.1^{k_{n+2}-k_n} \Vert a_{i,k_{n+2}} \Vert_{L^{3}_x} \Vert t u \Vert_{L^{\infty}_t L^6 _x} \Vert u_{k_{n+1}} \Vert_{L^{4/3}_t L^6 _x} \delta^n C^n G(\textbf{k}) \\
& \lesssim \sum_{k_{n+2}} 1.1^{k_{n+2}-k_n} \Vert a_{i,k_{n+2}} \Vert_{L^{3}_x} 1.1^{-m/4+} C^n \delta^n G(\textbf{k}) \varepsilon_1 ^2,
\end{align*}
and we are done in this case as well since we can sum over $k_{n+2}.$ \\ 
\underline{Case 3: $k_{n+2} < k_n-1$} \\
We write that 
\begin{align*}
\Vert III \Vert_{L^2_x} & \lesssim \sum_{k_{n+2}} \delta^n C^n G(\textbf{k}) 1.1^{k_{n+2}-k_n} \Vert t (a_i u) \Vert_{L^2_x} \Vert u_{k_{n+1}} \Vert_{L^{4/3}_t L^6 _x}  \\
& \lesssim \sum_{k_{n+2}} 1.1^{k_{n+2}-k_n} \delta^n C^n G(\textbf{k}) \Vert a_{i} \Vert_{L^3_x} \Vert t u \Vert_{L^{\infty}_t L^6 _x} \Vert u_{k_{n+1}} \Vert_{L^{4/3}_t L^6 _x} ,
\end{align*}
and we can conclude using Lemma \ref{summation} and the fact that $k_{n+2}<k_n-1$ to sum this bound. 
\end{proof}

%%%%%%%%%%%%%%%%%%%%%%%%%%%%%%%%%%%%%%%%%%%%%%%%%%%%%%%%%%%%%%%%%%%%%%%%%%%%%%%%%%%%%%%%%%%%%%%%%%%%%%%%%%%%%%%%%%%%%%%%%%%%%%%%%%%%%%%%%%%%%%%%%%%%%%%%%%%%%%%%%%%%%%%%%%%%%%%%%%%%%%%%%%%%%%%%%%%%%%%%%%%%%%%%%%%%%%%%%%%%%%%%%%%%%%%%%%%%%%%%%%%%%%%%%%%%%%%%%%%%%%%%%%%%%%%%%%%%%%%%%%%%%%%%%%%%%%%%%%%%%%%%%%%%%%%%%%%%%%%%%%%%%%%%%%%%%%%%%%%%%%%%%%%%%%%%%%%%%%%%%%%%%%%%%%%%%%%%%%%%%%%%%%%%%%%%%%%%%%%%%%%%%%%%%%%%%%%%%%%%%%%%%%%%%%%%%%%%%%%%%%%%

\section{Energy estimate} \label{energy}
Here we prove the $H^{10}$ estimate on the solution. The method is, as in the proof of \eqref{goal1}, to
expand the solution as a series. This case however is simpler, in the sense that only integrations by parts in time are required. In other words the series is genuinely obtained by repeated applications of the Duhamel formula. The terms of the series are then estimated using lemmas from Section \ref{multilinkey}.
\\
\\
First recall that the bilinear part of the Duhamel formula has already been estimated in \cite{L}, Lemma 8.1:
\begin{lemma}\label{H10bilin}
We have the bound
\begin{align*}
\Bigg \Vert P_k (\xi) \int_1 ^t \int_{\mathbb{R}^3} e^{is(\vert \xi \vert^2 - \vert \eta_1 \vert^2 - \vert \xi - \eta_1 \vert^2)} \widehat{f}(s,\eta_1) \widehat{f}(s,\xi-\eta_1) d\eta_1 ds \Bigg \Vert_{H^{10}_x} \lesssim \varepsilon_1 ^2.
\end{align*}
\end{lemma}
Therefore we must estimate the $H^{10}$ norms of the potential parts. \\
Now we expand the solution as a series by repeated integrations by parts in time for the potential parts (with suitable regularizations when the phase is close to 0). \\
At the $n-$th step of this process we obtain the following terms :
\begin{align*}
&\mathcal{F} J_1 ^n f :=  \int \prod_{\gamma=1}^{n-1} \frac{\alpha_{\gamma}(\eta_{\gamma})\widehat{W_{\gamma}}(\eta_{\gamma-1}-\eta_{\gamma})P_{k_{\gamma}}(\eta_{\gamma}) P_k(\xi)}{\vert \xi \vert^2 - \vert \eta_{\gamma} \vert^2} d\eta_1 ... d\eta_{n-1} \widehat{f}(s,\eta_n) d\eta_n 
\end{align*}
which is a boundary term when doing the integration by parts. \\
There are also the terms corresponding to the main terms 
\begin{align*}
&\mathcal{F} J_2 ^n f :=  \int \prod_{\gamma=1}^{n-1} \frac{\alpha_{\gamma}(\eta_{\gamma})\widehat{W_{\gamma}}(\eta_{\gamma-1}-\eta_{\gamma})P_{k_{\gamma}}(\eta_{\gamma}) P_k(\xi)}{\vert \xi \vert^2 - \vert \eta_{\gamma} \vert^2} d\eta_1 ... d\eta_{n-2} \\
& \times \int_1 ^t \int_{\eta_n} e^{is(\vert \xi \vert^2 - \vert \eta_n \vert^2 - \vert \eta_{n-1}-\eta_n \vert^2)} \widehat{f}(\eta_{n-1}-\eta_n)  \widehat{f}(s,\eta_n) d\eta_n ds d\eta_{n-1}
\end{align*}
and
\begin{align} \label{lastterm}
&\mathcal{F} J_3 ^n f :=  \int \prod_{\gamma=1}^{n-1} \frac{\alpha_{\gamma}(\eta_{\gamma})\widehat{W_{\gamma}}(\eta_{\gamma-1}-\eta_{\gamma})P_{k_{\gamma}}(\eta_{\gamma}) P_k(\xi)}{\vert \xi \vert^2 - \vert \eta_{\gamma} \vert^2} d\eta_1 ... d\eta_{n-2} \\
\notag & \times \int_1 ^t \int_{\eta_n} e^{is(\vert \xi \vert^2 - \vert \eta_n \vert^2)} \widehat{W_n}(\eta_{n-1}-\eta_n) \alpha_{n}(\eta_n) \widehat{f}(s,\eta_n) d\eta_n ds d\eta_{n-1}.
\end{align}
To obtain the next terms in the expansion, we integrate by parts in time in that last term. Therefore to show that the series converges in $H^{10}$ and to estimate its size, it is enough to estimate only the first two iterates. 
\\
The following proposition gives a bound on the $H^{10}$ norm of $J_1 ^n:$
\begin{lemma} \label{en-1}
We have:
\begin{align*}
\Vert J_1 ^n f \Vert_{H^{10}_x} \lesssim C^n G(\textbf{k}) \delta^n \varepsilon_1.
\end{align*}
\end{lemma}
\begin{proof}
The proof is almost the same as in \cite{L} Lemma 8.2, therefore it is omitted. 
\end{proof}
Finally we estimate the $H^{10}$ norm of $J_2 ^n f.$ 
\begin{lemma}
We have the following bound on the $H^{10}$ norm of the solution:
\begin{align*}
\Vert J_2 ^n f \Vert_{H^{10}_x} \lesssim C^n G(\textbf{k}) \delta^n \varepsilon_1.
\end{align*}
\end{lemma}
\begin{remark} \label{rate}
In the case where the time integral $\int_1 ^t $ is replaced by $\int_{\tau} ^t $ for some $\tau < t,$ a similar estimate holds (the only modification is that the right-hand side has an additional $\tau^{-a}$ for some $a>0$). 
\end{remark}
\begin{proof}
Since the corresponding proof was omitted in \cite{L}, we carry it out here for completness. It is almost identical to that of Lemma 8.2 in \cite{L}.\\ 
We start by decomposing the $\xi$ and $\eta_n$ variables dyadically. We denote $k_n$ and $k$ the corresponding exponents. \\
\underline{Case 1: $k_{n} \geqslant k-1$} \\
Using Lemma \ref{multiin} with $(p,q) = (4,3)$ and
\begin{align*}
\widehat{g_{k_{n-1}}}(s,\eta_{n-1}) = \int_{\eta_n} e^{-is\vert \eta_n \vert^2} \widehat{f_{k_n}}(s,\eta_n) \widehat{f}(s,\eta_{n-1}-\eta_n) d\eta_n,
\end{align*}
we obtain by Lemma \ref{summation} and the energy bound 
\begin{align*}
1.1^{10 k^{+}} \Vert P_k (\xi) J_2 ^n f \Vert_{L^2 _x} & \lesssim 1.1^{10 k^{+}} C^n G(\textbf{k}) \delta^n \Vert f_{k_n} \Vert_{L^{\infty}_t L^2_{x}} \Vert e^{it\Delta} f \Vert_{L^{4/3}_t L^6 _x} \\
& \lesssim 1.1^{10(k^{+}-k_n^{+})} G(\textbf{k}) \delta^n C^n \varepsilon_1 ^2.
\end{align*}
\underline{Case 2: $k_n < k-1$} \\
\underline{Subcase 2.1: $\forall j \in \lbrace 1; ...; n \rbrace, k_j<k-1 $} \\
Then the first potential in the product $\widehat{W_1}(\xi-\eta_1)$ is localized at frequency $1.1^k.$ \\
Therefore we can carry out the same proof as above and obtain
\begin{align*}
1.1^{10 k^{+}} \Vert P_k (\xi) J_2 ^n f \Vert_{L^2 _x} & \lesssim C^n G(\textbf{k}) \delta^n 1.1^{10k^{+}} \Vert W_{1,k} \Vert_{Y} \Vert f \Vert_{L^{\infty}_t L^2_{x}} \Vert e^{it\Delta} f_{k_n} \Vert_{L^{4/3}_t L^6 _x},
\end{align*}
which can we summed over $k$, adding an additional $\delta$ factor to the product.
\\
\underline{Subcase 2.2: $\exists j \in \lbrace 1;...; n \rbrace, k_j \geqslant k-1.$} 
\\
Let $j'=\max \lbrace j ; k_j \geqslant k-1 \rbrace.$ \\
If $k_{j'}>k+1,$ then $W_{j'+1}(\eta_{j'}-\eta_{j'+1})$ (with the convention that $W_{n}=f$ Indeed in this case the second $f$ factor would be localized at $1.1^{k_{n-1}}$ since $1.1^{k_{n-1}}>k+1>1.1^{k_n}+2$) is localized at frequency $1.1^{k_{j'}}.$ We can then conclude as in the above case by effectively absorbing the $1.1^{10 k^{+}}$ factor at the price of replacing $W_{j'+1}$ by $\nabla^{10} W_{j'+1}.$ \\
If $\vert k_{j'}-k \vert \leqslant 1,$ then if there exists some $j'' \in \lbrace j'+1,...,n \rbrace $ such that $\vert k_{j''} - k_{j''-1} \vert > 1,$ the factor $\widehat{W_{j''}}$ is localized at frequency $1.1^{k_j ''}.$ But since there are $n$ terms in the product, $k_{j''} \geqslant k-n-1.$ Therefore by the same proof as above:
\begin{align*}
1.1^{10 k^{+}} \Vert P_k (\xi) J_2 ^n f \Vert_{L^2 _x} & \lesssim C^n 1.1^{10 n} G(\textbf{k}) \delta^n 1.1^{10(k^{+}-n+k_{j''}} \Vert \nabla^{10} W_{j'',k_{j''}} \Vert_{Y} \Vert f \Vert_{L^{\infty}_t L^2_{x}} \Vert e^{it\Delta} f_{k_n} \Vert_{L^{4/3}_t L^6 _x}
\end{align*}
and we get the desired result. 
\\
Finally if $\forall j'' \geqslant j'+1,\vert k_{j''} - k_{j''-1} \vert \leqslant 1, $ then $k_n \geqslant k-n.$ Then we can conclude by writing that
\begin{align*}
1.1^{10 k^{+}} \Vert P_k (\xi) J_2 ^n f \Vert_{L^2 _x} & \lesssim C^n G(\textbf{k}) 1.1^{10 n} \delta^n 1.1^{10(k^{+}-n+k_{n})} 1.1^{10 k_n ^{+}} \Vert f_{k_n} \Vert_{L^{\infty}_t L^2_{x}} \Vert e^{it\Delta} f \Vert_{L^{4/3}_t L^6 _x}.
\end{align*}
\end{proof}

\begin{proof}[Proof of \eqref{goal2}]
We conclude with the proof of \eqref{goal2}. Note that since at each step of the iteration $O(4^n)$ terms appear, given the estimates proved in this section we can write that there exists some large constant $D$ such that
\begin{align*}
\Vert f \Vert_{H^{10}_x} \leqslant \varepsilon_0 + D \sum_{n=0}^{+\infty} \delta^n 4^n C^n \varepsilon_1 \leqslant \frac{\varepsilon_1}{2}
\end{align*}
for $\delta$ small enough.

\end{proof}

\appendix

\section{Basic estimates}
In this appendix we prove a few basic estimates based on dispersive properties of the free Schr\"{o}dinger flow.
\begin{lemma}\label{dispersive}  
We have
\begin{equation*}
\Vert e^{it \Delta} f_k \Vert_{L^6_x} \lesssim \frac{\varepsilon_1}{t}. 
\end{equation*}
\end{lemma}
\begin{proof}
See for example \cite{L}, Lemma 3.5.
\end{proof}

For the summations we will need the following lemma
\begin{lemma} \label{summation}
For any $p, q \geqslant 1$ and $1>c >0$
\begin{align*}
\Vert e^{it \Delta} f_{k} \Vert_{L^p _t L^q _x}  \lesssim  \Vert e^{it \Delta} f_{k} \Vert_{L^{p(1-c)} _t L^q _x}^{1-c} 1.1^{-8c k} \varepsilon_1^{c}   
\end{align*}
and 
\begin{align*}
\Vert e^{it \Delta} f_{k} \Vert_{L^p _t L^q _x}  \lesssim 1.1^{\frac{3c}{q(1-c)}k} \Vert e^{it \Delta} f_{k} \Vert_{L^{p} _t L^{q(1-c)} _x}.
\end{align*}
\end{lemma}
\begin{proof}
We write using Sobolev embedding and the energy bound
\begin{align*}
\Vert e^{it \Delta} f_{k} \Vert_{L^q _x} & \lesssim \Vert e^{it \Delta} f_{k} \Vert_{L^q _x} ^{1-c} \Vert e^{it \Delta} f_{k} \Vert_{H^{2}_x}^{c} \\
& \lesssim \Vert e^{it \Delta} f_{k} \Vert_{L^q _x} ^{1-c} 1.1^{-8kc} \varepsilon_1 ^c
\end{align*}
and then we take the $L^p _t$ norm of that expression and obtain
\begin{align*}
\Vert e^{it \Delta} f_{k} \Vert_{L^p _t L^q _x} \lesssim \Vert e^{it \Delta} f_{k} \Vert_{L^{p(1-c)} _x L^q _x}^{1-c} 1.1^{-c k} \varepsilon_1^{c}.
\end{align*}
Similarly using Bernstein's inequality 
\begin{align*}
\Vert e^{it \Delta} f_{k} \Vert_{L^q _x} \lesssim 1.1^{\frac{3c}{q(1-c)}k} \Vert e^{it \Delta} f_{k} \Vert_{L^{q(1-c)} _x},
\end{align*}
and we take the $L^p _t$ norm of that expression and obtain the result.
\end{proof}

\section{Technical Lemmas}
In this section we collect a number of basic technical lemmas that are used in the paper.
\\
The following lemma decomposes the frequency space according to a dominant direction:
\begin{lemma} \label{direction}
There exist three functions $\chi_j : \mathbb{R}^3  \longrightarrow \mathbb{R}$ such that
\begin{itemize}
\item $1 = \chi_1 + \chi_2 + \chi_3.$ 
\item On support of $\chi_j,$ we have $\vert \xi_j \vert \geqslant \lbrace \frac{9}{10} \vert \xi_k \vert; k=1,2,3  \rbrace.$
\end{itemize}
\end{lemma}
\begin{proof}
Appendix A, \cite{PuM}
\end{proof}
We record a basic bilinear estimate:
\begin{lemma}\label{bilin}
The following inequality holds
\begin{equation}
\Big \Vert \mathcal{F}^{-1} \int_{\mathbb{R}^3} m(\xi,\eta) \widehat{f}(\xi-\eta) \widehat{g}(\eta) d\eta \Big \Vert_{L^r} \lesssim \Vert \mathcal{F}^{-1}(m(\xi-\eta,\eta)) \Vert_{L^1} \Vert f \Vert_{L^p} \Vert g \Vert_{L^q} 
\end{equation}
where $1/r = 1/p+1/q.$
\end{lemma}
\begin{proof}
See \cite{L}, Lemma 3.1 for example.
\end{proof}
The following bound on the norm $X'$ is convenient as it appears naturally in the estimates.  
\begin{lemma} \label{X'}
Define the $X'-$norm as \begin{align*}
\Vert f \Vert_{X'} = \sup_{k \in \mathbb{Z}} \Vert \big( \nabla_{\xi} \widehat{f} \big) P_k (\xi) \Vert_{L^2} .
\end{align*}
Then 
\begin{align*}
\Vert f \Vert_{X'} \lesssim \Vert f \Vert_X.
\end{align*}
\end{lemma}
\begin{proof}
See \cite{L}, Lemma 3.8
\end{proof}

\section{Scattering} \label{scattering}
In this section we prove the scattering statement in Theorem \ref{maintheorem}. This is essentially a consequence of estimates proved earlier in the paper.
\\
We start with the expansion from Section \ref{energy}:
\begin{align*}
\widehat{f}(t) &= e^{i \vert \xi \vert^2} \widehat{u_1}(\xi) - i \int_1 ^t e^{is \vert \xi \vert^2} \int_{\mathbb{R}^3} e^{-is \vert \xi-\eta_1 \vert^2 - \vert \eta_1 \vert^2} \widehat{f}(s,\eta_1) \widehat{f}(s,\xi-\eta_1) d\eta_1  \\
&+ \sum_{k=0}^{+\infty} \sum_{n=2}^{+\infty} \sum_{k_1,...,k_{n-1} \in \mathbb{Z}} \sum_{\gamma=1}^{n-1} \sum_{W_{\gamma} \in \lbrace a_1,a_2,a_3,V \rbrace} \frac{i^{n+1}}{(2\pi)^{3(n-1)}} \mathcal{F} J^{n}_1 \\
&+\sum_{k=0}^{+\infty} \sum_{n=2}^{+\infty} \sum_{k_1,...,k_{n-1} \in \mathbb{Z}} \sum_{\gamma=1}^{n-1} \sum_{W_{\gamma} \in \lbrace a_1,a_2,a_3,V \rbrace} \frac{i^{n-1}}{(2\pi)^{3(n+1)}} \mathcal{F}  J^{n}_2 . 
\end{align*}
We define the operator $W:H^{10}_x \rightarrow H^{10}_x$ as
\begin{align*}
W u(t) = u(t) - e^{it\Delta}   \sum_{k=0}^{+\infty} \sum_{n=2}^{+\infty} \sum_{k_1,...,k_{n-1} \in \mathbb{Z}} \sum_{\gamma=1}^{n-1} \sum_{W_{\gamma} \in \lbrace a_1,a_2,a_3,V \rbrace} \frac{i^{n+1}}{(2\pi)^{3(n-1)}} J^{n}_1. 
\end{align*}
The boundedness of $W$ for small enough $\delta>0$ is a consequence of Lemma \ref{en-1}. \\
Let $1<\tau<t.$ With the above expansion, we obtain the estimate (using the remark \ref{rate})
\begin{align*}
\big \Vert Wu(t) - Wu(\tau) \big \Vert_{H^{10}_x} & \leqslant \bigg \Vert \mathcal{F}^{-1}_{\xi} \int_{\tau} ^t e^{is \vert \xi \vert^2} \int_{\mathbb{R}^3} e^{-is \vert \xi-\eta_1 \vert^2 - \vert \eta_1 \vert^2} \widehat{f}(s,\eta_1) \widehat{f}(s,\xi-\eta_1) d\eta_1  \bigg \Vert_{H^{10}_x}  \\
&+ \bigg \Vert \sum_{k=0}^{+\infty} \sum_{n=2}^{+\infty} \sum_{k_1,...,k_{n-1} \in \mathbb{Z}} \sum_{\gamma=1}^{n-1} \sum_{W_{\gamma} \in \lbrace a_1,a_2,a_3,V \rbrace} \frac{i^{n-1}}{(2\pi)^{3(n+1)}} \big(J^{n}_2(t)-J^{n}_2(\tau) \big) \bigg \Vert_{H^{10}_x} \\
& \lesssim \tau^{-a} \varepsilon_1 ^2 + \tau^{-a} \sum_{n=2}^{+\infty} 4^n C^{n} \delta^{n-1} \varepsilon_1 ^2 .
\end{align*}
This shows that if $\delta>0$ is small enough, $\big(e^{-it\Delta} W u(t)\big)$ is Cauchy in $H^{10}_x$ and the scattering statement follows. In fact a closer inspection of the proof of Lemma \ref{en-1} would yield a quantitative polynomial decay rate for the above convergence.
%%%%%%%%%%%%%%%%%%%%%%%%%%%%%%%%%%%%%%%%%%%%%%%%%%%%%%%%%%%%%%%%%%%%%%%%%%%%%%%%%%%%
%%%%%%%%%%%%%%%%%%%%%%%%%%%%%%%%%%%%%%%%%%%%%%%%%%%%%%%%%%%%%%%%%%%%%%%%%%%%%%%%%%%%
%%%%%%%%%%%%%%%%%%%%%%%%%%%%%%%%%%%%%%%%%%%%%%%%%%%%%%%%%%%%%%%%%%%%%%%%%%%%%%%%%%%%
%%%%%%%%%%%%%%%%%%%%%%%%%%%%%%%%%%%%%%%%%%%%%%%%%%%%%%%%%%%%%%%%%%%%%%%%%%%%%%%%%%%%

\end{document}